\documentclass[11pt]{amsart}

\pdfoutput=1

\usepackage[text={420pt,660pt},centering]{geometry}

\usepackage{cancel}
\usepackage{esint,amssymb} 
\usepackage{graphicx}

\usepackage{MnSymbol}
\usepackage{mathtools} 
\usepackage[colorlinks=true, pdfstartview=FitV, linkcolor=blue, citecolor=blue, urlcolor=blue,pagebackref=false]{hyperref}
\usepackage{microtype}
\usepackage{bm}
\usepackage{mathrsfs}
\usepackage{xcolor}
\usepackage{enumitem}
\setlist{leftmargin=0.5\leftmargin}

\newtheorem{proposition}{Proposition}
\newtheorem{theorem}[proposition]{Theorem}
\newtheorem{lemma}[proposition]{Lemma}
\newtheorem{corollary}[proposition]{Corollary}

\theoremstyle{remark}
\newtheorem{remark}[proposition]{Remark}

\theoremstyle{definition}
\newtheorem{definition}[proposition]{Definition}

\numberwithin{equation}{section}
\numberwithin{proposition}{section}
\numberwithin{figure}{section}
\numberwithin{table}{section}

\newcommand{\N}{\mathbb{N}}
\newcommand{\Q}{\mathbb{Q}}
\newcommand{\R}{\mathbb{R}}

\newcommand{\E}{\mathbb{E}}
\renewcommand{\P}{\mathbb{P}}

\newcommand{\ep}{\varepsilon}
\newcommand{\eps}{\varepsilon}

\renewcommand{\le}{\leqslant}
\renewcommand{\ge}{\geqslant}
\renewcommand{\leq}{\leqslant}
\renewcommand{\geq}{\geqslant}

\renewcommand{\subset}{\subseteq}
\renewcommand{\bar}{\overline}
\renewcommand{\tilde}{\widetilde}

\renewcommand{\hat}{\widehat}

\newcommand{\Ll}{\left}
\newcommand{\Rr}{\right}
\renewcommand{\d}{\mathrm{d}}
\newcommand{\dr}{\partial}

\newcommand{\1}{\mathbf{1}}

\newcommand{\mcl}{\mathcal}
\newcommand{\msf}{\mathsf}
\newcommand{\mfk}{\mathfrak}
\newcommand{\bs}{\boldsymbol}
\newcommand{\al}{\alpha}

\newcommand{\de}{\delta}
\newcommand{\si}{\sigma}

\DeclareMathOperator{\tr}{tr}
\DeclareMathOperator{\supp}{supp}

\newcommand{\la}{\left\langle}
\newcommand{\ra}{\right\rangle}

\renewcommand{\S}{S}

\newcommand{\bff}{\mathbf{f}}

\newcommand{\cH}{{L^2}}
\newcommand{\C}{\mathcal{Q}}
\newcommand{\D}{{D}}

\newcommand{\id}{\mathrm{Id}}

\newcommand{\sP}{\mathscr{P}}

\newcommand{\bal}{\Lambda}
\newcommand{\bh}{\boldsymbol{h}}
\newcommand{\bg}{\boldsymbol{g}}
\newcommand{\fR}{\mathfrak{R}}
\newcommand{\bomega}{\varpi }
\newcommand{\bom}{\varpi }

\newcommand{\rhs}{\mathrm{RHS}}
\newcommand{\cF}{\mathcal{F}}
\newcommand{\lambdamin}{\lambda_{\mathrm{min}}}
\newcommand{\ellipt}{\mathsf{Ellipt}}
\newcommand{\cav}{{\circ}} \newcommand{\orig}{{}}

\begin{document}

\author{Hong-Bin Chen}
\address[Hong-Bin Chen]{Institut des Hautes \'Etudes Scientifiques, France}
\email{hbchen@ihes.fr}

\author{Jean-Christophe Mourrat}
\address[Jean-Christophe Mourrat]{Department of Mathematics, ENS Lyon and CNRS, Lyon, France}
\email{jean-christophe.mourrat@ens-lyon.fr}

\keywords{}
\subjclass[2010]{}
\date{\today}

\title{On the free energy of vector spin glasses with non-convex interactions}

\begin{abstract}
The limit free energy of spin-glass models with convex interactions can be represented as a variational problem involving an explicit functional. Models with non-convex interactions are much less well-understood, and simple variational formulas involving the same functional are known to be invalid in general. We show here that a slightly weaker property of the limit free energy does extend to non-convex models. Indeed, under the assumption that the limit free energy exists, we show that this limit can always be represented as a critical value of the said functional. Up to a small perturbation of the parameters defining the model, we also show that any subsequential limit of the law of the overlap matrix is a critical point of this functional. We believe that these results capture the fundamental conclusions of the non-rigorous replica method.
\end{abstract}

\maketitle

\section{Introduction}

\subsection{Informal summary of the main results}

We fix an integer $D \ge 1$, and let $(H_N(\sigma))_{\sigma \in (\R^N)^\D}$ be a centered Gaussian field such that, for every $\sigma = (\sigma_1,\ldots, \sigma_D)$ and $\tau = (\tau_1, \ldots, \tau_D) \in (\R^N)^\D$, we have
\begin{equation}
\label{e.def.xi}
\E \Ll[ H_N(\sigma) H_N(\tau) \Rr] = N \xi \Ll( \frac{\si \tau^\intercal}{N} \Rr) ,
\end{equation}
where $\xi \in C^\infty(\R^{D\times D}; \R)$ is a smooth function admitting an absolutely convergent power-series expansion with $\xi(0) = 0$, and where $\sigma \tau^\intercal$ denotes the matrix of scalar products 
\begin{equation}
\label{e.def.sigma.tau}
\sigma \tau^\intercal = (\sigma_d \cdot \tau_{d'})_{1 \le d,d' \le D}.
\end{equation}
The notation in \eqref{e.def.sigma.tau} is natural if we think of $\si$ and $\tau$ as $D$-by-$N$ matrices, and we often identify $(\R^N)^D$ with $\R^{D \times N}$. The class of functions $\xi$ such that there exists a Gaussian random field satisfying \eqref{e.def.xi}, together with an explicit construction of the Gaussian process $H_N$, are given in Subsection~\ref{ss.explicit.H} below. We are mostly interested in models for which the function $\xi$ is non-convex. A representative example in this class is the bipartite model, which is obtained by choosing  $D = 2$ and $\xi((A_{d,d'})_{1 \le d,d' \le 2}) = A_{1,1} A_{2,2}$. This model can be realized explicitly by setting $H_N(\sigma)$ to be 
\begin{equation*}  \frac 1 {\sqrt{N}} \sum_{i,j = 1}^N J_{i,j} \sigma_{1,i} \sigma_{2,j},
\end{equation*}
where $(J_{i,j})_{1 \le i,j \le N}$ are independent standard Gaussian random variables. 

We also fix a probability measure $P_1$ on $\R^D$ with compact support, and for each integer $N \ge 1$, we denote by $P_N = P_1^{\otimes N}$ the $N$-fold tensor product of $P_1$. We think of $P_N$ as a probability measure on $\R^{D \times N} \simeq (\R^N)^D$. 
We are interested in understanding, for each fixed $\beta \ge 0$, the large-$N$ limit of the free energy
\begin{equation*}
\frac 1 N \E \log \int \exp \Ll(\beta H_N(\si) \Rr) \, \d P_N(\si).
\end{equation*}
For $D = 1$, $\xi(r) = r^2$, and $P_1 = \tfrac 1 2 \de_{-1} + \tfrac 1 2  \de_{1}$, this corresponds to the Sherrington--Kirkpatrick model \cite{sherrington1975solvable}. In this context, a candidate for the limit free energy was derived in \cite{MPV, parisi1979infinite, parisi1980order, parisi1980sequence, parisi1983order} using sophisticated non-rigorous arguments within the framework of the replica method. In this approach, a number of formal manipulations are performed in order to be in a position to appeal to the Laplace principle for integrals of exponential functions, say of the form 
\begin{equation}  
\label{e.laplace}
\frac 1 N \log \int \exp(N \mathscr J(x)) \, \d x.
\end{equation}
The actual expression that shows up within the replica method also involves the number of replicas as a parameter. As this number, a positive integer, is sent to zero, unexpected things happen. In particular, the idea that the limit of~\eqref{e.laplace} ought to be the supremum of~$\mathscr J$ is no longer valid; for the SK model and all other scalar models with Ising spins, the limit becomes the infimum of $\mathscr J$ instead. For more general models, even this ``inverted variational principle'' is no longer necessarily valid. To the best of our understanding, in general, the main takeaway of the replica method is that the limit free energy can be represented as a critical value of this function $\mathscr J$; by this, we mean that the limit can be represented in the form of $\mathscr J(x)$ for some $x$ such that $\nabla \mathscr J(x) = 0$ (see for instance \cite{fyo1, fyo2, hartnett2018replica, korenblit1985spin} in the bipartite case). Crucially, one also expects that the critical point itself encodes the asymptotic law of the overlap between two independent replicas sampled according to the corresponding Gibbs measure.

The main results of this paper are rigorous versions of statements of this form, for general models and for a suitable choice of $\mathscr J$. Precisely, under the assumption that the free energy converges to some limit, we show that this limit must be a critical value of some explicit functional. Whether or not the free energy converges, we also show that, up to a small perturbation of the energy function $H_N$, any subsequential limit of the law of the overlap matrix is described by a critical point of the same functional.

A related route towards obtaining a fully unambiguous characterization of the limit free energy was explored in \cite{mourrat2021nonconvex, mourrat2023free}; see also \cite{HJbook}. The idea there is to consider an enriched version of the free energy that depends on more parameters, and aim to show that this enriched free energy is the solution to an explicit Hamilton--Jacobi equation. 
We will also work with this enriched version of the free energy throughout this paper. 

One possible approach to solving Hamilton--Jacobi equations is through the method of characteristics. In our context, the characteristics are straight lines, and formal calculations suggest that the gradient of the free energy ought to be constant along those lines. Using the equation then yields a prescription for the value of the solution itself along each of these lines. As long as the characteristic lines do not cross each other, one can show that defining a function according to this prescription yields a valid solution to the equation (see for instance \cite[Section~3.5]{HJbook} for more precision). In this perspective, our main result can be rephrased as saying that in fact, the limit free energy can \emph{always} be represented as the value prescribed by one of these characteristic lines. The only remaining source of ambiguity is that for large times, or in other words at low temperature, there may be multiple characteristic lines reaching a point, and we only assert that one of these characteristic lines prescribes the correct value. 

\subsection{Precise statement of the main results}
Let $\S^\D$ be the space of $D\times D$ real symmetric matrices, and $\S^\D_+$ be the subset of positive semi-definite matrices. 
Let 
\begin{equation}
\label{e.def.mclQ}
\mcl Q := \Ll\{ q : [0,1) \to S^\D_+ \ : \ q \text{ is right-continuous with left limits, and is increasing} \Rr\} ,
\end{equation}
where ``$q$ is increasing'' means that, for every $u,v \in \Ll[0,1\Rr)$,
\begin{equation*}  u \le v \quad \implies \quad q(u) \le q(v),
\end{equation*}
and the latter inequality is interpreted as $q(v) - q(u) \in S^D_+$. For every $r \in [1,\infty]$, we set $\mcl Q_r := \mcl Q \cap L^r([0,1]; S^D)$. The enriched free energy will be a function defined over $\R_+ \times \mcl Q_1$.
Here and throughout, we set $\R_+=[0,\infty)$.
When $q \in \mcl Q$ is constant, i.e. $q = h$ with $h \in S^\D_+$, this enriched free energy is defined as
\begin{multline}  
\label{e.def.simple.F}
\bar F_N(t,h) 
\\
:= -\frac 1 N \E \log \int \exp \Ll( \sqrt{2t} H_N(\sigma) - t N \xi\Ll(\frac{\sigma\sigma^\intercal}{N}\Rr) + \sqrt{2h} z \cdot \sigma - h \cdot \sigma \sigma^\intercal \Rr) \, \d P_N(\sigma),
\end{multline}
where $z = (z_1,\ldots, z_D)$ is a random element of $(\R^{N})^D \simeq \R^{D\times N}$ with independent standard Gaussian entries, independent of $(H_N(\sigma))_{\sigma \in (\R^N)^\D}$, and $\E$ denotes the expectation with respect to all randomness. In~\eqref{e.def.simple.F} and throughout the paper, whenever $a$ and $b$ are two matrices of the same size, we denote by $a \cdot b = \tr(a b^\intercal)$ the entrywise scalar product, and by $|a| = (a \cdot a)^{1/2}$ the associated norm. The compensating term $t N \xi\Ll(\frac{\sigma\sigma^\intercal}{N}\Rr)$ is equal to half the variance of the Gaussian random variable that immediately precedes it in the expression~\eqref{e.def.simple.F}; and as a consequence, an application of Jensen's inequality yields that~$\bar F_N$ is non-negative. This compensating term is helpful for the analysis, and we expect that it can be removed a posteriori using similar but simpler arguments as those presented here (in the sense that the final results are in terms of the critical points of a functional that has been enriched with a finite number of additional parameters). Here we will content ourselves with the observation that this term is constant for the bipartite model with $\pm 1$ spins, or more generally for any model with $\pm 1$ spins for which the function~$\xi$ only depends on the diagonal of its argument, so its presence causes no harm in these cases. (We also explain how this compensating term can be removed in the case when~$\xi$ is convex in Section~\ref{s.convex}.) The definition of $\bar F_N(t,q)$ for non-constant $q \in \mcl Q_1$ is more involved, as one needs to replace the ``replica-symmetric'' random magnetic field $\sqrt{2h} z$ appearing in~\eqref{e.def.simple.F} by a more complicated one that is coupled with a Poisson--Dirichlet cascade; we refer to~\eqref{e.def.FN.piece.const} and Proposition~\ref{p.barF_N_Lip} for a precise definition. Even with general $q \in \mcl Q_1$, the structure of this term remains that of a ``one-body'' potential, in the sense that it does not feature any interaction between spins indexed by different indices in $\{1,\ldots, N\}$. In particular, for every $N \ge 1$, we have that 
\begin{equation}  
\label{e.def.psi}
\psi(q) := \bar F_1(0,q) = \bar F_N(0,q).
\end{equation}
For every $q,q' \in \mcl Q_2$, $p \in L^2([0,1];S^\D)$, and $t \ge 0$, we define
\begin{equation}
\label{e.def.Parisi.functional}
\mcl J_{t,q}(q',p) := \psi(q') +  \la p , q-q'\ra_\cH + t \int_0^1 \xi(p).
\end{equation}
In this expression, we use the shorthand notation $\int_0^1 \xi(p) = \int_0^1 \xi(p(u)) \, \d u$, and the inner product between $p$ and $q-q'$ is that of the space $L^2([0,1], S^\D)$, that is, 
\begin{equation*}  \la p, q-q'\ra_\cH = \int_0^1 p(u) \cdot (q(u)-q'(u)) \, \d u.
\end{equation*}
The mapping in~\eqref{e.def.Parisi.functional} has a close relationship with the Hamilton--Jacobi equation
\begin{equation}  
\label{e.hj}
\Ll\{
\begin{array}{ll}
\displaystyle{\dr_t f - \int_0^1 \xi(\partial_q f)  = 0} & \quad \text{ on } \R_+ \times \mcl Q_2, \\
\displaystyle{f(0,\cdot) = \psi} & \quad \text{ on } \mcl Q_2.
\end{array}
\Rr.
\end{equation}
A first indication of this fact is that, for each fixed $q'$ and $p$, the mapping $(t,q) \mapsto \mcl J_{t,q}(q',p)$ is an affine solution to the equation~\eqref{e.hj} (with a different initial condition). More importantly, under convexity conditions on $\xi$ and $\psi$ respectively, the Hopf--Lax and Hopf formulas allow us to write the solution to~\eqref{e.hj} evaluated at $(t,q)$ as a saddle-point problem for the functional $\mcl J_{t,q}$, as was shown in \cite{chen2022hamilton} (see also \cite[Section~3.4]{HJbook} for the case of equations in finite dimension). For these reasons, we call the mapping $\mcl J_{t,q}$ the \emph{Hamilton--Jacobi functional}. While neither $\xi$ nor $\psi$ need to be convex nor concave in general, the conjecture of \cite{mourrat2021nonconvex, mourrat2023free} (see also \cite[Question~6.11]{HJbook}) is that $\bar F_N$ converges to the solution to~\eqref{e.hj}. The function $\mcl J_{t,q}$ will play the role of the function $\mathscr J$ discussed informally in the previous subsection. We say that the pair $(q',p) \in \mcl Q_2 \times L^2([0,1]; S^D)$ is a \emph{critical point} of the mapping $\mcl J_{t,q}$ if
\begin{equation}
\label{e.crit.p}
q = q' - t \nabla \xi(p)
\end{equation}
and
\begin{equation}
\label{e.crit.q'}
p = \partial_q \psi(q').
\end{equation}
In more explicit notation, these identities state that, for almost every $u \in [0,1)$, we have
\begin{equation*}  q(u) = q'(u) - t \nabla \xi(p(u)) \qquad \text { and } \qquad p(u) = \partial_q \psi(q',u),
\end{equation*}
and $\partial_q \psi(q,\cdot)$ denotes the Gateaux derivative of $\psi$ at $q$, see \eqref{d.gateaux} for a precise definition.
The condition \eqref{e.crit.p} is the first-variation relation for the variable $p$ in the mapping $\mcl J_{t,q}$, while the condition \eqref{e.crit.q'} is the first-variation relation for the variable $q'$ (ignoring the constraint $q' \in \mcl Q_2$). Notice that inserting \eqref{e.crit.q'} into \eqref{e.crit.p} yields that
\begin{equation}
\label{e.charact.line}
q = q' - t \nabla \xi(\dr_q \psi(q')).
\end{equation}
At least formally, the characteristic line associated with the equation \eqref{e.hj} and emanating from the point $q' \in \mcl Q_2$ is the trajectory 
\begin{equation*}  t' \mapsto \big(t', q'-t'\nabla \xi(\dr_q \psi(q'))\big)
\end{equation*}
(cf.\ \cite[(3.94)]{HJbook}).
The relation \eqref{e.charact.line} can therefore be rephrased as saying that the characteristic line emanating from $q'$ passes through the point $(t,q)$.

In order to clarify the relationship between the Hamilton--Jacobi functional and the Parisi formula, we first present a characterization of the limit free energy that is valid under the assumption that $\xi$ is convex over $S^D_+$. The limit free energy was first identified for a class of scalar models including the Sherrington--Kirkpatrick model in \cite{gue03, Tpaper}. The argument was then extended to general scalar models in \cite{pan.aom, pan, pan14}, using an alternative approach based on the ultrametricity property of the Gibbs measure. This approach was then extended to models as in \eqref{e.def.xi} in \cite{barcon, pan.multi, pan.potts, pan.vec}, under the assumption that the function $\xi$ is convex over $\R^{D\times D}$. Here we show that the more general assumption that $\xi$ is convex over $S^D_+$ is sufficient for this result to be valid. To generate examples of functions that satisfy the weaker condition but not the stronger one, we note that if $\xi_0 : \R^{D\times D} \to \R$ is convex over~$\R^{D\times D}$, then any function of the form 
\begin{equation*}  \xi : A = (A_{d,d'})_{d,d' \le 1} \mapsto  \xi_0(A) + \sum_{d = 1}^D \sum_{k = 1}^{+\infty} \beta_{d,k}^2 A_{d,d}^k,
\end{equation*}
with $(\beta_{d,k})_{d \le D, k \ge 1}$ decaying sufficiently rapidly to ensure the absolute summability of the series, is convex over $S^D_+$, but it is not necessarily convex over $\R^{D\times D}$.
\begin{theorem}
\label{t.parisi.convex}
If $\xi$ is a convex function over $S^\D_+$, then for every $t\ge 0$ and $q \in \mcl Q_2$, we have 
\begin{equation}  
\label{e.parisi.convex}
\lim_{N \to \infty} \bar F_N(t,q)= \sup_{q' \in q + \mcl Q_\infty} \inf_{p \in \mcl Q_\infty} \mcl J_{t,q}(q',p) = f(t,q) ,
\end{equation}
where $f$ is the solution to~\eqref{e.hj}. 
\end{theorem}
The fact that the Parisi formula ($q = 0$) can be recast in terms of the Hamilton--Jacobi functional in the case $D = 1$ is from \cite{mourrat2022parisi}, and was extended to arbitrary $q$ in \cite{mourrat2020extending}. The fact that the solution to  the Hamilton--Jacobi equation~\eqref{e.hj} with convex $\xi_{|S^D_+}$ admits a variational representation is from \cite{chen2022hamilton}, and 
the inequality 
\begin{equation*}  \liminf_{N \to \infty} \bar F_N(t,q) \ge f(t,q)
\end{equation*}
is shown in \cite{mourrat2021nonconvex, mourrat2023free} for arbitrary $\xi$. The converse bound will follow from the results obtained below, which are more precise than necessary in this case.

To clarify further the connection between Theorem~\ref{t.parisi.convex} and the classical expression of the Parisi formula, when $\xi$ is convex over $S^\D_+$, we define its convex dual $\xi^*$ by setting, for every $a \in \R^{D \times D}$, 
\begin{equation}
\label{e.def.xistar}
\xi^*(a) := \sup_{b \in S^\D_+} \Ll\{ a \cdot b - \xi(b)  \Rr\} .
\end{equation}
As will be shown here, one can rewrite the saddle-point problem in~\eqref{e.parisi.convex} as 
\begin{equation}  
\label{e.parisi.convex.var}
\sup_{q' \in \mcl Q_\infty} \Ll\{ \psi(q + q') - t \int_0^1 \xi^* \Ll( \frac{q'}{t} \Rr)  \Rr\} ,
\end{equation}
which perhaps most closely resembles the classical finite-dimensional Hopf--Lax formula adapted to the equation in~\eqref{e.hj}. 
For scalar models and $q = 0$, it was shown in~\cite{mourrat2022parisi} that one can recover the classical expression of the Parisi formula from the formula in~\eqref{e.parisi.convex} via a change of variables. 
This change of variables essentially consists in replacing the variable $q'$ with the variable $p$ defined so that $q' = t \nabla \xi(p)$, so that~\eqref{e.crit.p} holds; see also \cite[Section~6.5]{HJbook} for more precision. 

For possibly non-convex $\xi$, the first identity in~\eqref{e.parisi.convex} is false in general \cite{mourrat2021nonconvex}. Another plausible candidate saddle-point formula is
\begin{equation}  
\label{e.hopf.formula}
\sup_{p \in \mcl Q_\infty} \inf_{q' \in \mcl Q_\infty} \mcl J_{t,q}(q',p) ,
\end{equation}
which is the Hopf formula for the solution to~\eqref{e.hj} when $\psi$ is convex \cite{chen2022hamilton}. A formula of this form arises in the context of certain problems of statistical inference in high dimensions, see \cite{barbier2016, barbier2019adaptive, barbier2017layered, chen2022statistical, HB1, HBJ, kadmon2018statistical,   lelarge2019fundamental, lesieur2017statistical, luneau2020high, mayya2019mutualIEEE, miolane2017fundamental, mourrat2020hamilton, mourrat2021hamilton, reeves2020information, reeves2019geometryIEEE}, as well as \cite[Chapter~4]{HJbook} for a presentation following the Hamilton--Jacobi approach. However, it is shown in \cite[Section~6]{mourrat2021nonconvex} that the variational formula in~\eqref{e.hopf.formula} cannot be the limit free energy in general, because the function $\psi$ is not convex in general. A similar phenomenon seems to occur for certain problems of statistical inference on sparse graphs, see \cite{dominguez2022infinite, dominguez2022mutual, kireeva2023breakdown}. More broadly, there is no arrangement of inf's and sup's in front of the Hamilton--Jacobi functional that could be equal to the limit free energy in general \cite{mourrat2021nonconvex}. In analogy with the change of variables described below \eqref{e.parisi.convex.var}, one might alternatively want to use~\eqref{e.crit.p} and ask whether the limit free energy could be written in the form of the supremum or the infimum over $p \in \mcl Q_\infty$ of the functional
\begin{equation}  
\label{e.parisi-type.functional}
p \mapsto \psi(q + t \nabla \xi(p)) - t \int_0^1 (p\cdot \nabla \xi(p) - \xi(p)),
\end{equation}
but these are also invalid candidates in general \cite{mourrat2021nonconvex}.

Our first main result states that, under the assumption that the limit free energy exists, it must be a critical value of the Hamilton--Jacobi functional. We understand that the assumption of convergence of the free energy $\bar F_N$ to some limit $f$ is pointwise, that is, for every $(t,q) \in \R_+ \times \mcl Q_1$, we assume that $\lim_{N \to \infty} \bar F_N(t,q) = f(t,q)$.
\begin{theorem}[Critical point representation] 
\label{t.main1}
Assume that $\bar F_N$ converges to some $f$. Then for every $t \ge 0$ and $q \in \mcl Q_2$, there exists $(q',p) \in \mcl Q_\infty^2$ that is a critical point of~$\mcl J_{t,q}$ and is such that 
    \begin{equation}
    \label{e.main1.sg}
        \lim_{N \to \infty} \bar F_N(t,q) = \mcl J_{t,q}(q',p).
    \end{equation}
\end{theorem}
As will be explained in Propositions~\ref{p.barF_N_Lip} and~\ref{p.precompact}, the function $\bar F_N$ is Lipschitz in its arguments, with a Lipschitz constant that is uniformly bounded in $N$, so the fact that the sequence $(\bar F_N)_{N \ge 1}$ has convergent subsequences is a consequence of the Arzelà-Ascoli theorem. The assumption of convergence could be weakened to a local statement around the point $(t,q)$ for which we want \eqref{e.main1.sg} to hold, but we prefer to refrain from doing so for clarity. As shown in Proposition~\ref{p.1st_stronger_id}, one can also represent the limit free energy at $(t,q)$ as a critical value of the function in~\eqref{e.parisi-type.functional}. 

In fact, for most choices of $(t,q)$, we can identify an explicit choice for the critical point~$(q',p)$ appearing in Theorem~\ref{t.main1}, which is expressed in terms of the limit free energy~$f$. Readers who are familiar with the method of characteristics (or see \cite[Section~3.5]{HJbook}) may guess our choice of $(q',p)$: since the derivative of the solution ought to be constant along the characteristic line, we aim to choose $q'$ so that $\dr_q f(t,q) = \dr_q \psi(q')$, and thus $p = \dr_q \psi(q') = \dr_q f(t,q)$. 

To make this statement precise, we define $\mcl Q_{\uparrow}$ to be the set of all $q \in \mcl Q_2$ such that $q(0) = 0$ and there exists a constant $c > 0$ satisfying, for every $u \le v \in [0,1)$,
\begin{equation}
\label{e.def.Qup}
q(v) - q(u) \ge c (v-u)\, \id \quad \text{ and } \quad \ellipt(q(v) - q(u)) \le c^{-1},
\end{equation}
where $\id$ denotes the identity matrix, and $\ellipt(a)$ denotes the ratio between the largest and the smallest eigenvalues of the matrix $a \in S^D_+$.
Roughly speaking, the space $\mcl Q_\uparrow$ is the set of all $q \in \mcl Q_2$ that are ``uniformly increasing'' with a bounded ellipticity ratio. Although the space $\mcl Q_2$ has empty interior as a subspace of $L^2([0,1];S^\D)$, the space $\mcl Q_\uparrow$ will functionally play the role of the interior of $\mcl Q_2$. This space is for instance convenient when exploring the notion of Gateaux differentiability of a function defined on $\mcl Q_2$; we refer to Definition~\ref{d.gateaux} for the precise notion of Gateaux differentiability we use. We will show below that if $\bar F_N$ converges to some function $f$, then the limit $f$ must be Gateaux differentiable almost everywhere in $\mcl Q_2$. Our usage of the phrase ``almost everywhere'' requires some explanation here, since the space $\mcl Q_2 \subset L^2([0,1];S^\D)$ is infinite-dimensional. What we mean is that the set of points at which $f$ is not Gateaux differentiable is a Gaussian null set (see Definition~\ref{d.gaussian.null}). In particular, if $\bar F_N$ converges to some function $f$, then $f$ must be differentiable on a subset of $\R_+ \times (\mcl Q_\uparrow \cap L^\infty)$ that is dense in $\R_+ \times \mcl Q_2$. 
\begin{theorem}[Critical point identification]
\label{t.main2}
Assume that $\bar F_N$ converges to some limit~$f$. For every $t \ge 0$ and $q \in \mcl Q_\uparrow \cap L^\infty([0,1];S^D)$, if $f(t,\cdot)$ is Gateaux differentiable at $q$, then letting $p = \dr_q f(t,q)$ and $q' = q + t \nabla \xi(p)$, we have that $(q',p) \in \mcl Q_\infty^2$ is a critical point of~$\mcl J_{t,q}$ and is such that~\eqref{e.main1.sg} holds.
\end{theorem}
Up to a small perturbation of the free energy, we can also identify the limit of the law of the overlap matrix in terms of this critical point. We denote by $\langle \cdot \rangle$ the Gibbs measure associated with the enriched free energy with parameters $(t,q)$, see \eqref{e.def.disc.gibbs} for a precise definition. The canonical random variable under this measure is denoted by $\sigma$, and we denote by $\sigma'$ an independent copy of $\sigma$ under $\la \cdot \ra$, also called a replica. The overlap matrix is the random variable $N^{-1} \sigma \sigma'^\intercal$, taking values in $\R^{D\times D}$. 

We now introduce the perturbation to the Hamiltonian that we need to add. We let $(\hat H_N(\sigma))_{\sigma \in (\R^\D)^N}$ be the centered Gaussian field such that, for every $\sigma, \tau \in (\R^\D)^N$,
\begin{equation}\label{e.hatH_N}  \E \Ll[ \hat H_N(\sigma) \hat H_N(\tau) \Rr] = N \Ll| \frac{\sigma \tau^\intercal}{N} \Rr| ^2.
\end{equation}
This Gaussian field can be constructed explicitly by setting
\begin{equation*}  \hat H_N(\sigma) := \frac{1}{\sqrt{N}} \sum_{d = 1}^\D \sum_{i,j = 1}^N W_{i,j} \sigma_{d,i} \sigma_{d,j},
\end{equation*}
where $(W_{i,j})$ are independent standard Gaussians. We impose this new Gaussian field to be independent of all other sources of randomness. For every $\hat t \ge 0$, we define $\hat F_N(t,\hat t,q)$ to be the enriched free energy as discussed earlier and to which we add the following quantity in the exponential:
\begin{equation}  
\label{e.def.tdH}
\sqrt{2\hat t} \, \hat H_N(\si) - \hat t N \Ll| \frac{\sigma \sigma^\intercal}{N} \Rr| ^2;
\end{equation}
see also \eqref{e.def.FN.t.hat} for more precision. The corresponding functional $\mcl J$ then becomes
\begin{equation}  \label{e.hatJ=}
\hat {\mcl J}_{t,\hat t,q}(q',p) := \psi(q') +  \la p , q-q'\ra_\cH + t \int_0^1 \xi(p) + \hat t \int_0^1 |p|^2.
\end{equation}
\begin{theorem}[Identification of the law of the overlap]
\label{t.main3}
Assume that $\hat F_N$ converges to some limit $\hat f$ along a subsequence. Suppose also that, for some $t \ge 0$, $\hat t > 0$, and $q \in \mcl Q_\uparrow \cap L^\infty([0,1];S^D)$, we have that $\hat f(t,\hat t, \cdot)$ is Gateaux differentiable at $q$, and that $\hat f(t,\cdot,q)$ is differentiable at $\hat t$. Then letting $p = \dr_q \hat f(t,\hat t, q)$ and $q' = q + t \nabla \xi(p) + 2 \hat t p$, we have that $(q',p)\in \mcl Q_\infty^2$ is a critical point of $\hat {\mcl J}_{t,\hat t, q}$. Moreover, as $N$ tends to infinity along the said subsequence, the overlap matrix converges in law under $\E \la \cdot \ra$ to the random variable $p(U)$, where $U$ is a uniform random variable over $[0,1]$. 
\end{theorem}
Notice that in Theorem~\ref{t.main3}, we do not require that the free energy converges as $N$ tends to infinity, as convergence along a subsequence suffices; and we recall that the Lipschitz continuity properties of $\hat F_N$ guarantee that we can always extract convergent subsequences. If we do make the assumption that $\hat F_N$ converges along the full sequence, then an application of Theorem~\ref{t.main2} yields a representation of the free energy itself as
\begin{equation*}  \lim_{N \to \infty} \hat F_N(t,\hat t,q) = \hat{\mcl J}_{t,\hat t,q}(q',p),
\end{equation*}
for the same critical point $(q',p)$ of $\hat{\mcl J}_{t,\hat t,q}$. To be clear, we mention that the statement that $(q',p)$ is a critical point of $\hat{\mcl J}_{t,\hat t,q}$ means that
\begin{equation*}  q = q'-t \nabla \xi(p) - 2\hat t  p \quad \text{ and } \quad p = \dr_q \psi(q').
\end{equation*}
As for Theorem~\ref{t.main2}, for each $t \ge 0$, almost every choice of $(\hat t, q) \in \R_+ \times (\mcl Q_\uparrow \cap L^\infty)$ is a point that satisfies the assumptions of Theorem~\ref{t.main3}. In particular, the set of such points is dense in $\R_+ \times \mcl Q_2$. 

In fact, we will prove a stronger version of Theorem~\ref{t.main3} in which we identify the limit law of the array of overlaps involving arbitrarily many independent copies of $\sigma$ under~$\la \cdot \ra$.

A natural question is whether results such as Theorem~\ref{t.main1} characterize the limit free energy $f$ uniquely. The next proposition answers this question positively at high temperature.
\begin{proposition}[Uniqueness of critical point at high temperature]
\label{p.unique.crit.high.temp}
There exists $t_c > 0$ such that for every $t \in [0,t_c)$ and $q \in \mcl Q_2$, the function $\mcl J_{t,q}$ has a unique critical point in~$\mcl Q_2^2$. 
\end{proposition}
For $t> 0$ sufficiently small and $q = 0$, Theorem~\ref{t.main1} is shown in \cite{dey2021fluctuation}. Showing Theorem~\ref{t.main1} at high temperature for non-constant paths $q$ is not a straightforward adaptation of the argument in \cite{dey2021fluctuation} because the system is no longer replica-symmetric in this case, no matter how small $t$ is. 

The Hamilton--Jacobi functional may have multiple critical points in general. Interestingly, this fact does not completely rule out the possibility that Theorem~\ref{t.main1} might fully characterize the limit free energy. Indeed, one can exclude some spurious critical points by exploiting the continuity properties of the free energy, for instance by using the following simple observation.

\begin{proposition}[Relevant critical points must be stable]
\label{p.stability}
For each integer $n \ge 1$, let $(t_n,q_n) \in \R_+ \times \mcl Q_2$, and let $(q'_n,p_n) \in \mcl Q_2^2$ be a critical point of $\mcl J_{t_n,q_n}$ such that 
\begin{equation*}
\lim_{N \to \infty} \bar F_N(t_n,q_n) = \mcl J_{t_n,q_n}(q'_n,p_n).
\end{equation*}
Suppose that $(t_n,q_n)$ converges towards $(t,q) \in \R_+ \times \mcl Q_2$. Then the sequence $(q'_n,p_n)$ is precompact in $\mcl Q_2^2$. Moreover, any subsequential limit $(q',p) \in \mcl Q_2^2$ is a critical point of~$\mcl J_{t,q}$, and is such that
\begin{equation}
\label{e.stability.conclusion}
\lim_{N \to \infty} \bar F_N(t,q) = \mcl J_{t,q}(q',p).
\end{equation}
\end{proposition}
Proposition~\ref{p.stability} implies that among all the candidate critical points of $\mcl J_{t,q}$, we can exclude those that cannot be represented as limits of critical points at nearby locations. While we will not prove this here, we expect that some spurious critical points can indeed be excluded using this property. 

In order to avoid possible confusion, we mention that the notion of stability of critical points implied by Proposition~\ref{p.stability} is different from another possible notion of stability sometimes evoked by physicists that has its clearest meaning for convex models. Indeed, since Theorem~\ref{t.parisi.convex} gives a variational formula for the limit free energy in this case, we can deduce constraints on the Hessian of the Hamilton--Jacobi functional at the relevant critical point, at least at a formal level. From the arguments of \cite[Section~6]{mourrat2021nonconvex}, we do not expect that one can strengthen Theorem~\ref{t.main1} and prescribe the ``orientation'' of the Hessian of the Hamilton--Jacobi functional at a critical point $(q',p)$ in general. Our understanding is that physicists sometimes proceed by inspecting the orientation of the Hessian of the unique critical point of the high-temperature phase, and by then postulating that the same orientation must be valid at low temperature. While this might work in some cases, we would be surprised if this procedure was valid in general.

We now comment on our assumption of convergence of the free energy in Theorems~\ref{t.main1} and \ref{t.main2}. Even for certain scalar models (with $\xi$ that is not convex on $\R$), the only technique we are aware of for showing the convergence of the free energy to some limit consists in obtaining a full characterization of the limit; in other words, it consists in showing that Theorem~\ref{t.parisi.convex} holds. In view of this, we find it more promising to direct further effort towards a full identification of the limit for general models, rather than towards the discovery of an argument that would only guarantee convergence.

\subsection{Related works}
We have already mentioned a number of works that are related to the present paper. In particular, we recall that the limit free energy of the Sherrington--Kirkpatrick model was determined in \cite{gue03, Tpaper, Tbook1, Tbook2}. The proof for general scalar models ($D = 1$) was obtained in \cite{pan.aom, pan} by showing that the Ghirlanda--Guerra identities \cite{aizenman1998stability, ghirlanda1998general, guerra1996overlap} suffice to ensure the ultrametricity of the Gibbs measure. The version of Theorem~\ref{t.parisi.convex} with $\xi$ assumed to be convex over $\R^{D\times D}$ was obtained in \cite{barcon, pan.multi, pan.potts, pan.vec}. The main ingredient of the proof there is that by enforcing ultrametricity of the overlaps associated with a large family of linear combinations of the spins of different types, one can deduce that the overlaps corresponding to these different types synchronize with one another. These properties of ultrametricity and synchronization will also be used crucially in the present work. Related results for spherical models were obtained in \cite{chen2013aizenman, talagrand2006free} for scalar models, and in \cite{bates2022crisanti, bates2022free, ko2020free, panchenko2007overlap} for models with multiple types of spins.  

Connections between mean-field models of statistical mechanics and Hamilton--Jacobi equations have already been noted in \cite{bra83, new86}; we also refer to \cite{bauerschmidt2023stochastic} for a recent survey on related topics. In the context of spin glasses, these connections were first explored in \cite{abarra, barra2,barra1, guerra2001sum}, under a replica-symmetric or one-step replica symmetry breaking assumption. The fact that the Parisi formula can be seen as the Hopf--Lax formula of an infinite-dimensional Hamilton--Jacobi equation is from \cite{mourrat2022parisi, mourrat2020extending}; see also \cite{HJbook} for a pedagogical presentation. The rigorous justification of the fact that this Hopf--Lax formula is indeed the solution to the Hamilton--Jacobi equation is from \cite{chen2022hamilton}. An inequality between the limit free energy and the solution to the Hamilton--Jacobi equation is obtained in \cite{mourrat2021nonconvex, mourrat2023free} for general models. 

Other inequalities for the limit free energy of some non-convex models were obtained in \cite{alberici2020annealing, alberici2021deep}. We do not expect the bounds in \cite{alberici2020annealing, alberici2021deep} to be sharp, except possibly under particular symmetry conditions. For more on possible simplifications under symmetry conditions, we also refer to \cite{bates2023parisi, chen2023parisi, issa2024existence}. Some works also focus on the high-temperature regime, including \cite{barramulti, dey2021fluctuation, genovese2023minimax}.

For certain spherical models with multiple types and non-convex $\xi$ called ``pure models'', the limit free energy was identified explicitly in \cite{subag2021multi2, subag2021free, subag2021multi1, subag2021multi3}, assuming that this limit free energy exists; see also \cite{baik2020free} for the case of the bipartite spherical Sherrington--Kirkpatrick model. A more geometric analysis of the energy landscape is in \cite{benarous2022exponential, huang2023strong, kivimae2023ground, mckenna2021complexity}; see also  \cite{auffinger2013complexity, auffinger2013random, fyodorov2004complexity, subag2017complexity} for scalar models. 

In the low-temperature regime, the free energy is a good approximation for the maximum of the energy function $\sigma \mapsto H_N(\sigma)$ over the support of $P_N$ (see for instance \cite[Exercise~6.3]{HJbook}). We think that it would be interesting to work out what consequences our results entail for the asymptotic behavior of this maximum, in the limit of large system size. For some convex models, a Parisi-like variational formula capturing this was obtained in \cite{auffinger2017parisi, WeiKuo_lp_GG, TD_lp_GG}.  There has also been much recent progress on the problem of determining whether there exists a polynomial-time algorithm that identifies a configuration $\sigma$ in the support of $P_N$ such that $H_N(\sigma)$ is close to its maximal value. The answer turns out to depend on the specifics of the model, and relates to the overlap gap property~\cite{gamarnik2021overlap-paper}; see also \cite{gamarnik2021overlap-survey, gamarnik2022disordered} for surveys. In general, one can define a specific value~$\mathsf{ALG}$ such that there exists a polynomial-time algorithm that can identify a configuration $\sigma$ with $H_N(\sigma) \simeq N \mathsf{ALG}$ with high probability, but such that any algorithm within a broad class will demonstrably fail to improve upon this limit value $\mathsf{ALG}$, as shown first for scalar models in \cite{ams, huang2021tight, jekel2024potential, montanari2021optimization,  sellke2021optimizing, sub18}. Remarkably, these results have been successfully extended to general, possibly non-convex spherical models with multiple types in \cite{huang2023algorithmic, huang2023optimization}, even though the asymptotic behavior of the maximum of $H_N$ is currently not understood at this level of generality.

\subsection{Outline of the paper}
In Section~\ref{s.diff}, we explore the differentiability properties of Lipschitz functions on $\mcl Q_2$. In particular, we show that any such function is Gateaux-differentiable outside of a Gaussian null set. In Section~\ref{s.def.disc}, we define the free energy $\bar F_N(t,q)$ for piecewise-constant paths $q$ using discrete Poisson--Dirichlet cascades, and then extend it to any path $q \in \mcl Q_1$ by continuity. We also show that the sequence $(\bar F_N)_{N \in \N}$ is precompact, and that $\bar F_N$ is locally semi-concave uniformly over $N$. It is convenient to also give a direct definition of the free energy and its associated Gibbs measure for arbitrary paths $q \in \mcl Q_\infty$. For this purpose, we derive a number of preliminary results involving continuous Poisson-Dirichlet cascades in Section~\ref{s.cascades}, and then give this direct definition of the free energy in Section~\ref{s.def.cont}. In that section, we also study the differentiability properties of $\bar F_N$ and of any subsequential limit with respect to the path $q$. We turn to cavity calculations in Section~\ref{s.cavity}. Using the ultrametricity and synchronization results from \cite{pan.aom} and  \cite{pan.multi, pan.potts, pan.vec}, we extend the Aizenman--Simms--Starr scheme \cite{aizenman2003extended} and obtain precise representations of $\bar F_{N+1} - \bar F_N$ and of $\dr_q \bar F_N$. Information about $\dr_q \bar F_N$ gives us access to the asymptotic law of the overlap matrix. In Section~\ref{s.concl}, we combine these results to show Theorem~\ref{t.main2} and then Theorem~\ref{t.main1}, as well as some refinements involving overlap arrays. We also show that any subsequential limit $f$ of the free energy $\bar F_N$ must be such that, at every point $(t,q)$ of Gateaux differentiability of $f$,  we have
\begin{equation}  
\label{e.satisfy.eqn.intro}
\dr_t f(t,q) = \int_0^1 \xi(\dr_q f(t,q,u)) \, \d u.
\end{equation}
In particular, this shows that the function $f$ must satisfy the Hamilton--Jacobi equation~\eqref{e.hj} everywhere outside of a Gaussian null set. Finally, Section~\ref{s.convex} derives further consequences of our results in the special case when $\xi$ is convex over $S^D_+$. In particular, we prove Theorem~\ref{t.parisi.convex} there.

\subsection{Explicit representation of the Gaussian field \texorpdfstring{$H_N$}{H\_N}}
\label{ss.explicit.H}
Before closing this introduction, we clarify the class of functions $\xi$ such that \eqref{e.def.xi} can hold for some Gaussian process~$H_N$, and give an explicit construction of~$H_N$ in this case. By \cite[Proposition~6.6]{mourrat2023free}, the function~$\xi$ must take the form
\begin{equation}  
\label{e.explicit.xi}
\xi(a) = \sum_{p=1}^{+\infty}\mathsf{C}^{(p)}\cdot a^{\otimes p},
\end{equation}
where $\mathsf{C}^{(p)}\in \S^{\D^p}_+$ for each $p \ge 1$. In \eqref{e.explicit.xi}, the notation $a^{\otimes p}$ stands for the $p$-fold tensor product of the matrix $a \in \R^{D\times D}$, which we interpret as a $D^p$-by-$D^p$ matrix. We recall that we assume that the series in \eqref{e.explicit.xi} is absolutely convergent for every $a \in \R^{D\times D}$. The summation in \eqref{e.explicit.xi} starts at $p = 1$ since we assume that $\xi(0) = 0$. There is no loss of generality in this assumption, since the general case can be obtained by simply adding a fixed Gaussian random variable to $H_N$, and this changes neither the averaged free energy nor the Gibbs measure. 

To give an explicit construction of the Hamiltonian $H_N$, it suffices to consider the case when only one of the terms in the series in \eqref{e.explicit.xi} is non-zero, since the general case can then be obtained by adding independent Gaussian fields corresponding to each term in the series. We therefore fix $p \ge 1$, and let
$(J^{(\mathbf d)})_{\mathbf d \in \{1,\ldots D\}^p}$ be a centered Gaussian vector with covariance matrix $\msf C^{(p)}$. We let $((J^{(\mathbf d)}_{\mathbf i})_{\mathbf d \in \{1,\ldots, D\}^p})_{\mathbf{i} \in \{1,\ldots, N\}^p}$ denote $N^p$ independent copies of the $D^p$-dimensional random vector $(J^{(\mathbf d)})_{\mathbf d \in \{1,\ldots D\}^p}$. In particular, we have for every $\mathbf{d}, \mathbf{d}' \in \{1,\ldots, D\}^p$ and $\mathbf{i}, \mathbf{i}'\in \{1,\ldots, N\}^p$ that
\begin{equation*}  \E \Ll[ J_{\mathbf{i}}^{(\mathbf d)}  J_{\mathbf{i}'}^{(\mathbf d')} \Rr] = \1_{\{\mathbf{i} = \mathbf{i}'\}} \, \msf C^{(p)}_{\mathbf{d}, \mathbf{d}'} .
\end{equation*}
We then set, for every $\sigma = (\sigma_{d,i})_{1 \le d \le D, 1 \le i \le N} \in \R^{D\times N}$,
\begin{equation*}  H_N(\si) := N^{-\frac{p-1}{2}} \sum_{d_1, \ldots, d_p = 1}^D \ \sum_{i_1,\ldots, i_p = 1}^N J_{i_1, \ldots, i_p}^{(d_1, \ldots, d_p)} \si_{d_1,i_1} \cdots \, \si_{d_p,i_p}.
\end{equation*}
This Gaussian process satisfies \eqref{e.def.xi}  for the function $\xi$ given by
\begin{equation*}  
\xi(a) = \msf C^{(p)} \cdot a^{\otimes p}. 
\end{equation*}

\subsection{Assumption on the support of \texorpdfstring{$P_1$}{P\_1}}
All the results stated above are valid under the assumption that the reference measure $P_N = P_1^{\otimes N}$ is such that $P_1$ has compact support in $\R^D$. For convenience, we assume from now on that 
\begin{equation}
\label{e.bound.support}
\mbox{the support of $P_1$ is contained in the closed unit ball of $\R^D$.}
\end{equation}
Up to a rescaling of the variables, this entails no loss in generality.

\section{Differentiability of Lipschitz functions}\label{s.diff}

The main goal of this section is to show that a Lipschitz function on the space of paths $\C$ is differentiable ``almost everywhere''. We do not know whether a statement in this spirit is valid if the notion of differentiability is meant in the sense of Fréchet. We will show that it is indeed valid if the notion of differentiability is somewhat weakened, and with ``almost everywhere'' being understood in the sense that the set of points of non-differentiability is a Gaussian null set. 

Throughout this section, we let $E$ be a separable Banach space, and let $G$ be a subset of $E$. In our application, we will choose $E$ to be $L^2([0,1]; S^D)$ and $G = \C_2$. We denote by $|\cdot|_E$ the norm on the space $E$, by $E^*$ the (continuous) dual of $E$, and by $\la \cdot, \cdot \ra_E$ the canonical pairing of $(E^*,E)$. We start by giving precise definitions of the notions of differentiability we use. 
\begin{definition}[Fr\'echet differentiability]A function $g:G\to\R$ is \emph{Fr\'echet differentiable} at ${q} \in G$ if there is a unique $y^* \in E^*$ such that
\begin{align}
\label{e.def.frechet}
    \lim_{r\to 0}\sup_{\substack{{q'}\in G\setminus\{{q}\}\\ |{q'}-{q}|_E\leq r}}\frac{\Ll|g({q'}) - g({q}) - \la y^* , {q'}-{q}\ra_E\Rr|}{|{q'}-{q}|_E} =0.
\end{align}
In this case, we call $y^*$ the \emph{Fréchet derivative} of $g$ at $q$.
\end{definition}
The condition in \eqref{e.def.frechet} can be rewritten as
\begin{align*}
    g({q'}) - g({q}) = \la y^*, {q'}-{q} \ra_E + o\Ll(\Ll|{q'}-{q}\Rr|_E\Rr)
\end{align*}
for ${q'} \in G$ tending to ${q}$. The notion of Fréchet differentiability is often defined only when~$G$ is an open set, in which case if there exists $y^* \in E^*$ satisfying \eqref{e.def.frechet}, then it is necessarily unique. But we do not require that $G$ be open here.

For every $q\in G$, we define
\begin{align*}
\mathrm{Adm}(G,q) :=\Ll\{e\in E\ \big|\  \exists r>0 \, : \, \forall t \in [0,r],\ q+te\in G\Rr\}
\end{align*}
to be the set of admissible directions at $q$ along which a small line segment is contained in $G$.
\begin{definition}[Gateaux differentiability]
\label{d.gateaux}
    A function $g:G\to \R$ is \emph{Gateaux differentiable} at $q\in G$ if the following two conditions hold:
\begin{itemize}
    \item  for every $e\in \mathrm{Adm}(G,q)$, the following limit exists:
    \begin{equation}
    \label{e.def.gateaux.lim}
    g'(q,e) := \lim_{r\searrow 0} \frac{g(q+re)-g(q)}{r};
    \end{equation}
    
    \item there is a unique $y^*\in E^*$ such that  for every $e\in \mathrm{Adm}(G,q)$, we have $g'(q,e) = \la y^*,e\ra_E$.
\end{itemize}
In this case, we call $y^*$ the \emph{Gateaux derivative} of $g$ at $q$.
\end{definition}
We stress that our definition of Gateaux differentiability may differ from the  most classical one. For one thing, the notion is most commonly used only when $G$ is an open set. Most importantly, the notion of Gateaux differentiability is often understood in the sense that only the first of the two properties we listed is valid, while here we require in particular that the mapping $e \mapsto g'(q,e)$ be linear and continuous. We refer to \cite{nashed1966remarks} for a comparison of the different notions.  When $G$ is an open set, our definition coincides with those of \cite[Definition~4]{nashed1966remarks}, \cite[(d) page~528]{phelps1978gaussian}, and \cite[Definition~4.1]{benyamini2000geometric} (we impose $r > 0$ in~\eqref{e.def.gateaux.lim} while the limit is two-sided in \cite{benyamini2000geometric, phelps1978gaussian}, but these are equivalent since we also have that $g'(q,e) = -g'(q,-e)$). An example of a Lipschitz function that is nowhere Fréchet differentiable is the $\ell^1$ norm on the space of absolutely summable sequences; this function is Gateaux differentiable at every nowhere-vanishing sequence.

In the problem we are ultimately interested in, the function defined on $G$ admits a Lipschitz extension on the full space $E$. Under the assumption that the norm $|\cdot |_E$ is Fréchet differentiable away from the origin, it was shown in \cite{preiss1990differentiability} that a Lipschitz function on $E$ must be Fr\'echet differentiable on a dense set. However, what we really want is to identify a dense set of points of differentiability \emph{in $G$}. Since in our application, the set $G$ has empty interior, the existence of a dense set of points of differentiability in $E$ does not guarantee the existence of any such point in $G$. 

We will therefore rely instead on the fact, proved in \cite{aronszajn1976differentiability, phelps1978gaussian}, that a Lipschitz function on $E$ must be Gateaux differentiable outside of a Gaussian null set. As will be explained next, a Gaussian null set is sufficiently ``thin'' that in our application, we can deduce in particular the existence \emph{in $G$} of a dense set of points of Gateaux differentiability.

We now give a precise definition of the notion of a Gaussian null set, and present the relevant results from \cite{phelps1978gaussian}. We equip $E$ with the Borel $\sigma$-algebra generated by its norm topology.

\begin{definition}
    A probability measure $\mu$ on $E$ is said to be a \emph{non-degenerate Gaussian measure} with mean $a \in E$ if for every non-zero $e^*\in E^*$, the measure $\mu\circ (e^*)^{-1}$ is a Gaussian measure with mean $e^*(a)$ and non-zero variance.
\end{definition}
Although we will not make use of this fact, we mention in passing that when $E$ is a separable Banach space as we assume here, the mean of a Gaussian measure can always be represented as an element of $E$, by \cite[Theorem~3.2.3]{bogachev1998gaussian}. This is easier to see when $E$ a separable Hilbert space (e.g. \cite[Theorem~2.3.1]{bogachev1998gaussian}), or more generally when $E$ is reflexive. In our application the space $E$ will be a separable Hilbert space.
\begin{definition} 
\label{d.gaussian.null}
    A Borel subset $B$ of $E$ is said to be a \emph{Gaussian null set} if for every non-degenerate Gaussian measure $\mu$ on $E$, we have $\mu(B)=0$.
\end{definition}
The following lemma is a simple modification of \cite[Lemma~3]{phelps1978gaussian}; we provide a complete proof for the reader's convenience.

\begin{lemma}\label{l.not_Gaussian_null}
Let $(w_n)_{n\in\N}$ be a sequence in $E$ that has dense linear span and satisfies $\lim_{n\to\infty}|w_n|_E=0$, and let $K$ be the closed convex hull of $\{0\}\cup\{w_n:n\in\N\}$. For every $x\in E$, the set $x+K$ is not a Gaussian null set.
\end{lemma}

\begin{proof}
We denote by $\ell^2$ the Hilbert space of square-summable sequences indexed by $\N$, and decompose the proof into two steps.

\medskip

\noindent \emph{Step 1.} Let $\mu$ be a non-degenerate Gaussian measure on $\ell^2$ with mean $a \in \ell^2$. In this step, we show that $\mu$ attributes a positive mass to every open set; or equivalently, that the support of $\mu$ is the full space $\ell^2$. By \cite[Theorem~2.3.1]{bogachev1998gaussian}, there exists a continuous self-adjoint operator $A : \ell^2 \to \ell^2$ such that for every $u \in \ell^2$, the measure $\mu \circ u^{-1}$ is Gaussian with mean $u \cdot a$ and variance $u \cdot A u$. In the notation $\mu \circ u^{-1}$, we interpret $u$ as the element of the dual of $\ell^2$ canonically associated with $u$, namely the mapping $x \mapsto u \cdot x$. Since $\mu$ is non-degenerate, the kernel of $A$ must be trivial. Since $A$ is self-adjoint, this implies that the image of $A$ is dense in $\ell^2$. We also observe that for every $h \in \ell^2$, the translation of $\mu$ by $Ah$ is equivalent to $\mu$, since its Radon--Nikodym derivative with respect to $\mu$ is the mapping
\begin{equation*}  x \mapsto \exp \Ll( (x-a) \cdot h - \frac 1 2 h \cdot A h  \Rr) .
\end{equation*}
Now suppose that there exists an open set $U \subset \ell^2$ such that $\mu(U) = 0$. We deduce that the translation of $U$ by any $Ah$, $h \in \ell^2$ also has zero measure. But since the image of $A$ is dense in $\ell^2$, we can find a sequence $(h_n)_{n \in \N}$ such that 
\begin{equation*}  \bigcup_{n \in \N} (U + A h_n) = \ell^2. 
\end{equation*}
We have thus reached the absurd conclusion that $\mu(\ell^2) = 0$; so there exists no open set~$U \subset \ell^2$ such that $\mu(U) = 0$. 

\medskip

\noindent \emph{Step 2.} 
We define the linear map $L:\ell^2\to E$ such that, for every $y = (y_n)_{n\in\N}\in \ell^2$,
\begin{align*}
    Ly := \sum_{n\in\N}2^{-n}y_n w_n.
\end{align*}
Let $U$ be the unit ball of $\ell^2$. Since the set $\{0\}\cup\{w_n:n\in\N\}\cup\{-w_n:n\in\N\}$ is compact in $E$, its closed convex hull $G$ is also compact (see for instance \cite[Theorem~5.35]{aliprantis2006infinite}).
Since $|y_n|\leq 1$ for every $y\in U$ and $n \in \N$, we have $LU \subset G$. Compactness of $G$ implies that $G$ is bounded and thus $L$ is continuous. We define $T:\ell^2\to E$ such that, for every $y \in \ell^2$,
\begin{equation*}  T(y) := x+\frac{1}{2}Ly + \sum_{n\in\N} 2^{-n-1}w_n =x+ \sum_{n\in\N} 2^{-n}\Ll(\frac{y_n+1}{2}\Rr)w_n.
\end{equation*}
Since $\frac{y_n+1}{2} \in [0,1]$ for every $y\in U$ and $n \in \N$, we have $T(U)\subset x+K$. We let $\mu$ be a non-degenerate Gaussian measure on $\ell^2$. By the result of the previous step, we have that $\mu$ attributes a positive mass to $U$. We denote $\nu := \mu \circ T^{-1}$. Since $T$ is affine and continuous, the measure $\nu$ is a Gaussian measure on $E$ (with a mean in $E$); and since $T(U) \subset x + K$, it attributes a positive mass to $x + K$. To see that it is non-degenerate, let $e^* \in E^*$ be non-zero. For every $y \in \ell^2$, we have
\begin{equation*}  \la e^*,T(y)\ra_E = \la z, y\ra_{\ell^2} + c ,
\end{equation*}
where
\begin{equation*}  z= \Ll(2^{-n-1}\la e^*,w_n\ra_E\Rr)_{n\in\N} \quad \text{ and } \quad c = \la e^*,x+\sum_{n\in\N}2^{-n-1}w_n\ra_E.
\end{equation*}
Recall that the linear span of $(w_n)_{n \in \N}$ is dense in $E$. If $z$ was zero, then by linearity and continuity of $e^*$, we would deduce that $e^*$ is zero as well, which is excluded. Since $\nu \circ (e^{*})^{-1}$ is a translation of the measure $\mu \circ z^{-1}$, and $\mu$ is non-degenerate, we conclude that $\nu$ is non-degenerate as well. To sum up, we have found a non-degenerate Gaussian measure $\nu$ such that $\nu(x+K) > 0$, so the set $x + K$ is not a Gaussian null set.
\end{proof}

The next theorem is a special case of \cite[Theorem~6]{phelps1978gaussian}. This result is itself based on \cite[Theorem~II.2.1]{aronszajn1976differentiability}, which shows a similar result using another notion of thin sets now called ``Aronszajn null sets'', and the main contribution of \cite{phelps1978gaussian} is to show that these sets are Gaussian null sets.  We also refer to \cite[Theorem~6.42 and Proposition~6.27]{benyamini2000geometric} for a recent monograph presentation of this result.

\begin{theorem}[Differentiability of Lipschitz functions \cite{aronszajn1976differentiability,phelps1978gaussian}]\label{t.gateaux_diff}

    If a function $g: E\to\R$ is  Lipschitz continuous, then it is Gateaux differentiable outside of a Gaussian null subset of $E$.
\end{theorem}

We now specialize this result to the sets $E$ and $G$ of interest to us, namely $E= L^2([0,1]; S^D)$ and $G = \mcl Q_2$. We recall that the $\mcl Q_2$ is the intersection of the set $\C$ defined in \eqref{e.def.mclQ} with $L^2([0,1];S^D)$. For every $p \in [1,\infty]$, we use the shorthand notation $L^p := L^p([0,1];S^D)$. We also recall that the set $\C_\uparrow$ is defined in the sentence containing~\eqref{e.def.Qup}. We observe that the set $\C_\uparrow\cap L^\infty$ is dense in $\C_2$. Indeed, denoting by $S^D_{++}$ the set of positive definite matrices, we have that 
\begin{equation*}  \Ll\{ u \mapsto cu\, \id + \sum_{k = 0}^K q_k \1_{ \Ll[ \frac{k}{K+1}, \frac{k+1}{K+1} \Rr) } \mid c > 0, \ q_0, \ldots, q_K \in S^D_{++} \Rr\} ,
\end{equation*}
is dense in $\mcl Q_2$ and is a subset of $\mcl Q_\uparrow \cap L^\infty$. 
\begin{proposition}\label{p.Gateaux_dff_dense}
If $g:\C_2\to\R$ is Lipschitz continuous, then there exists a Gaussian null set $\mathcal{N}$ of $\cH$ such that $g$ is Gateaux differentiable on $\C_2\setminus\mathcal{N}$. Moreover, $(\C_\uparrow\cap L^\infty)\setminus\mathcal{N}$ is dense in $\C_2$.

Similarly, if $g:\R_+\times\C_2\to\R$ is Lipschitz continuous (in its two variables), then there exists a Gaussian null set $\mathcal{N}$ of $\R\times\cH$ such that $g$ is Gateaux differentiable jointly in its two variables on $(\R_+\times\C_2)\setminus\mathcal{N}$. Moreover, $(\R_+\times(\C_\uparrow\cap L^\infty))\setminus\mathcal{N}$ is dense in $\R_+\times\C_2$.
\end{proposition}

\begin{proof}
We only prove the first part of the proposition, since the second part is obtained similarly. We decompose the proof into two steps. 

\medskip

\noindent \emph{Step 1}. We construct a Lipschitz extension $\bar g : L^2 \to \R$ to the Lipschitz function $g : \mcl Q_2 \to \R$. 

We denote by $P : L^2 \to \mcl Q_2$ the orthogonal projection to $\mcl Q_2$. By definition, for every $q \in L^2$, the point $P(q)$ is the unique minimizer in $\mcl Q_2$ of the mapping 
\begin{equation*}  \Ll\{
\begin{array}{rcl}  \mcl Q_2 & \to & \R \\
q' & \mapsto & |q-q'|_{L^2}^2.
\end{array}
\Rr.
\end{equation*}
We refer to \cite[Exercise~2.16 and solution]{HJbook} for a justification that the mapping $P$ is well-defined, based on the fact that $\mcl Q_2$ is a closed and convex subset of $L^2$. It is also shown there using the convexity of $\mcl Q_2$ that, for every $q \in L^2$ and $p \in \mcl Q_2$, we have
\begin{equation}  
\label{e.first.variation.P}
\la q - P(q) , p - P(q) \ra_\cH \le 0.
\end{equation}
We show that, for every $q,q' \in L^2$,
\begin{equation}  
\label{e.lipschitz.projection}
|P(q') - P(q)|_{L^2} \le|q'-q|_{L^2}.
\end{equation}
For every $q,q' \in L^2$, we use the identity \eqref{e.first.variation.P} with $p$ substituted with $P(q')$, and similarly with $q$ and $q'$ interchanged, to get that 
\begin{equation*}  \la q- P(q), P(q') - P(q) \ra_\cH \le 0 \le \la q'-P(q'), P(q') - P(q) \ra_\cH.
\end{equation*}
Rearranging and using the Cauchy--Schwarz inequality, this yields that 
\begin{equation*}  |P(q') - P(q)|_{L^2}^2 \le \la q'-q, P(q') - P(q) \ra_\cH \le |q'-q|_{L^2} \, |P(q') - P(q)|_{L^2}.
\end{equation*}
This shows \eqref{e.lipschitz.projection}. To construct the extension of $g$ to $L^2$, we set, for every $q \in \mcl Q_2$,
\begin{equation*}  \bar g(q) := g(P(q)). 
\end{equation*}
Since $P$ is the identity on $\mcl Q_2$, the function $\bar g$ is indeed an extension of the function $g$. In view of \eqref{e.lipschitz.projection} and of the fact that $g$ is Lipschitz, the function $\bar g$ is also Lipschitz, as desired. 

\medskip

\emph{Step 2.}
We let $\bar g : L^2 \to \R$ be a Lipschitz extension of $g$ as given by the previous step. By Theorem~\ref{t.gateaux_diff}, there exists a Gaussian null set $\mathcal{N}\subset \cH$ such that $\bar g$ is Gateaux differentiable on $\cH\setminus \mathcal{N}$. 
Letting $x \in \C_2 \setminus \mcl N$, we verify that $g$ is Gateaux differentiable at $x$. Let $y^* \in (L^2)^* \simeq L^2$ denote the Gateaux derivative of $\bar g$ at $x$. The limit $g'(x,e)$ in~\eqref{e.def.gateaux.lim} exists for every $e \in \mathrm{Adm}(\C_2,x)$, and it is equal to $\la y^*, e \ra_{L^2}$. One can check that the orthogonal of $\C_2$ in $L^2$ is trivial, so the closed linear span of $\C_2$ is $L^2$. Moreover, we clearly have $\C_2\subset \mathrm{Adm}(\C_2,x)$, so there can be only one $y^*$ such that $g'(x,e) = \la y^*, e \ra_{L^2}$ for every $e \in \mathrm{Adm}(\C_2,x)$. This shows that $g$ is Gateaux differentiable at $x \in \C_2\setminus\mathcal{N}$.

We next show that $(\C_\uparrow\cap L^\infty)\setminus\mathcal{N}$ must be dense in $\C_2$.
Notice that the set 
\begin{equation*}  W:=\{q\in \C:\: q(0)=0\ \text{and}\ \forall u\leq v\in[0,1),\ q(v)-q(u)\leq (v-u)\id \}
\end{equation*}
has dense linear span in $\cH$.
Let $(w_n)_{n\in\N}$ be a sequence in $W$ whose linear span is dense in $\cH$. We can rescale~$w_n$ to ensure that $|w_n|_\cH\leq 1$ and $\lim_{n\to\infty} |w_n|_\cH=0$. Let $K$ be as given in Lemma~\ref{l.not_Gaussian_null}. We fix any $x\in \C_\uparrow\cap L^\infty$ and any $\eps>0$. By Lemma~\ref{l.not_Gaussian_null}, we have that $x+\eps K$ is not a Gaussian null set, which implies that $(x+\eps K)\not\subset\mathcal{N}$. Hence, there is $y \in (x+\eps K)\setminus\mathcal{N}$ such that $g$ is Gateaux differentiable at $y$.

To complete the argument, we show that $x+\eps K \subset \C_\uparrow\cap L^\infty$. Let $\kappa\in K$ and let $c$ be the constant in~\eqref{e.def.Qup} for $x$. We write $x':=x+\eps \kappa$ and we clearly have $x'(0)=0$. Fix any $0\leq u\leq v\leq 1$. We have $x'(v)-x'(u)\geq x(v)-x(u) \geq c(v-u)\id$. Let $\lambda_{\max}$ and $\lambda_{\min}$ be the largest and the smallest eigenvalue of $x(v)-x(u)$, respectively. Note that $\lambda_{\min}\geq c(v-u)$ and that the eigenvalues of $\kappa(v)-\kappa(u)$ lie in $[0,v-u]$. So, we get
\begin{align*}
    \ellipt(x'(v)-x'(u)) \leq \frac{\lambda_{\max}+\eps(v-u)}{\lambda_{\min}+0} = \frac{\lambda_{\max}}{\lambda_{\min}} + \frac{\eps(v-u)}{\lambda_{\min}} \leq c^{-1}+\frac{\eps}{c}.
\end{align*}
So, we conclude that $x'\in \C_\uparrow$, and we clearly also have that $x'\in L^\infty$. Therefore, we have verified $x+\eps K \subset \C_\uparrow\cap L^\infty$.

Hence, $y\in (\C_\uparrow\cap L^\infty)\setminus\mathcal{N}$.
Since $|y-x|_\cH\leq \eps$ and $\ep > 0$ was arbitrary, we conclude that $(\C_\uparrow\cap L^\infty)\setminus\mathcal{N}$ is dense in $\C_\uparrow\cap L^\infty$ and thus also in $\C_2$. 
\end{proof}

\section{Definition and regularity properties of the free energy}
\label{s.def.disc}

In this section, we define the free energy $\bar F_N(t,q)$ for arbitrary choices of the path $q \in \mcl Q_1$. When the path $q$ is constant equal to $h \in S^D_+$, we have already displayed the definition of the free energy in \eqref{e.def.simple.F}. Here we will first extend this definition to piecewise-constant $q \in \mcl Q$. One can then verify that the free energy is $L^1$-Lipschitz continuous in~$q$, and this therefore defines the free energy $\bar F_N(t,q)$ for arbitrary choices of $q \in \mcl Q_1$ by density. 
Another consequence of this Lipschitz continuity property is that the sequence of functions $(\bar F_N)_{N \in \N}$ is precompact. We also show that the free energy is locally semi-concave. This property will be used later to assert the convergence of the derivatives of $\bar F_N$ at any point of differentiability of the limit.

We start by defining the free energy $\bar F_N(t,q)$ for piecewise-constant $q \in \mcl Q$. When $q$ is constant equal to $h \in S^D_+$, we had introduced in~\eqref{e.def.simple.F} the external random magnetic field $\sqrt{2h} z$ acting on the spin configuration $\sigma$, where $z$ takes values in $\R^{D\times N}$ and has independent standard Gaussian entries. We will here replace this random magnetic field by a richer object that, in the language of physics, has a finite number of replica-symmetry breaks. Precisely, for each choice of integer $K \in \N$ (the number of replica-symmetry breaks) and parameters $(q_k)_{0 \le k \le K} \subset S^D_+$ and $(\zeta_k)_{1 \le k \le K} \subset [0,1]$ satisfying the constraints
\begin{equation}
\label{e.def.q}
0= q_{-1} \le q_0 < q_1 < \cdots < q_K
\end{equation}
and
\begin{equation}
\label{e.def.zeta}
0 = \zeta_0 < \zeta_1 < \cdots < \zeta_K < \zeta_{K+1} = 1,
\end{equation}
we aim to make sense of $\bar F_N(t,q)$ where $q$ is given by
\begin{equation}  
\label{e.def.piece.const.q}
q = \sum_{k = 0}^K q_k \1_{ \Ll[ \zeta_{k}, \zeta_{k+1} \Rr) }.
\end{equation}
In \eqref{e.def.q} and throughout the paper, whenever $a,b \in S^D$, we write $a < b$ to mean that $b-a \in S^D_+$ and $a \neq b$. The construction involves the notion of Poisson--Dirichlet cascades, which we now introduce briefly. We let $\mcl A$ denote the infinite-ary tree of depth $K$, which we encode as
\begin{align}\label{e.tree}
\mcl {A} := \N^0 \cup \N^1\cup\cdots \cup \N^K,
\end{align}
with the understanding that $\N^0 =\{\emptyset\}$. We understand that $\emptyset$ is the root of the tree $\mcl A$, and that for each $k \in \{0,\ldots,K-1\}$ and $\alpha = (n_1,\ldots, n_k) \in \N^k$, the children of $\alpha$ are all the nodes
\begin{equation*}  \alpha n := (n_1, \ldots, n_k,n) \in \N^{k+1},
\end{equation*}
where $n$ ranges in $\N$. For every $\alpha = (n_1,\ldots,n_k) \in \mcl A$, we denote its depth by $|\alpha| := k$, and for every $\ell \le k$, we write
\begin{equation*}  
\alpha_{|\ell} := (n_1,\ldots,n_\ell) 
\end{equation*}
to denote the ancestor of $\alpha$ at depth $\ell$. For any two leaves $\alpha, \beta \in \N^K$, we write
\begin{equation}  
\label{e.def.wedge}
\alpha \wedge \beta := \max \Ll\{ k \le K \mid \alpha_{|k} = \beta_{|k} \Rr\} 
\end{equation}
to denote the depth of the most recent common ancestor to $\alpha$ and $\beta$. The \emph{Poisson--Dirichlet cascade} (also called Ruelle probability cascade after \cite{ruelle1987mathematical}) with parameters $(\zeta_k)_{1 \le k \le K}$  is a random probability measure on the leaves $\N^K$, whose weights we denote by $(v_\alpha)_{\alpha \in \N^K}$; in particular, $v_\alpha \ge 0$ and $\sum_{\al \in \N^K} v_\al = 1$. Briefly, this probability measure is constructed as follows (more details can be found in \cite[Section~2.3]{pan} or \cite[Section~5.6]{HJbook}). With each non-leaf vertex $\alpha$ at depth $k \le K-1$, we attach an independent Poisson point process on $\R_+$ with intensity measure $x^{-1-\zeta_{k+1}} \, \d x$. This Poisson point process can be ordered decreasingly, say $u_{\alpha 1} \ge u_{\alpha 2} \ge \cdots$, and we attribute the weight $u_{\alpha n}$ to the child vertex $\alpha n$. This attributes a weight to each vertex except the root. Finally, for each leaf vertex $\alpha \in \N^K$, we compute the product of all the weights attached to the nodes from the root to the leaf vertex $\alpha$, excluding the root, and we denote the result by $w_\alpha$, that is, $w_\alpha := \prod_{k = 1}^{K} u_{\alpha_{|k}}$. The family $(v_\alpha)_{\alpha \in \N^K}$ is then obtained by normalizing the family $(w_\alpha)_{\alpha \in \N^K}$ into a probability measure, that is, $v_\alpha := w_\alpha / \sum_\beta w_\beta$. We construct this object independently of $(H_N(\sigma))_{\sigma \in \R^{D\times N}}$.

In order to define the random magnetic field, we also give ourselves $(z_\beta)_{\beta\in \mathcal{A}}$ a collection of independent standard $\R^{D \times N}$-valued Gaussian vectors, independent of the Gaussian field $(H_N(\sigma))$ and the cascade~$(v_{\alpha})$. 
For every $\alpha \in \N^K$, we set
\begin{align}\label{e.Z(h_alpha)=}
    \msf w^q(\al) := \sum_{k = 0}^K(q_k - q_{k-1})^{1/2} z_{\alpha_{|k}}.
\end{align}
The random variable $\msf w^q(\alpha)$ takes values in $\R^{D\times N}$. 
For every $\al, \al' \in \N^K$, we have
\begin{align*}\E \msf w^q(\al)\msf w^q(\al')^\intercal = N q_{\al\wedge\al'}.
\end{align*}
Recalling the definition of the piecewise-constant path $q$ in \eqref{e.def.piece.const.q}, we define
\begin{equation}
\label{e.def.FN.piece.const}
\bar F_N(t,q) := -\frac 1 N \E \log \int \sum_{\alpha \in \N^K} \exp \Ll( H_N^{\mathrm{disc}}(t,q,\sigma,\alpha) \Rr) \, v_\alpha \, \d P_N(\sigma),
\end{equation}
where
\begin{equation*}
H_N^{\mathrm{disc}}(t,q,\sigma,\alpha) := \sqrt{2t} H_N(\sigma) -  t N \xi\Ll(\frac{\sigma\sigma^\intercal}{N}\Rr) 
+ \sqrt{2} \msf w^q \cdot \sigma - \sigma \cdot q_K  \sigma.
\end{equation*}
We denote by $\langle \cdot \rangle$ the associated Gibbs measure, and by $(\sigma,\alpha)$ the canonical random variable. That is, for every bounded measurable $f : \R^N \times \N^K \to \R$,
\begin{equation}
\label{e.def.disc.gibbs}
\la f(\sigma,\alpha) \ra := \frac{ \int \sum_{\alpha \in \N^K} f(\sigma, \alpha) \exp \big(H_N^{\mathrm{disc}}(t,q,\sigma,\alpha)\big) v_\alpha \, \d P_N(\sigma) } { \int \sum_{\alpha \in \N^K} \exp \big(H_N^{\mathrm{disc}}(t,q,\sigma,\alpha)\big) v_\alpha \, \d P_N(\sigma) }.
\end{equation}
We denote by $((\sigma^\ell, \alpha^\ell))_{\ell \in \N}$ independent copies of $(\sigma, \alpha)$ under $\la \cdot \ra$, also called replicas. We may also write $(\sigma',\alpha')$, $(\sigma'',\alpha'')$ for independent copies of $(\sigma,\alpha)$ in calculations in which only two or three such copies suffice. As is classical and will be explained in more details below (or see \cite[(6.60)]{HJbook}), one can show that for every $k \in \{0,\ldots, K\}$, 
\begin{equation}  
\label{e.first.cascade.identity}
\E \la \1_{\{\alpha \wedge \alpha'= k\}} \ra = \zeta_{k+1} - \zeta_k. 
\end{equation}
Using this, one can show the following Lipschitz continuity property of $\bar F_N$ as a function of the path $q$. 
\begin{proposition}[Lipschitz continuity of $\bar F_N$]
\label{p.barF_N_Lip}
For every $t,t'\ge 0 $ and every piecewise-constant ${q},{q'} \in \C$, we have
\begin{align}
\label{e.barF_N_Lip}
    \Ll|\bar F_N(t,{q}) - \bar F_N(t',{q'})\Rr|\leq \Ll|{q}-{q'}\Rr|_{L^1} + |t-t'| \, \sup_{|a| \le 1} |\xi(a)|.
\end{align}
As a consequence, the free energy $\bar F_N$ admits a unique extension to $\R_+ \times \C_1$ satisfying~\eqref{e.barF_N_Lip} for every $t,t' \ge 0$ and $q,q' \in \C_1$. 
\end{proposition}
In \eqref{e.barF_N_Lip} and throughout, for every $p \in [1,\infty]$, we use the notation
\begin{equation*}
|q-q'|_{L^p} := \Ll( \int_0^1 |q-q'|^p \Rr) ^{1/p},
\end{equation*}
with the usual understanding as an essential supremum when $p = \infty$. We also often use the shorthand $L^p := L^p([0,1]; S^D)$, and we recall that $\C_p := \C \cap L^p$. 
\begin{proof}[Proof of Proposition~\ref{p.barF_N_Lip}]
The Lipschitz continuity in $q$ is in \cite[Proposition~3.1]{mourrat2023free}; see also \cite[Proposition~6.3]{HJbook} for a proof when $D = 1$ using similar notation as here. Besides the verification of \eqref{e.first.cascade.identity}, the main step of the proof is to show that if we allow for repetitions in the sequence $(q_k)_{0 \le k \le K}$, replacing the strict inequalities ``$<$'' in \eqref{e.def.q} by ``$\le$'', and follow along with the exact same construction of $\bar F_N$, then we end up with the same quantity as if we had removed the repetitions in the sequence $(q_k)_{0 \le k \le K}$ (and thereby reduced the value of $K$ and removed some ``unused'' $\zeta_k$'s in the representation \eqref{e.def.piece.const.q}) to begin with. 

For the Lipschitz continuity in $t$, we use the Gaussian integration by parts in \cite[Theorem~4.6]{HJbook}, to write, for every $t \ge 0$ and piecewise-constant $q \in \mcl Q$,
\begin{equation}
\label{e.der.t}
\partial_t \bar F_N(t,q) = -\frac 1 N \E \la  \frac 1 {\sqrt{2t}} H_N(\sigma) - N \xi \Ll( \frac{\sigma \sigma^\intercal}{N} \Rr) \ra = \E \la \xi \Ll( \frac{\sigma \sigma'^\intercal}{N} \Rr) \ra . 
\end{equation}
By the assumption \eqref{e.bound.support} on the support of $P_1$ and the Cauchy--Schwarz inequality, we have for every~$\sigma$ in the support of $P_N$ that
\begin{equation}  
\label{e.bound.on.sigma}
|\sigma \sigma^\intercal|^2 = \sigma \sigma^\intercal \cdot \sigma \sigma^\intercal =  |\sigma^\intercal \sigma|^2 = \sum_{i,j = 1}^N \Ll( \sum_{d= 1}^D \sigma_{d,i} \sigma_{d,j} \Rr) ^2 \le N^2. 
\end{equation}
In particular, we have that
\begin{equation*}  \Ll| \partial_t \bar F_N(t,q) \Rr| \le \sup_{|a| \le 1} |\xi(a)|,
\end{equation*}
so the proof is complete. 
\end{proof}

We also recall the following observation justifying the identity in \eqref{e.def.psi}.
\begin{proposition}[Initial condition]
\label{p.init.cond}
For every $N \in \N$ and $q \in \mcl Q_1$, we have
\begin{equation}  
\label{e.init.cond}
\bar F_N(0,q) = \bar F_1(0,q).
\end{equation}
\end{proposition}
\begin{proof}
The proof of this result is based on the fact that we chose $P_N$ to be a product measure, $P_N = P_1^{\otimes N}$. 
By the Lipschitz property in Proposition~\ref{p.barF_N_Lip}, it suffices to show the identity \eqref{e.init.cond} for piecewise-constant $q$. The proof is then classical, for $D = 1$ it can be found in the first part of \cite[Lemma~6.4]{HJbook} using similar notation as here, and the argument applies verbatim to arbitrary $D$. 
\end{proof}
Given two metric spaces $\mathcal{X}$ and $\mathcal{Y}$, we say that a sequence $(g_n)_{n\in\N}$ of functions $g_n:\mathcal{X}\to\mathcal{Y}$ converges \emph{locally uniformly} to some $g:\mathcal{X}\to\mathcal{Y}$, if $(g_n)_{n\in\N}$ converges to $g$ uniformly on every metric ball of $\mathcal{X}$.
By combining Propositions~\ref{p.barF_N_Lip} and \ref{p.init.cond}, we can verify that the sequence $\bar F_N$ is precompact. 
\begin{proposition}
[Precompactness of $\Ll(\bar F_N\Rr)_{N\in\N}$]
\label{p.precompact}
For every $p \in (1,+\infty]$, the sequence $(\bar F_N)_{N \in\N}$ is precompact for the topology of local uniform convergence in $\R_+ \times \C_p$. 
\end{proposition}
The proof of this proposition relies on the following lemma, which states that for every $r  < p \in [1,+\infty]$, the embedding $\mcl Q_r \subset \mcl Q_p$ is compact.
\begin{lemma}[Compact $\mcl Q_p$ embeddings]
\label{l.compact.embed}
Let $r \in (1,+\infty]$, and let $\Ll(q^{(n)}\Rr)_{n \in \N}$ be a sequence of paths in $\mcl Q$ such that 
\begin{equation}  
\label{e.compact.embed}
\sup_{n \in \N} \Ll|q^{(n)}\Rr|_{L^r} < +\infty.
\end{equation}
There exist a subsequence $\Ll(q^{(n_k)}\Rr)_{k \in \N}$ from $\Ll(q^{(n)}\Rr)_{n \in \N}$ and a path $q \in \mcl Q_r$ such that for every $p \in [1,+\infty)$ with $p < r$, the subsequence $\Ll(q^{(n_k)}\Rr)_{k \in \N}$ converges almost everywhere and in $L^p$ to $q$. 
\end{lemma}
\begin{proof}
Let $M < +\infty$ denote the supremum in \eqref{e.compact.embed}. Since the path $q^{(n)}$ is increasing, we use Jensen's inequality to write, for every $u < 1$,
\begin{equation*}  \Ll|q^{(n)}(u)\Rr| \le \frac{1}{1-u} \int_u^1 \Ll|q^{(n)}(t)\Rr| \, \d t \le \frac 1 {1-u} \Ll|q^{(n)}\Rr|_{L^1} \le  \frac 1 {1-u} \Ll|q^{(n)}\Rr|_{L^r} \le \frac{M}{1-u}. 
\end{equation*}
By a diagonal extraction argument, we can therefore extract a subsequence from $(q^{(n)})_{n \in \N}$ that converges pointwise at every rational number in $[0,1)$;  for every $u \in [0,1) \cap \Q$, we denote by $q^{(\infty)}(u)$ this limit. For every $u \in [0,1)$, we set
\begin{equation*}  q(u) := \lim_{\substack{t \in \Q \\ t \searrow u}} q^{(\infty)}(u). 
\end{equation*}
Since $q^{(\infty)}$ is increasing, this limit is well-defined. By construction, the path $q$ belongs to $\mcl Q$, and whenever $u < 1$ is a point of continuity of $q$, we must have that $q^{(n)}(u)$ converges to $q(u)$ along the subsequence. In particular, Fatou's lemma ensures that~${q \in \mcl Q_r}$. Moreover, by the dominated convergence theorem, the convergence of $q^{(n)}$ to $q$ along the subsequence is valid in $L^p([0,u]; S^D)$ for every $u < 1$ and $p < +\infty$. So it suffices to verify that, for each $p < r$,
\begin{equation}
\label{e.neglect.last.piece}
\lim_{u \nearrow 1} \sup_{n \in \N} \int_u^1 \Ll|q^{(n)}(t)\Rr|^p \, \d t = 0.
\end{equation}
By Jensen's inequality, we have
\begin{equation*}  \Ll(\frac{1}{1-u} \int_u^1 \Ll|q^{(n)}(t)\Rr|^p \, \d t\Rr)^{1/p} \le \Ll(\frac{1}{1-u} \int_u^1 \Ll|q^{(n)}(t)\Rr|^r \, \d t\Rr)^{1/r} \le \frac{M}{(1-u)^{1/r}}.
\end{equation*}
The right-hand side is understood as $M$ when $r = +\infty$. The inequality above implies~\eqref{e.neglect.last.piece} and therefore completes the proof.  
\end{proof}

\begin{proof}[Proof of Proposition~\ref{p.precompact}]
We fix $p \in (1,+\infty]$. In view of Propositions~\ref{p.barF_N_Lip} and \ref{p.init.cond}, it suffices to show that the intersection of any $L^p$ ball with the set $\mcl Q$ is compact in $L^1$, by the Arzel\`a-Ascoli theorem.  This follows from Lemma~\ref{l.compact.embed}. 
\end{proof}
Proposition~\ref{p.precompact} implies that if there exists a subsequence along which $\bar F_N$ converges pointwise to some limit, then the convergence in fact holds locally uniformly in $\R_+ \times \mcl Q_p$ for every $p > 1$. 

Our next result records a monotonicity property of the mapping $q \mapsto \bar F_N(t,q)$. One way to think of this result is to say that the derivative $\partial_q \bar F_N(t,q,\cdot)$ always belongs to $\C_2$. Since we would like to use common language also concerning possible limits of $f$, which are not necessarily differentiable everywhere, we prefer to phrase it as a monotonicity property per se, using the notion of dual cones. We denote by $\C_2^*$ the cone dual to $\C_2$, which is defined by
\begin{align*}\C_2^* := \Ll\{\kappa \in\cH \ \mid \ \forall q \in \C_2, \ \la \kappa, q \ra_\cH\geq 0\Rr\}.
\end{align*}
The following result can be found in \cite[Lemma~3.4~(2)]{chen2022hamilton}. We also give a proof here for the reader's convenience.
\begin{lemma}
\label{l.charact.C2star}
We have 
\begin{equation}  
\label{e.charact.Cstar}
\C_2^* = \left\{\kappa\in \cH \ \mid \ \forall t \in [0,1), \ \int_t^1\kappa(s)\d s\in\S^\D_+\right\}.
\end{equation}
\end{lemma}
\begin{proof}
We start by recalling, e.g.\ from \cite[Lemma~2.2]{mourrat2020hamilton}, that for every $a \in S^D$, we have the equivalence
\begin{equation}
\label{e.charact.SD+}
a \in S^D_+ \quad \iff \quad \forall b \in S^D_+, \ a \cdot b \ge 0. 
\end{equation}
(In other words, the cone $S^D_+$ is self-dual.)
We denote by $K$ the set on the right side of~\eqref{e.charact.Cstar}. Let $\kappa \in K$ and $q \in \C_2$ be a continuously differentiable path, whose derivative we denote by $\dot q$. Integrating by parts, we see that
\begin{align}  
\label{e.calc.charact.C2star}
\la \kappa,q \ra_\cH 
& = \int_0^1 \kappa(u) \cdot q(u) \,\d u 
\\
\notag
& = q(0) \cdot \int_0^1 \kappa(u) \, \d u + \int_0^1 \kappa(u) \cdot \int_0^u \dot q(t) \, \d t \, \d u
\\
\notag
& = q(0) \cdot \int_0^1 \kappa(u) \, \d u + \int_0^1 \dot q(t) \cdot \int_t^1 \kappa(u) \, \d u \, \d t.
\end{align}
Using that $q$ is increasing and \eqref{e.charact.SD+}, we see that $\dot q$ must take values in $S^D_+$. Since we also have $q(0) \in S^D_+$, we conclude that $\la \kappa,q \ra_\cH \ge 0$. Arguing by approximation, this result can be extended to every $q \in \C_2$. We have thus shown that $K \subset \C_2^*$. 

Conversely, let $\kappa \in \C_2^*$. For every continuous path $\dot q \in C([0,1];S^D_+)$, the path defined for every $u \in [0,1]$ by
\begin{equation*}  q(u) := \int_0^u \dot q(t) \, \d t
\end{equation*}
belongs to $\C_2$, and the notation $\dot q$ is consistent. We conclude from the calculation in~\eqref{e.calc.charact.C2star} that
\begin{equation*}  \int_0^1 \dot q(t) \cdot \int_t^1 \kappa(u) \, \d u \, \d t \ge 0.
\end{equation*}
Using that this property holds for every choice of $\dot q \in C([0,1];S^D_+)$, that the mapping $t \mapsto \int_t^1 \kappa(u) \, \d u$ is continuous, and the characterization in \eqref{e.charact.SD+}, we conclude that $\kappa \in K$, thereby finishing the proof.
\end{proof}
For every subset $G$ of $\cH$ and function $g:G\to\R$, we say that $g$ is  \emph{$\C_2^*$-increasing} if for every $q, q' \in  G$, we have
\begin{align}\label{e.C_2^*-nondec}
 {q}-{q'}\in \C^*_2\quad \implies \quad g({q}) \geq g({q'}).
\end{align}
The following result is a rephrasing of \cite[Proposition~3.8]{mourrat2023free}. 
\begin{proposition}
\label{p.monotone}
For every $t\geq 0$ and $N\in\N$, the mapping $\bar F_N(t,\cdot)$ is $\C_2^*$-increasing. 
\end{proposition}

It will be very useful for our purposes to notice that the free energy $\bar F_N$ is locally semi-concave as a function of $q$. We first obtain a local semi-concavity property for regularly-spaced piecewise-constant paths, and then also derive a version that concerns sufficiently regular continuous paths. We recall that, for every $a \in S^D_{+}$, we denote by $\ellipt(a)$ the ratio between the largest and the smallest eigenvalues of $a$. 

\begin{proposition}[Semi-concavity of $\bar F_N$ for piecewise-constant paths]
\label{p.semiconc.disc}
There exists a constant $C < +\infty$ depending only on $D$ and $\xi$ such that the following holds. Let $(\zeta_k)_{1 \le k \le K}$ be as in~\eqref{e.def.zeta}, let $t,t' \ge 0$, $0 = q_{-1} \le q_0 < \cdots < q_K\in S^D_+$, and $0 = q'_{-1} \le q'_0 < \cdots < q'_K\in S^D_+$, and set
\begin{equation}  
\label{e.def.disc.q.q'}
q := \sum_{k = 0}^K q_k \1_{\Ll[\zeta_k, \zeta_{k+1} \Rr)} \quad \text{ and } \quad q' := \sum_{k = 0}^K q_k' \1_{\Ll[\zeta_k, \zeta_{k+1} \Rr)}.
\end{equation}
Letting $c \in (0,1]$ be such that $t,t' \ge c$ and 
\begin{equation}  
\label{e.semiconc.cond.q}
\forall k \in \{0, \ldots, K\}, \qquad \frac{c}{K+1} \id \le q_{k} - q_{k-1}  \quad \text{ and } \quad \ellipt(q_k - q_{k-1}) \le c^{-1},
\end{equation}
and similarly with $q$ replaced by $q'$ in \eqref{e.semiconc.cond.q}, we have, for every  $\lambda \in [0,1]$,
\begin{multline}  
\label{e.semiconc.disc}
(1-\lambda) \bar F_N(t,q) + \lambda \bar F_N(t',q') - \bar F_N \big( (1-\lambda)(t,q) + \lambda(t',q') \big) 
\\
\le C\lambda (1-\lambda) c^{-2} \Ll(  (t-t')^2 
+  (K+1)\sum_{k=0}^K\Ll|\Ll(q_k-q_{k-1}\Rr)-\Ll(q'_k-q'_{k-1}\Rr)\Rr|^2\Rr).
\end{multline}
\end{proposition}
\begin{proof}
We start with a preliminary step. 

\medskip

\noindent \emph{Step~0}.
Recall that we denote by $S^D_{++}$ the set of $D$-by-$D$ positive definite matrices. For every $h \in S^D_{++}$ and  $a\in\S^\D$, we set
\begin{align*}
    \mathsf{D}_{\sqrt{h}}(a) := \lim_{\eps\to 0}\eps^{-1}\Ll(\sqrt{h+\eps a}-\sqrt{h}\Rr),
\end{align*}
as well as
\begin{align*}
    \mathsf{D}^2_{\sqrt{h}}(a) = \lim_{\eps\to  0} \ep^{-1} \Ll(\mathsf{D}_{\sqrt{h+\eps a}}(a) - \mathsf{D}_{\sqrt{h}}(a)\Rr).
\end{align*}
We denote by $\lambdamin(h)$ the smallest eigenvalue of $h$.
In this step, we show that 
\begin{align}
\label{e.|D|<...}
    \Ll|\mathsf{D}_{\sqrt{h}}(a)\Rr|\leq \frac 1 2 |a|\lambdamin(h)^{-\frac 1 2}
\end{align}
and 
\begin{align}\label{e.D^2<...}
    \Ll|\mathsf{D}^2_{\sqrt{h}}(a)\Rr|\leq \frac 1 4|a|^2 \lambdamin(h)^{-\frac 3 2}.
\end{align}
For two matrices $b,b'\in \R^{D\times D}$, we write $b\circledast b' = \frac{1}{2}( bb' + b'b)$ for the symmetrized matrix multiplication.
Writing $h+\eps a = \Ll(\sqrt{h} +\eps \mathsf{D}_{\sqrt{h}}(a) +o(\eps)\Rr)^2$ and expanding the square, we see that
\begin{align}\label{e.sqrt(h)D+...}
    2\sqrt{h}\circledast \mathsf{D}_{\sqrt{h}}(a)=a
\end{align}

Let $\lambda = \mathrm{diag}(\lambda_1,\dots,\lambda_D)$ and $o$ be an orthogonal matrix such that $h = o \lambda o^\intercal$. We have $\sqrt{h} = o\sqrt{\lambda} o^\intercal$, where $\sqrt{\lambda}=\mathrm{diag}(\sqrt{\lambda_1},\dots,\sqrt{\lambda_D})$. 
Writing $D' := o^\intercal \mathsf{D}_{\sqrt{h}}(a)  o $ and $a' := o^\intercal a o$, we get from \eqref{e.sqrt(h)D+...} that for every $d,d' \in \{1,\ldots D\}$,
\begin{equation*}  D'_{d,d'}\Ll(\sqrt{\lambda_d}+\sqrt{\lambda_{d'}}\Rr)=a'_{d,d'}.
\end{equation*}
This implies that
\begin{equation*}  |\mathsf{D}_{\sqrt{h}}(a)| = |D'|\leq \frac 1 2 |a'| \ {\max_{d \le D} \lambda_d^{-1/2}}= \frac 1 2 |a| \lambdamin(h)^{-\frac{1}{2}},
\end{equation*}
as announced in \eqref{e.|D|<...}.

Turning to \eqref{e.D^2<...}, we start with~\eqref{e.sqrt(h)D+...} applied to $h+\eps a$:
\begin{align*}
    2\sqrt{h+\eps a }\circledast \mathsf{D}_{\sqrt{h+\eps a}}(a) =a.
\end{align*}
By the definitions of the derivatives, this means that, as $\eps$ tends to zero,
\begin{align*}
    2\Ll(\sqrt{h}+\eps \mathsf{D}_{\sqrt{h}}(a)+o(\eps)\Rr)\circledast \Ll(\mathsf{D}_{\sqrt{h}}(a)+\eps \mathsf{D}^2_{\sqrt{h}}(a)+o(\eps)\Rr) = a
\end{align*}
Using~\eqref{e.sqrt(h)D+...} to cancel out terms and sending $\eps$ to zero, we get
\begin{align*}
    \sqrt{h}\circledast \mathsf{D}^2_{\sqrt{h}}(a) = -  \Ll(\mathsf{D}_{\sqrt{h}}(a)\Rr)^2.
\end{align*}
Using the same diagonalization argument as for \eqref{e.|D|<...} yields that 
\begin{align}
\label{e.D2squareroot}
    \Ll|\mathsf{D}^2_{\sqrt{h}}(a)\Rr|\leq  \Ll|\Ll(\mathsf{D}_{\sqrt{h}}(a)\Rr)^2\Rr| \lambdamin(h)^{-\frac{1}{2}}.
\end{align}
Observing also that, for every matrix $b \in S^D$ with eigenvalues $\lambda_1,\ldots, \lambda_D$, 
\begin{equation*}  \Ll|b^2\Rr|^2 =  \sum_{d = 1}^D\lambda_d^4 \le \Ll(\max_{d \le D} \lambda_d^2 \Rr)\, \sum_{d = 1}^D\lambda_d^2 \le \Ll(\sum_{d = 1}^D \lambda_d^2\Rr)^2 = |b|^4,
\end{equation*}
we obtain \eqref{e.D^2<...} from \eqref{e.|D|<...} and \eqref{e.D2squareroot}.

Now that this preliminary step is complete, we turn to the proof of Proposition~\ref{p.semiconc.disc} per se, which will be decomposed into two steps. For notational clarity we will only show the semi-concavity property in \eqref{e.semiconc.disc} in the special case of $t = t'$, but it will be clear from the proof that the argument also covers the general case. To lighten notation, we therefore fix $t \ge 0$ from now on.

\medskip

\noindent \emph{Step~1}.
For every $s\geq0$ and $y = (y_0,\ldots, y_K) \in (\S^\D_+)^{K+1}$, we define
\begin{align*}
    \tilde H_N^{s,y}(\sigma,\alpha) & =\sqrt{2}sH_N(\sigma) - s^2N\xi\Ll(\frac{\sigma\sigma^\intercal}{N}\Rr) +  \sqrt{2}\sum_{k=0}^Ky_k {z}_{\alpha_{|k}}\cdot \sigma - \sum_{k=0}^K y_k^2\cdot\sigma\sigma^\intercal,
    \\
    G_N(s,y)  & = -\frac{1}{N}\E\log \int \sum_{\alpha \in \N^K}  \exp\Ll(\tilde H_N^{s,y}(\sigma,\alpha)\Rr)\, v_\alpha \, \d P_N(\sigma).
\end{align*}
Notice that, with notation as in \eqref{e.def.q}-\eqref{e.def.piece.const.q}, we have
\begin{equation}
\label{e.turn.F.to.G}
\bar F_N(t,{q}) = G_N(\sqrt{t},(\sqrt{{q}_k-{q}_{k-1}})_{0\leq k\leq K}).
\end{equation}
As already mentioned, we will keep $t$ fixed in this proof for notational clarity, but the required modifications for the general case should be clear. We therefore fix $s = \sqrt{t}$ from now on and suppress it from the notation, simply writing $G_N(y)$ in place of $G_N(\sqrt{t}, y)$.

In this step, we show that the Hessian of $G_N$ is uniformly bounded from above. With the notation $y  = (y_0,\ldots, y_K) \in (S^D_+)^{K+1}$, we denote by $\partial_{y_k} G_N(y)$ the derivative with respect to the $k$-th coordinate of $y$. This derivative is an element of $S^D$. In this proof, it is at times convenient to think of each coordinate $y_k$ of $y$ simply as a $\frac{D(D+1)}{2}$-dimensional vector, so that we can view $\partial_{y_k} \partial_{y_{l}} G_N(y)$ as a $\frac{D(D+1)}{2}$-by-$\frac{D(D+1)}{2}$ matrix, and we can for instance write the Taylor expansion, as $y' =(y'_0,\ldots, y_K') \in (S^D_+)^{K+1}$ tends to $y \in (S^D_+)^{K+1}$,
\begin{multline*}  G_N(y') = G_N(y)  + \sum_{k = 0}^K (y_k'-y_k)\cdot \partial_{y_k} G_N(y) \\
+ \frac 1 2 \sum_{k,l = 0}^K (y'_k - y_k) \cdot \partial_{y_k} \partial_{y_l} G_N(y) (y'_l-y_l) + o(|y-y'|^2).
\end{multline*}
Notice that our definitions of scalar products are consistent regardless of whether we regard $y_k$ as a matrix or as a $\frac{D(D+1)}{2}$-dimensional vector. We also abbreviate the expansion above as
\begin{equation*}  G_N(y') = G_N(y)  + (y'-y)\cdot \nabla G_N(y) 
+ \frac 1 2 (y'-y) \cdot \nabla^2 G_N(y) (y'-y) + o(|y-y'|^2).
\end{equation*}
With this notation in place, the goal of this step is to show that, for every $k \in \{0,\ldots, K\}$, we have
\begin{align}\label{e.|nablaG_N|<}
    \Ll|\partial_{y_k} G_N(y)\Rr|\leq 2|y_k|,
\end{align}
and that for every $a \in (S^D)^{K+1}$,
\begin{align}\label{e.zGz<2}
    a\cdot  \nabla^2 G_N(y)a\leq 2|a|^2.
\end{align}
We start with the first-order derivative of $G_N$. For every $k\in\{0,\dots,K\}$ and $a \in \S^\D_+$, using the Gaussian integration by parts in \cite[Theorem~4.6]{HJbook}, we can compute 
\begin{align*}
    a\cdot \partial_{y_k} G_N(y) 
    & = \frac{\d}{\d\eps}G_N(y_1,\dots, y_{k-1},y_k+\eps a,y_{k+1},\dots,y_K)\Big|_{\eps =0}
    \\
    & = -\frac{1}{N} \E \la \sqrt{2} az_{\alpha_{|k}}\cdot\sigma - 2 ay_k \cdot \sigma\sigma^\intercal \ra = \frac{2}{N}\E \la ay_k \cdot \sigma\sigma'^\intercal \1_{\{\alpha_{|k} = \alpha'_{|k}\}}\ra.
\end{align*}
By the Cauchy--Schwarz inequality and \eqref{e.bound.on.sigma}, we have 
\begin{align}  
\label{e.bound.on.sisi'}
|\sigma \sigma'^\intercal|^2  
 = \sigma \sigma'^\intercal \cdot \sigma \sigma'^\intercal
 = \sigma^\intercal \sigma \cdot \sigma'^\intercal \sigma'
 \le |\sigma^\intercal \sigma| \, |\sigma'^\intercal \sigma'| \le N^2.
\end{align}
Combining the last two displays yields that
\begin{equation*}  \Ll| a\cdot \partial_{y_k} G_N(y) \Rr| \le 2 |a| \, |y_k|,
\end{equation*}
which is \eqref{e.|nablaG_N|<}. 

Turning to the proof of \eqref{e.zGz<2}, we first observe that, for every $a,b \in \R^{D\times D}$ and $\lambda \in \R$,
\begin{align*}
    (1-\lambda)a^2 + \lambda b^2 = ((1-\lambda)a+\lambda b)^2 + \lambda(1-\lambda)(a-b)^2.
\end{align*}
Using this and H\"older's inequality, we obtain that, for every $y,y' \in (S^D)^{K+1}$ and $\lambda \in [0,1]$,
\begin{align*}
    &(1-\lambda) G_N(y)+\lambda G_N(y') = -\frac{1}{N}\E\log (\cdots)^{1-\lambda}(\cdots)^\lambda 
    \\
    &\leq -\frac{1}{N}\E \log \int \sum_{\alpha \in \N^K}  \exp\Ll((1-\lambda)\tilde H_N^{s,y}(\sigma,\alpha) +\lambda \tilde H_N^{s,y'}(\sigma,\alpha)\Rr)\, v_\alpha \, \d P_N(\sigma)
    \\
    &= -\frac{1}{N}\E \log \int\sum_{\alpha \in \N^K}  \exp\Ll(\tilde H_N^{s,(1-\lambda)y+\lambda y'}(\sigma,\alpha)+ r \Rr)\, v_\alpha \, \d P_N(\sigma)
\end{align*}
where
\begin{align*}
    r := - \lambda(1-\lambda) \sum_{k=0}^K\Ll(y_{k}-y_{k}'\Rr)^2\cdot \sigma\sigma^\intercal  \ge - \lambda(1-\lambda) N\Ll|y-y'\Rr|^2,
   \end{align*}
using also \eqref{e.bound.on.sigma} in the last inequality. 
We thus deduce that
\begin{align*}
    (1-\lambda) G_N\Ll(y\Rr)+\lambda G_N\Ll(y'\Rr) \leq G_N\Ll((1-\lambda)y+\lambda y'\Rr) + \lambda(1-\lambda)\Ll|y-y'\Rr|^2.
\end{align*}
Substituting $y+\eps a$ for $y'$ in the above display, setting $\lambda$ to be $\frac{1}{2}$, and then sending $\eps$ to zero, we obtain \eqref{e.zGz<2}.

\medskip

\noindent \emph{Step~2}.
In this step, we change variables and use the results of the previous step to bound the second derivatives in the $q_k$'s of $\bar F_N$.
For every $x\in (\S^\D_{+})^{K+1}$, we write $\sqrt{x} := \Ll(\sqrt{x_k}\Rr)_{1\leq k\leq K}$ and set $\tilde G_N(x) := G_N\Ll(\sqrt{x}\Rr)$. For every $x \in (S^D_{++})^{K+1}$ and $a \in (S^D)^{K+1}$, we have
\begin{align*}  \frac{\d}{\d \ep} \tilde G_N(x+\ep a) = \sum_{k = 0}^K D_{\sqrt{x_k + \ep a_k}}(a_k) \cdot \partial_{y_k} G_N(\sqrt{x + \ep a}),
\end{align*}
and thus
\begin{align*}
    a \cdot  \nabla^2 \tilde G_N(x) a 
    & = \frac{\d^2}{\d \eps^2} \tilde G_N(x+\eps a)\Big|_{\eps =0}
    \\
    &= \sum_{k,l=0}^K \mathsf{D}_{\sqrt{x_k}}(a_k)\cdot \partial_{y_k} \partial_{y_l} G_N\Ll(\sqrt{x}\Rr) \mathsf{D}_{\sqrt{x_l}}(a_l)  + \sum_{k=0}^K \mathsf{D}_{\sqrt{x_k}}^2(a_k)\cdot \partial_{y_k} G_N\Ll(\sqrt{x}\Rr).
\end{align*}
Applying the estimates~\eqref{e.|nablaG_N|<} and ~\eqref{e.zGz<2}, and then \eqref{e.|D|<...} and~\eqref{e.D^2<...}, we get
\begin{align*}
    a\cdot   \nabla^2 \tilde G_N(x) a  
    & \leq 2\sum_{k=0}^K \Ll| \mathsf{D}_{\sqrt{x_k}}(a_k) \Rr|^2+ 2\sum_{k=0}^K \Ll|\mathsf{D}^2_{\sqrt{x_k}}(a_k)\Rr|\Ll|\sqrt{x_k}\Rr|
    \\
    & \le \frac 1 2 \sum_{k = 0}^K |a_k|^2 \lambdamin(x_k)^{-1} + \frac 1 2 \sum_{k = 0}^K |a_k|^2 \lambdamin(x_k)^{-\frac 3 2} |\sqrt{x_k}|.
\end{align*}
Notice that
\begin{equation*}  |\sqrt{x_k}|^2 \le D \, \ellipt(x_k) \lambdamin(x_k).
\end{equation*}
In particular, for every $x \in (S^D_+)^{K+1}$ satisfying
\begin{equation}  
\label{e.constraints.x}
\mbox{for every $k \in \{0,\ldots, K\}$, } \ \frac{c}{K+1} \id \le x_k \quad \text{and} \quad \ellipt(x_k) \le c^{-1},
\end{equation}
we have that, for some constant $C < +\infty$ depending only on $D$, 
\begin{equation*}  a \cdot  \nabla^2 \tilde G_N(x) a  \le C(K+1)c^{-2} |a|^2.
\end{equation*}
For every $x,x' \in (S^D_+)^{K+1}$ with $x$ satisfying \eqref{e.constraints.x} and $x'$ satisfying these same bounds as~$x$ displayed in \eqref{e.constraints.x}, we therefore have that, for every $\lambda \in [0,1]$,
\begin{align*}
    (1-\lambda) \tilde G_N(x) + \lambda \tilde G_N\Ll(x'\Rr) \leq \tilde G_N\Ll((1-\lambda)x +\lambda x'\Rr) + C \lambda (1-\lambda)  c^{-2} (K+1)|x-x'|^2.
\end{align*}
We now recall from \eqref{e.turn.F.to.G} that for $0 = q_{-1} < q_0 < \cdots < q_K$, and for $q$ as in \eqref{e.def.disc.q.q'}, we have that $\bar F_N(t,q) = \tilde G_N\big((q_k - q_{k-1})_{0 \le k \le K}\big)$. For $q,q'$ as in the assumptions of the theorem, a change of variables in the previous display therefore yields the desired result.
\end{proof}

We now extend this semi-concavity result to the case of continuous paths. For every $c > 0$, we write
\begin{multline}
\label{e.def.C_c}
    \C_{\uparrow,c} :=\big \{{q}\in\C_2 \ \mid \ q(0) = 0 \ \text{ and } \ \forall u \le v \in [0,1), \quad 
    q(v) - q(u) \ge c (v-u) \id \quad 
    \\ \text{and } \quad \ellipt(q(v) - q(u)) \le c^{-1}
    \big\}.
\end{multline}
We observe that $\bigcup_{c>0}\C_{\uparrow,c}=\C_\uparrow$, where we recall that $\C_\uparrow$ was introduced in the sentence containing~\eqref{e.def.Qup}. For every path $q \in \C$, we denote by $\dot q$ its distributional derivative. Whenever this distributional derivative can be represented as an element of $L^2$, we write that $\dot q \in L^2$.

\begin{proposition}[semi-concavity of $\bar F_N$ for continuous paths]
\label{p.semiconc}
There exists a constant $C < +\infty$ such that, for every $N\in\N$, $c>0$, $t,t' \ge c$, $q,q' \in \C_{\uparrow,c}$ with $\dot q- \dot q' \in L^2$, and $\lambda \in [0,1]$,
\begin{multline}
\label{e.semiconc}
    (1-\lambda)\bar F_N(t,{q}) + \lambda \bar F_N(t',q') 
    - \bar F_N \big((1-\lambda)(t,q) + \lambda(t',q')\big)
    \\
    \leq  C\lambda(1-\lambda)c^{-2}\Ll((t-t')^2+\Ll|\dot{q}-\dot{q}'\Rr|^2_\cH\Rr).
\end{multline}
\end{proposition}

\begin{proof}
By a density argument and Proposition~\ref{p.barF_N_Lip}, it suffices to show the result for smooth paths $q, q'$. We therefore choose two paths $q,q'$ in $\C_{\uparrow,c} \cap C^\infty([0,1];S^D)$, and for each $K\in\N$ and $k\in\{0,\dots,K\}$, we set
\begin{equation*}  q_k := q \Ll( \frac{k}{K+1} \Rr) ,
\end{equation*}
as well as the piecewise-constant path $q^K \in \C$ defined by
\begin{equation*}  q^K := \sum_{k = 0}^{K} q_k \1_{ \Ll[ \frac{k}{K+1}, \frac{k+1}{K+1} \Rr)} .
\end{equation*}
We do the same construction for the path $q'$, obtaining a family of parameters $(q'_k)_{0 \le k \le K}$ and a path $q'^K$. As always, we set $q_{-1} = q'_{-1} = 0$. An application of Proposition~\ref{p.semiconc.disc} yields  
\begin{multline}  
\label{e.semiconc.K}
(1-\lambda)\bar F_N\Ll(t,q^K\Rr) + \lambda \bar F_N\Ll(t',q'^K\Rr) - \bar F_N \Ll( (1-\lambda)\Ll(t,q^K\Rr) + \lambda\Ll(t',q'^K\Rr)\Rr)
\\
\le 
C\lambda (1-\lambda) c^{-2} \Ll(  (t-t')^2 
+  (K+1)\sum_{k=0}^K\Ll|\Ll(q_k-q_{k-1}\Rr)-\Ll(q'_k-q'_{k-1}\Rr)\Rr|^2\Rr).
\end{multline}
By Proposition~\ref{p.barF_N_Lip}, the left side of this inequality converges to the left side of \eqref{e.semiconc}. Under our smoothness assumption on $q$, the difference
\begin{equation*}  \dot q \Ll( \frac{k}{K+1} \Rr) - \frac{q_k - q_{k-1} }{K+1} = \dot q \Ll( \frac{k}{K+1} \Rr)  - \frac{1}{K+1} \Ll( q \Ll( \frac{k}{K+1} \Rr) -q \Ll( \frac{k-1}{K+1} \Rr) \Rr) 
\end{equation*}
converges to zero as $K$ tends to infinity, uniformly over $k \in \{0,\ldots, K\}$. Since the same goes for the path $q'$, we can replace the second term in the large parentheses on the right side of \eqref{e.semiconc.K} by
\begin{equation*}  \frac{1}{K+1} \sum_{k = 0}^K \Ll| \dot q \Ll( \frac{k}{K+1} \Rr) - \dot q' \Ll( \frac{k}{K+1} \Rr) \Rr|^2 ,
\end{equation*}
up to an error that tends to zero as $K$ tends to infinity. This is a Riemann-sum approximation of the integral $|\dot q - \dot q'|_{L^2}^2$, so under our smoothness assumption on $q$ and $q'$, it converges to this integral as $K$ tends to infinity. Recalling \eqref{e.semiconc.K}, we thus obtain the desired result. 
\end{proof}

In short, we think of Proposition~\ref{p.semiconc} as stating that the function $q \mapsto \bar F_N(t,q)$ is locally semi-concave with respect to the quadratic form $q \mapsto |\dot q|_{L^2}^2$. One of the useful consequences of this fact for our purposes is that if $\bar F_N$ converges to some limit $f$, and if $q \in \C_\uparrow$ is a point of differentiability of $f(t,\cdot)$, then the $q$-derivative of $\bar F_N$ at $(t,q)$ must converge to the $q$-derivative of $f$ at $(t,q)$. We postpone a precise statement to this effect to Section~\ref{s.def.cont}.

\section{Continuous cascades and overlap approximations}
\label{s.cascades}

Recall that in Section~\ref{s.def.disc}, we defined the free energy $\bar F_N(t,q)$ for piecewise-constant paths $q$, and then extended this function to all $q \in \C_1$ by continuity. In this section, we lay the groundwork that will allow us to give a direct definition of $\bar F_N(t,q)$ for arbitrary paths $q$. We also show some useful properties of the objects involved, and how to identify them with the discrete objects introduced in Section~\ref{s.def.disc} when the path $q$ is piecewise constant. One of the advantages of the direct construction is that it allows us to also directly manipulate the associated Gibbs measure for continuous paths $q$. 

The basic object in this construction is a Poisson--Dirichlet cascade whose overlap is uniformly distributed over the interval $[0,1]$. 
For the reader's convenience, we briefly review its construction here. Recall from the previous section that for each choice of the parameters $(\zeta_{k})_{1 \le k \le K}$ as in \eqref{e.def.zeta}, we can build the associated Poisson--Dirichlet cascade $(v_\alpha)_{\alpha \in \N^K}$, which encodes the weights of a random probability measure on $\N^K$. For each realization of $(v_\alpha)_{\alpha \in \N^K}$, we construct independent random variables $(\alpha^\ell)_{\ell \ge 1}$ taking values in $\N^K$ that are sampled according to these weights. Choosing also real numbers $(q_k)_{0 \le k \le K}$ satisfying~\eqref{e.def.q}, we then build the ``overlap'' array $(q_{\alpha^\ell \wedge \alpha^{\ell'}})_{\ell, \ell' \ge 1}$, where we recall the notation $\wedge$ from \eqref{e.def.wedge}. We have thus defined a mapping taking any piecewise-constant path~$q \in L^\infty([0,1];\R)$ of the form in \eqref{e.def.piece.const.q} as input, and returning a random array indexed by $\N\times \N$ (and in truth it is rather the law of this random array that is unambiguously well-defined). If we take any sequence of piecewise-constant paths $q^{(n)}$ that converges in $L^\infty([0,1];\R)$ to the identity map, then the sequence of resulting overlap arrays converges in law in the sense of finite-dimensional distributions, by \cite[Corollary~5.32]{HJbook}. The law of the limit does not depend on the approximating sequence, and we denote it by~$R = (R_{\ell,\ell'})_{\ell,\ell' \ge 1}$. By the Dovbysh--Sudakov theorem in \cite[Theorem~1.7]{pan}, this limit array can be represented as a ``true'' overlap array. That is, there exists a separable Hilbert space $\mathfrak{H}$, whose scalar product we denote by $\wedge$, and a random probability measure~$\fR$ on the unit sphere of $\mathfrak{H}$, such that if $(\alpha^\ell)_{\ell \ge 1}$ denote independent random variables sampled according to $\fR$, then the overlap array $(\alpha^\ell \wedge \alpha^{\ell'})_{\ell, \ell' \ge 1}$ has the same law as the limit array~$R$ constructed above. In this statement, the law of $(\alpha^\ell \wedge \alpha^{\ell'})_{\ell, \ell' \ge 1}$ is understood after we average both over the sampling of the random variables $(\alpha^\ell)_{\ell \ge 1}$ and over the randomness inherent to the random probability measure $\fR$ itself. 

We denote by $\langle \cdot \rangle_{\fR}$ the expectation with respect to the sampling of the independent variables $(\alpha^\ell)_{\ell \ge 1}$, defined for each fixed realization of $\fR$, and we denote by $(\bs \Omega, \P)$ the probability space, with associated expectation $\E$, corresponding to the sampling of $\fR$ itself. It will at times be convenient to denote explicitly that $\fR$ is a random variable, and we use the notation $\bomega \mapsto \fR_{\bomega}$ when we want to clarify that $\fR$ is a mapping from $\bs{\Omega}$ to the space of probability measures on $\mfk H$.

By construction, the overlap $\alpha^1 \wedge \alpha^2$ is distributed uniformly over $[0,1]$ under $\P \fR$. (Strictly speaking, we should rather write $\P \, \fR^{\otimes \N}$ or $\P \, \fR^{\otimes 2}$ in place of $\P \fR$, but we allow ourselves this abuse of notation.) In particular, we have $\P \fR$-almost surely that $\alpha^1 \wedge \alpha^2 \in [0,1]$. Equivalently, denoting by $\mfk U$ the (random) support of the measure $\fR$, we have with $\P$-probability one that for every $\alpha^1, \alpha^2 \in \mfk U$,
\begin{equation*}
\al^1 \wedge \al^2 \ge 0.
\end{equation*}
As a direct consequence of the construction of $\fR$ from discrete tree structures, we have with $\P\fR$-probability one that
\begin{equation}  
\label{e.fR.ultrametric}
\alpha^1 \wedge \alpha^3 \ge \min\Ll(\alpha^1 \wedge \alpha^2, \alpha^2 \wedge \alpha^3\Rr). 
\end{equation}
Equivalently, this states that $\P$-almost surely, the relation~\eqref{e.fR.ultrametric} holds for every $\alpha^1, \alpha^2, \alpha^3$ in $\mfk U$. Since~$\mfk U$ is a subset of the unit sphere of $\mfk H$, this property states that the set $\mfk U$ is ultrametric. Explicitly, this means that $\P$-almost surely, we have for every $\alpha^1, \alpha^2, \alpha^3 \in \mfk U$ that
\begin{equation*}
\Ll|\alpha^1 - \alpha^3\Rr|_{\mfk H} \leq \max \Ll( \Ll|\alpha^1 - \alpha^2\Rr|_{\mfk H}, \Ll|\alpha^2 - \alpha^3\Rr|_{\mfk H} \Rr) ,
\end{equation*}
where $|\cdot|_{\mfk H}$ denotes the norm in $\mfk H$.
The set $\mfk U$ is the continuous analogue of the leaves $\N^K$ of the tree $\mcl A$ that appeared in the construction of the discrete cascades.

To summarize, there exists a measurable subset $\bs \Omega^*$ of $\bs \Omega$ with $\P(\bs \Omega^*) = 1$ such that for every $\bomega \in \bs \Omega^*$, we have that 
\begin{equation}
\label{e.good.mfkU}
\begin{cases}
\mbox{$\mfk U_\bomega$ is a subset of the unit sphere in $\mfk H$},\\
\forall \al^1, \al^2 \in \mfk U_\bomega, \quad \al^1 \wedge \al^2 \ge 0,\\
\forall \al^1, \al^2, \al^3 \in \mfk U_\bomega, \quad \alpha^1 \wedge \alpha^3 \ge \min(\alpha^1 \wedge \alpha^2, \alpha^2 \wedge \alpha^3). 
\end{cases}
\end{equation}

Recall that in the construction of the free energy for piecewise-constant paths, we used the random Gaussian field in \eqref{e.Z(h_alpha)=}. Once such a Gaussian field is constructed for $N = 1$, it is easy to extend the construction to general $N$ by using independent copies. Our next goal is to construct an analogous object in the setting of continuous cascades. We recall that we write $\mcl Q_\infty = \mcl Q \cap L^\infty$. By monotonicity, for every $q \in \mcl Q_\infty$, the limit
\begin{equation}  
\label{e.def.q.1}
q(1) := \lim_{u \nearrow 1} q(u) \in S^D_+
\end{equation}
is well-defined. We always extend a path $q \in \mcl Q_\infty$ by continuity at $1$ according to \eqref{e.def.q.1}. With this convention, every path $q \in \mcl Q_\infty$ is continuous at $1$.
\begin{proposition}[Construction of Gaussian cascade]
\label{p.gaussian_cascade}
We fix $\mfk U = \mfk U_\bomega$ such that \eqref{e.good.mfkU} holds, and let $q \in \mcl Q_\infty$. There exists an $\R^\D$-valued  centered Gaussian process $(w^q(\alpha))_{\alpha \in \mathfrak{U}}$ such that for every $\alpha, \alpha' \in \mfk U$,
\begin{align}
\label{e.gaussian_cascade}
    \E  \Ll[w^q(\alpha) w^q(\alpha')^\intercal \Rr] = q(\alpha\wedge\alpha').
\end{align}
Moreover, if we display the randomness explicitly by writing $\omega \mapsto w^q(\alpha,\omega)$ for $\omega$ varying in the underlying probability space $\Omega$, then we can construct $w^q$ in such a way that the mapping $(\alpha,\omega)\mapsto w^q(\alpha,\omega)$ is jointly measurable.
\end{proposition}
We understand that the set $\mathfrak{U}$ is equipped with the Borel $\sigma$-algebra induced by the norm topology of $\mfk H$.

Since we aim to prove Proposition~\ref{p.gaussian_cascade} by appealing to the Kolmogorov extension theorem, we need to construct such a process indexed by any finite subset of $\mathfrak{U}$ and verify the consistency condition. When $D = 1$, we thus need to show that for every $n \in \N$ and $\alpha^1, \ldots, \alpha^n \in \mfk U$, the matrix $(q(\alpha^i\wedge\alpha^j))_{1\leq i,j\leq n}$ is positive semi-definite. This can be obtained by observing that the matrix $(q(\alpha^i\wedge\alpha^j))_{1\leq i,j\leq n}$ is ultrametric in the sense of \cite[Definition~3.2]{dellacherie2014inverse}, and by quoting \cite[Theorem~3.5]{dellacherie2014inverse} to conclude. 

We can also give a direct proof, which is more in the spirit of the construction of the Gaussian process itself and clearly applies for general $D \in \N$. 
As a first step, we show that any finite ultrametric set can be embedded into the leaves of a tree in a natural way. This result is classical but we give a brief proof for the reader's convenience.

\begin{lemma}\label{l.tree}
Let $\{\alpha^1,\dots,\alpha^n\}$ be an ultrametric subset of the unit sphere in $\mathfrak{H}$. Let $K\ge 0$ be an integer and let real numbers
\begin{align*}
s_0< s_1< \dots< s_K=1
\end{align*}
be such that $\{s_k\}_{k=0}^K = \{\alpha^i\wedge\alpha^j\}_{1\leq i,j\leq n}$. There exists a rooted tree of depth $K$ with leaves $\{\bal^i\}_{1\leq i\leq n}$ such that for every $i,j \in \{1,\ldots, n\}$, we have
\begin{align}\label{e.overlap=tree-overlap}
    \alpha^i\wedge\alpha^j = s_{\bal^i\wedge \bal^j},
\end{align}
where $\bal^i\wedge \bal^j$ denotes the depth of the most recent common ancestor to $\bal^i$ and $\bal^j$.
\end{lemma}
\begin{proof}
We recall that for every $i,j,l \in \{1\ldots, n\}$, we have
\begin{equation}
\label{e.ultrametric}
\alpha^{i} \wedge \alpha^{l} \ge \min(\alpha^i \wedge \alpha^{j}, \alpha^{j} \wedge \alpha^{l}). 
\end{equation}
For each $k \in \{0,\ldots, K\}$ and $i,j \in \{1,\ldots, n\}$, we write 
$i \sim_k j$ if and only if $\alpha^i \wedge \alpha^j \ge s_k$. By~\eqref{e.ultrametric}, the relation $\sim_k$ is an equivalence relation. We define the set of nodes at depth $k$ (of the tree we are going to construct) to be the set of equivalence classes of the relation~$\sim_k$. For $k = K$, the equivalence classes are simply singletons since $s_K = 1$, and we write $\Lambda^i := \{i\}$. By definition, for each $k \ge 1$, an equivalence class for the relation $\sim_k$ must be a subset of exactly one equivalence class for the relation $\sim_{k-1}$. We draw an edge between the equivalence classes at levels $k$ and $k-1$ whenever this occurs. This construction results in a graph of depth $K$ without loops. At level $0$, there is only one equivalence class since $\al^i \wedge \al^j \ge s_0$ for every $i,j \in \{1,\ldots,n\}$. So our graph is connected, i.e.\ it is a tree, and we take this unique equivalence class at level $0$ as the root of the tree. For every $i,j \in \{1,\ldots,n\}$ and $k \in \{0,\ldots, K\}$, if $\alpha^i \wedge \alpha^j = s_k$, then by definition $i$ and $j$ belong to the same equivalence class for the relation $\sim_\ell$ if and only if $\ell \le k$. This implies that $\Lambda^i \wedge \Lambda^j = k$, and thus that the identity~\eqref{e.overlap=tree-overlap} holds.
\end{proof}

\begin{proof}[Proof of Proposition~\ref{p.gaussian_cascade}]
We first show the existence of the Gaussian process $(w^q(\alpha))_{\alpha\in \mathfrak{U}}$. With a view towards applying the Kolmogorov extension theorem, we  verify that for every $n \in \N$ and $\al^1, \ldots, \al^n \in \mfk U$, there exists a centered Gaussian vector $(w(\alpha^i))_{1 \le i \le n}$ with the desired covariance structure. Letting $K \ge 0$ and $s_0 < \cdots < s_K = 1$ be such that $\{s_k\}_{k = 0}^K = \{\al^i \wedge \al^j\}_{1 \le i,j \le n}$, we appeal to Lemma~\ref{l.tree} to find a tree of depth $K$ with leaves $\Lambda^1,\ldots, \Lambda^n$ such that \eqref{e.overlap=tree-overlap} holds. We denote by $\mathscr{A}$ the set of vertices of the tree, and let $(g_\beta)_{\beta\in\mathscr{A}}$ consist of independent standard $\R^\D$-valued Gaussian vectors. For every $i \in \{1,\ldots, n\}$ and $k \in \{0,\ldots, K\}$, we write $\Lambda^i_{|k}$ to denote the ancestor of the leaf $\Lambda^i$ at depth $k$. For each $i \in \{1,\ldots, n\}$, we set
\begin{align*}
    w\Ll(\alpha^i\Rr) :=  \sum_{k=0}^K(q(s_k)-q(s_{k-1}))^{1/2} g_{\bal^i_{|k}},
\end{align*}
with the understanding that $q(s_{-1}) = 0$. 
Using \eqref{e.overlap=tree-overlap}, we see that
\begin{align*}
    \E w\Ll(\alpha^i\Rr)w\Ll(\alpha^j\Rr)^\intercal =  q\Ll(s_{\bal^i\wedge\bal^j}\Rr) =  q\Ll(\alpha^i\wedge\alpha^j\Rr)\quad\forall i,j\in\{1,\dots,n\}.
\end{align*}
This is the desired form for the covariance. It is straightforward to verify the consistency condition in the Kolmogorov extension theorem. Appealing to this theorem (see e.g.\ \cite[Theorem~2.4.3]{tao2011introduction}), we thus conclude for the existence of the Gaussian process $(w^q(\alpha))_{\alpha\in \mathfrak{U}}$.

Next, we verify the measurability. For every $\al, \al' \in \mfk U$, we have
\begin{align}\label{e.cts_at_1}
    \E \Ll|w^q(\alpha) - w^q(\alpha')\Rr|^2 = \tr \Ll(q(1) - q(\alpha\wedge\alpha')\Rr).
\end{align}
Since moreover, $q$ is continuous at $1$ due to \eqref{e.def.q.1},
we deduce that $w^q$ is stochastically continuous; in other words, for every $\alpha\in\mathfrak{U}$ and $\eps>0$, we have that $\P\Ll\{\Ll|w^q(\alpha) - w^q(\alpha')\Rr|>\eps\Rr\}$
converges to $0$ as $\alpha'$ approaches $\alpha$.
We can therefore apply \cite[Theorem~III.3.1]{gikhman2004theory} to find a version of the process that is jointly measurable.
\end{proof}

To summarize and introduce more explicit notation, there exists a measurable subset~$\bs \Omega^*$ of $\bs \Omega$ with $\P(\bs \Omega^*)$ such that \eqref{e.good.mfkU} holds for every $\bomega \in \bs \Omega^*$. For each such $\bomega$ and $q \in \mcl Q_\infty$, we can construct a probability space $(\Omega_{\bomega, q}, \P_{\bomega,q})$ with expectation $\E_{\bomega,q}$, and a centered Gaussian process $(w^q(\alpha))_{\alpha \in \mfk U_\bomega}$ on $(\Omega_{\bomega, q}, \P_{\bomega,q})$, such that for every $\al, \al' \in \mfk U_\bomega$, we have
\begin{equation}
\label{e.summary.gaussian}
\E_{\bomega,q} \Ll[ w^q(\alpha) w^q(\alpha')^\intercal \Rr] =  q(\alpha\wedge\alpha'),
\end{equation}
and the mapping 
\begin{equation*}  \Ll\{
\begin{array}{rcl}  \mfk U_\bom\times \Omega_{\bom, q} & \to & \R^D
\\
(\al, \omega) & \mapsto & w^q(\alpha,\omega)
\end{array}
\Rr.
\end{equation*}
is measurable. The object $w^q$ will serve as the continuous analogue of the process in \eqref{e.Z(h_alpha)=}. 

Notice that we do not assert any measurability property of this Gaussian process with respect to $\bomega$. We do not know of a simple way to improve upon this. Fortunately, as will be explained shortly, this will not be a problem as long as we take $\E_{\bomega,q}$-expectations first, provided that we show that these expectations themselves are measurable with respect to~$\bomega$. This is the content of Lemma~\ref{l.meas} below.

Lest the reader thinks that this is a technical detail that they would rather skip, it is worth noting that the justification of this measurability property will involve ingredients that we need to appeal to later as well. In a nutshell, we will show that the relevant $\E_{\bomega,q}$-expectations can be approximated arbitrarily closely by continuous functions of the overlap arrays. This result will not be of much surprise to specialists, as we will essentially state a general version of \cite[Theorem~1.3]{pan}.

\begin{definition}[Overlap structure]
\label{d.os}
For a measurable subset $S$ of a Euclidean space, we say that a pair $(\Gamma,\mathcal{R})$ is an \emph{$S$-valued overlap structure} if 
\begin{itemize}
    \item $\Gamma$ is a probability measure on some Polish space;
    \item $\mathcal{R}:(\supp\Gamma)\times(\supp\Gamma)\to S$ is measurable and satisfies $|\mathcal{R}|\leq 1$ a.s.
\end{itemize}
For $d \in\N$ and a locally bounded function $C:S\to\R^{d \times d}$, $(\Gamma,\mathcal{R})$ is said to be \emph{$C$-admissible} if there exists an $\R^d$-valued centered Gaussian process $(z(\rho))_{\rho\in\supp\Gamma}$ on some probability space $\Omega_{\Gamma,\mathcal{R}}$ such that
\begin{itemize}
    \item the covariance is given by
    \begin{align*}
        \E_{\Omega_{\Gamma,\mathcal{R}}} \Ll[z(\rho)z(\rho')^\intercal \Rr] = C(\mathcal{R}(\rho,\rho')),\quad\forall \rho,\rho'\in\supp\Gamma;
    \end{align*}
    \item the process $z$ is measurable as a function on $(\supp\Gamma)\times \Omega_{\Gamma,\mathcal{R}}$.
\end{itemize}
\end{definition}

For every $\eps>0$ and $a,b\in\R$, we write
\begin{align}\label{e.approx_eps}
    a \approx_\eps b \quad \Longleftrightarrow \quad |a-b|\leq \eps.
\end{align}
\begin{proposition}[Continuity w.r.t.\ overlap array]
\label{p.approx_os}
We take the setting and notation as in Definition~\ref{d.os}. For every Lipschitz $\bg:\R^d\times \R^{d\times d}\to\R$, every bounded measurable $\bh:S^n\to\R$ with $n\in\N$, and every $\eps>0$, there are $n_1,n_2\in\N$ and bounded continuous functions $F^i_\eps: (\R^{d\times d})^{n_i}\to\R$, with $i\in\{1,2\}$, such that
\begin{align}
    \E_{\Omega_{\Gamma,\mathcal{R}}} \log \int \exp\Ll(\mathbf{g}\Ll(z(\rho),C( \mathcal{R}(\rho,\rho))\Rr)\Rr) \d \Gamma(\rho) 
    &\approx_\eps \la F^1_\eps \Ll(C^{\leq n_1}\Rr) \ra_\Gamma,\label{e.1st_approx_os}
    \\
    \E_{\Omega_{\Gamma,\mathcal{R}}} \la \mathbf{h}\Ll(\mathcal{R}^{\leq n}\Rr) \ra_{\Gamma^\mathbf{g}} &\approx_\eps \la \mathbf{h}\Ll(\mathcal{R}^{\leq n}\Rr)F^2_\eps \Ll(C^{\leq n_2}\Rr) \ra_\Gamma, \label{e.2nd_approx_os}
\end{align}
uniformly over all $C$-admissible $S$-valued overlap structures $(\Gamma,\mathcal{R})$.
Here, 
$\la\cdot\ra_\Gamma$ and $\la\cdot\ra_{\Gamma^\mathbf{g}}$ denote the expectations with respect to $\Gamma^{\otimes \N}$ and $({\Gamma^\mathbf{g}})^{\otimes\N}$ respectively, with canonical random variable $(\rho^\ell)_{\ell \ge 1}$; $\mathcal{R}^{\leq n} := \Ll(\mathcal{R}\Ll(\rho^{\ell},\rho^{\ell'}\Rr)\Rr)_{1\leq \ell,\ell'\leq n}$, $C^{\leq n_i} = \Ll(C\Ll( \mathcal{R}\Ll(\rho^{\ell},\rho^{\ell'}\Rr)\Rr)\Rr)_{1\leq \ell,\ell'\leq n_i}$; and
\begin{align*}
    \d \Gamma^\mathbf{g}(\rho) := \frac{\exp\Ll(\mathbf{g}\Ll(z(\rho),\mathcal{R}(\rho,\rho)\Rr)\Rr) \d \Gamma(\rho)}{\int \exp\Ll(\mathbf{g}\Ll(z(\rho'),\mathcal{R}(\rho',\rho')\Rr)\Rr) \d \Gamma(\rho')}.
\end{align*}  
\end{proposition}

\begin{proof}
For brevity, we write $\E = \E_{\Omega_{\Gamma,\mathcal{R}}}$. The estimates below are uniform in $(\Gamma,\mathcal{R})$ due to $|\mathcal{R}|\leq 1$.
For the first approximation \eqref{e.1st_approx_os}, we can use the same argument as in \cite[Theorem~1.3]{pan}. We sketch the main steps, omitting some details.
For $a> 0$, define $\log_a=\max\{-a, \min\{a,\log\}\}$ and $\mathbf{g}_a$ similarly. Since $\mathbf{g}$ is Lipschitz, we can use a Gaussian concentration inequality such as \cite[Theorem~4.7]{HJbook} to approximate the left-hand side in~\eqref{e.1st_approx_os} by
\begin{align*}
    \E \log_a \la\exp\Ll(\mathbf{g}_a\Ll(z(\rho),C(\mathcal{R}(\rho,\rho))\Rr)\Rr) \ra_\Gamma
\end{align*}
for large $a$.
Noticing that $\log_a$ is continuous and that its argument in the expression above remains in a compact set, we can approximate it by polynomials, so that the expression above becomes a linear combination of terms of the form
\begin{align*}
    \E \la\exp\Ll(\mathbf{g}_a\Ll(z(\rho),C(\mathcal{R}(\rho,\rho))\Rr)\Rr) \ra_\Gamma^j 
    &= \E \la \prod_{\ell=1}^j\exp\Ll(\mathbf{g}_a\Ll(z(\rho^\ell),C(\mathcal{R}(\rho^\ell,\rho^\ell))\Rr)\Rr) \ra_\Gamma,
\end{align*}
for $j\in\N$. We can move $\E$ inside $\la\cdot\ra_\Gamma$ to see that the right-hand side is equal to $ \la F^{1}_{\eps,j}\Ll(C^{\leq j}\Rr)\ra_\Gamma$ for a continuous function $F^1_{\eps,j}$ satisfying
\begin{align*}
    F^{1}_{\eps,j}\Ll(C^{\leq j}\Rr) = \E\prod_{\ell=1}^j\exp\Ll(\mathbf{g}_a\Ll(z(\rho^\ell),C(\mathcal{R}(\rho^\ell,\rho^\ell))\Rr)\Rr).
\end{align*}
This is possible since we can write the right-hand side above explicitly as a Gaussian integral. 
Collecting $F^{1}_{\eps,j}$ back into the linear combination, we obtain $F^1_\eps$.

The second approximation~\eqref{e.2nd_approx_os} is done in a similar way. 
Instead of $\log$, we approximate the function $x\mapsto \frac{1}{x}$. The detail can be seen in \cite[Section~3.4]{pan.spin_dist}, more precisely, the approximation after \cite[(3.25)]{pan.spin_dist}.
One can also see the argument in the second half of \cite[Proof (Theorem~4.2) in Section~4.3]{pan}. Again, we only give a sketch. Writing $V(\rho) := \exp\Ll(\mathbf{g}(z(\rho),C(\mathcal{R}(\rho,\rho)))\Rr)$, we have
\begin{align*}
    \E \la \mathbf{h}\Ll(\mathcal{R}^{\leq n}\Rr) \ra_{\Gamma^\mathbf{g}} = \E \frac{\la\mathbf{h}\Ll(\mathcal{R}^{\leq n}\Rr)\prod_{\ell=1}^n V\Ll(\rho^\ell\Rr)\ra_\Gamma}{\la V\Ll(\rho\Rr) \ra_\Gamma^n}.
\end{align*}
Using the Gaussian concentration, we can approximate $\la V\Ll(\rho\Rr) \ra_\Gamma^{-n}$ by a polynomial of $\la V\Ll(\rho\Rr) \ra_\Gamma$. Hence, the right-hand side in the above is approximated by a linear combination of
\begin{align*}
    \E \la\mathbf{h}\Ll(\mathcal{R}^{\leq n}\Rr)\prod_{\ell=1}^n V\Ll(\rho^\ell\Rr)\ra_\Gamma \la V\Ll(\rho\Rr) \ra_\Gamma^j = \E \la\mathbf{h}\Ll(\mathcal{R}^{\leq n}\Rr)\prod_{\ell=1}^{n+j} V\Ll(\rho^l\Rr)\ra_\Gamma =  \la\mathbf{h}\Ll(\mathcal{R}^{\leq n}\Rr)F^2_{\eps,j}\Ll(C^{\leq n+j}\Rr)\ra_\Gamma
\end{align*}
for some $j\in\N$, where the continuous function $F^2_{\eps,j}$ satisfies
\begin{align*}
    F^2_{\eps,j}\Ll(C^{\leq n+j}\Rr) = \E \prod_{\ell=1}^{n+j} V\Ll(\rho^\ell\Rr).
\end{align*}
Again, collecting $F^2_{\eps,j}$ back into the linear combination, we obtain $F^2_\eps$.
\end{proof}

We now come back to the setting explained in the paragraph containing \eqref{e.summary.gaussian}, and settle the question of measurability of $\E_{\bom, q}$-expectations.

\begin{lemma}[Measurability of expectations]
\label{l.meas}
For every $q\in \mcl Q_\infty$, every Lipschitz $\mathbf{g}:\R^\D\to\R$, and every bounded measurable $\mathbf{h}:\R^{n\times n}\to \R$ with $n\in\N$, the functions
\begin{align}
    \bomega &\mapsto \E_{{\bomega,q}} \log\int \exp\Ll(\mathbf{g}\Ll(w^q(\alpha)\Rr)\Rr) \d \fR_\bomega(\alpha),\label{e.1_bomega_mapsto}
    \\
    \bomega &\mapsto \E_{{\bomega,q}} \la  \mathbf{h}\Ll(\Ll(\alpha^\ell\wedge\alpha^{\ell'}\Rr)_{1\leq \ell,\ell'\leq n}\Rr) \ra_{\fR^\mathbf{g}_\bomega}\label{e.2_bomega_mapsto}
\end{align}
are measurable on $\boldsymbol{\Omega}^\star$, where $\la\cdot\ra_{\fR^\mathbf{g}_\bomega}$ denotes the expectation with respect to the measure $(\fR^\mathbf{g}_\bomega)^{\otimes \N}$ with canonical random variables $(\alpha^\ell)_{\ell \in \N}$, and with
\begin{align}\label{e.dfR^g}
    \d \fR^\mathbf{g}_\bomega (\alpha):= \frac{\exp\Ll(\mathbf{g}\Ll(w^q(\alpha)\Rr)\Rr) \d \fR_\bomega(\al)}{\int \exp\Ll(\mathbf{g}\Ll(w^q(\al')\Rr)\Rr) \d \fR_\bomega(\al')}.
\end{align}
\end{lemma}

\begin{proof}[Proof of Lemma~\ref{l.meas}]
First of all, due to the joint measurability of the Gaussian process, both functions are well-defined. We show that they can be approximated uniformly in $\bomega$ by measurable functions.
For the first function \eqref{e.1_bomega_mapsto}, we can apply Proposition~\ref{p.approx_os} (with $\Gamma=\fR_\bomega$, $R(\alpha,\alpha')=\alpha\wedge\alpha'$, and $C=q$) to get that, for every $\eps>0$, there is an integer $n_1$ and a continuous function $F_\eps^1:(\R^{\D\times\D})^{n_1}\to\R$ such that
the the function \eqref{e.1_bomega_mapsto} is approximated by the function
\begin{align}\label{e.3_bomega_mapsto}
    \bomega\mapsto \la F^1_\eps\Ll(\Ll(q\Ll(\alpha^l\wedge\alpha^{l'}\Rr)\Rr)_{1\leq l,l'\leq n_1}\Rr) \ra_{\fR_\bomega}
\end{align}
uniformly in $\bomega$ up to an error bounded by $\eps$.
Since \eqref{e.3_bomega_mapsto} is measurable, we can conclude the measurability of the function~\eqref{e.1_bomega_mapsto}.
Similarly, for the second function~\eqref{e.2_bomega_mapsto}, we can deduce the measurability using Proposition~\ref{p.approx_os}.
\end{proof}

Henceforth, we omit $\omega$ and $\bomega$ for brevity, and simply write $\P$ and $\E$ to denote the probability and the expectation over both $\bs \Omega$ and $\Omega_{\bomega,q}$, with the understanding that we always integrate over the space $\Omega_{\bomega,q}$ first.

We now show that for piecewise-constant paths, the free energies and Gibbs expectations defined using the discrete construction introduced in the previous section coincide with those defined using the continuous cascade. 
    
\begin{proposition}[Discrete and continuous constructions are equivalent]
\label{p.equiv_cascade}

Let $K \in \N$, let $(q_k)_{0 \le k \le K} \subset S^D_+$ and $(\zeta_k)_{1 \le k \le K} \subset \R_+$ be such that \eqref{e.def.q} and \eqref{e.def.zeta} hold, let $(\nu_\bal)_{\bal\in\N^K}$ be the associated Poisson--Dirichlet cascade as constructed below~\eqref{e.def.wedge}, let $q \in \mcl Q$ be the piecewise-constant path in \eqref{e.def.piece.const.q}, and let $(\msf w^q(\bal))_{\bal\in\N^K}$ and $(w^q(\alpha))_{\alpha \in \mfk U}$ be as in \eqref{e.Z(h_alpha)=} (with $N = 1$) and \eqref{e.summary.gaussian} respectively.
For every Lipschitz $\mathbf{g}:\R^\D\to\R$ and every bounded measurable $\mathbf{h}:\R^{n\times n}\to\R$ with $n\in\N$, it holds that
\begin{gather}\label{e.l.equiv_cascade_1}
    \E \log \sum_{\bal \in \N^K} \nu_\bal \exp\Ll(\mathbf{g}\Ll(\msf w^q(\bal)\Rr)\Rr) = \E \log \int \exp\Ll(\mathbf{g}\Ll(w^{q}(\alpha)\Rr)\Rr) \d \fR(\alpha),
    \\
    \label{e.l.equiv_cascade_2}
    \E  \sum_{\bal^1,\dots,\bal^n \in \N^K} \nu^{\mathbf{g}}_{\bal^1}\cdots \nu^{\mathbf{g}}_{\bal^n} \mathbf{h}\Ll(\Ll(\zeta_{\bal^\ell\wedge\bal^{\ell'}}\Rr)_{1\leq \ell,\ell'\leq n}\Rr) = \E  \la \mathbf{h}\Ll(\Ll(\zeta\Ll(\alpha^\ell\wedge\alpha^{\ell'}\Rr)\Rr)_{1\leq \ell,\ell'\leq n}\Rr)\ra_{\fR^\mathbf{g}}
\end{gather}
where $\fR^\mathbf{g}$ is as defined in \eqref{e.dfR^g}, $\langle \cdot \rangle_{\fR^\mathbf{g}}$ as in Lemma~\ref{l.meas},
and
\begin{align}\label{e.nu^g_bal=}
    \nu^\mathbf{g}_\bal := \frac{ \nu_\bal\exp\Ll(\mathbf{g}\Ll(\msf w^q(\bal)\Rr)\Rr)}{\sum_{\bal' \in \N^K} \nu_{\bal'} \exp\Ll(\mathbf{g}\Ll(\msf w^q(\bal')\Rr)\Rr)}, \qquad\forall \bal \in \N^K.
\end{align}
\end{proposition}
In \eqref{e.l.equiv_cascade_1} and~\eqref{e.l.equiv_cascade_2}, the expectations integrate all the randomness. For the expectations on the right side, we first integrate the Gaussian randomness in $(w^{q}(\alpha))_{\alpha \in \mfk U}$ and then the randomness of $\mathfrak{R}$, which is justified by Lemma~\ref{l.meas}.

\begin{proof}[Proof of Proposition~\ref{p.equiv_cascade}]
We start by observing that, for every $\bal, \bal' \in \N^K$,
\begin{equation*}  \E \Ll[ \msf w^q(\bal) \msf w^q(\bal')^\intercal \Rr]  = q \Ll( \zeta_{\bal \wedge \bal'} \Rr) ,
\end{equation*}
where we recall that $(q(u))_{u \in [0,1)}$ denotes the piecewise-constant path in \eqref{e.def.piece.const.q}. In the language of Definition~\ref{d.os}, the probability measure with weights $(\nu_\bal)_{\bal \in \N^K}$ and the field $(\msf w^q(\bal))_{\bal \in \N^K}$ form a $q$-admissible overlap structure, with the mapping $\mcl R$ chosen to be $(\bal, \bal') \mapsto \zeta_{\bal \wedge \bal'}$.

By \eqref{e.summary.gaussian}, the measure $\mfk R$ and the Gaussian field $(w^q(\al))_{\al \in \mfk U}$ also form a $q$-admissible overlap structure, with $\mcl R$ chosen to be $(\al,\al') \mapsto \al \wedge \al'$. Since $q$ is constant on each interval $[\zeta_k, \zeta_{k+1})$, we can also choose $\mcl R$ to be the mapping $(\al,\al') \mapsto \zeta(\al \wedge \al')$, where we write
\begin{equation}  
\label{e.def.zetapath}
\zeta := \sum_{k = 0}^K \zeta_k \1_{ \Ll[ \zeta_k, \zeta_{k+1} \Rr) }.
\end{equation}
Denoting by $(\bal^\ell)_{\ell \in \N}$ and $(\alpha^\ell)_{\ell \in \N}$ two families of independent random variables sampled according to $(\nu_\bal)_{\bal \in \N^K}$ and $\mfk R$ respectively, we see from Proposition~\ref{p.approx_os} that in order to prove the result, it suffices to argue that for each $n \in \N$, the overlap arrays $(\zeta_{\bal^\ell \wedge \bal^{\ell'}})_{\ell, \ell' \le n}$ and $(\zeta(\al^\ell \wedge \al^{\ell'}))_{\ell, \ell' \le n}$ have the same law after averaging over all sources of randomness. By \cite[Theorem~2.10]{pan} (or   \cite[Theorem~5.28]{HJbook}), every discrete Poisson--Dirichlet cascade satisfies the Ghirlanda--Guerra identities, and these identities are preserved whether we look at the overlap as being the mapping $(\bal, \bal') \mapsto \bal \wedge \bal'$ or the mapping $(\bal, \bal') \mapsto \zeta_{\bal \wedge \bal'}$. We recall that we built the continuous Poisson--Dirichlet cascade as a limit of discrete Poisson--Dirichlet cascades, using e.g.\ \cite[Corollary~5.32]{HJbook}. It therefore follows that the continuous Poisson--Dirichlet cascade $\mfk R$ still satisfies the Ghirlanda--Guerra identities, and in particular the overlap array built from the overlap mapping $(\al,\al') \mapsto \zeta(\al \wedge \al')$ satisfies the Ghirlanda--Guerra identities. By e.g.\ \cite[Proposition~5.31]{HJbook}, the laws of the overlap arrays $(\zeta_{\bal^\ell \wedge \bal^{\ell'}})_{\ell, \ell' \le n}$ and $(\zeta(\al^\ell \wedge \al^{\ell'}))_{\ell, \ell' \le n}$ are uniquely determined by the laws of $\zeta_{\bal \wedge \bal'}$ and $\zeta(\al \wedge \al')$ respectively. By \cite[Theorem~2.10]{pan} (or \cite[Theorem~5.28]{HJbook}), we have that the probability for $\zeta_{\bal \wedge \bal'}$ to be equal to $\zeta_k$ is $\zeta_{k+1} - \zeta_k$, for each $k \in \{0, \ldots, K\}$. By the construction of $\mfk R$ from \cite[Corollary~5.32]{HJbook}, we have that $\al \wedge \al'$ is uniformly distributed over $[0,1]$, and in particular, by \eqref{e.def.zetapath}, the probability for $\zeta(\al \wedge \al')$ to be equal to $\zeta_k$ is equal to $\zeta_{k+1} - \zeta_k$ as well. The proof is therefore complete. 
\end{proof}

To close this section, we show that the law of a Poisson--Dirichlet cascade under $\E \la \cdot \ra_\fR$ remains the same under any probability measure of the form of $\E \la \cdot \ra_{\fR^\mathbf{g}}$, with $\fR^{\mathbf{g}}$ as in \eqref{e.dfR^g} with $\mathbf{g}$ an arbitrary Lipschitz function. In particular, the overlap $\alpha \wedge \alpha'$ under $\E \la \cdot \ra_{\fR^{\mathbf{g}}}$ is uniformly distributed over $[0,1]$, a fact that we will use repeatedly later. 

We start by recalling the following invariance property of discrete Poisson--Dirichlet cascades from \cite[Theorem~4.4]{pan}.

\begin{lemma}[Invariance of discrete cascades]
\label{l.invar_discrete}
Let $K \in \N$, let $\mathcal{A}$ be the rooted tree in \eqref{e.tree},  let $(q_k)_{0 \le k \le K} \subset S^D_+$ and $(\zeta_k)_{1 \le k \le K} \subset \R_+$ be such that \eqref{e.def.q} and \eqref{e.def.zeta} hold, let $(\nu_\bal)_{\bal\in\N^K}$ be the associated Poisson--Dirichlet cascade as constructed below~\eqref{e.def.wedge}, and let $(\msf w^q(\bal))_{\bal\in\N^K}$ be as in \eqref{e.Z(h_alpha)=}.
For every Lipschitz $\mathbf{g}:\R^\D\to\R$, there is a random bijection $\pi:\mathcal{A}\to\mathcal{A}$, preserving the parent-child relation, such that
\begin{align*}
    \Ll(\nu^\mathbf{g}_{\pi(\bal)}\Rr)_{\bal \in\N^K} \stackrel{\d}{=} \Ll(\nu_\bal\Rr)_{\bal\in\N^K}
\end{align*}
where $ \Ll(\nu^\mathbf{g}_{\bal}\Rr)_{\bal \in\N^K}$ is given in \eqref{e.nu^g_bal=}.
In particular, for every bounded measurable function $\mathbf{h}:\R^{n\times n}\to\R$ with $n\in\N$, we have
\begin{align*}
    \E  \sum_{\bal^1,\dots,\bal^n \in \N^K} \nu^{\mathbf{g}}_{\bal^1}\cdots \nu^{\mathbf{g}}_{\bal^n} \mathbf{h}\Ll(\Ll(\zeta_{\bal^\ell\wedge\bal^{\ell'}}\Rr)_{1\leq \ell,\ell'\leq n}\Rr) = \E  \sum_{\bal^1,\dots,\bal^n \in \N^K} \nu_{\bal^1}\cdots \nu_{\bal^n} \mathbf{h}\Ll(\Ll(\zeta_{\bal^\ell\wedge\bal^{\ell'}}\Rr)_{1\leq \ell,\ell'\leq n}\Rr).
\end{align*}
\end{lemma}

We can now prove that this invariance property transfers to continuous cascades.

\begin{proposition}[Invariance of continuous cascades]
\label{p.invar_cts}
For every $q \in \mcl Q_\infty$, every Lipschitz $\mathbf{g}:\R^\D\to\R$, and every bounded measurable $\mathbf{h}:\R^{n\times n}\to\R$ with $n\in\N$, we have
\begin{align*}
    \E  \la \mathbf{h}\Ll(\Ll(\alpha^\ell\wedge\alpha^{\ell'}\Rr)_{1\leq \ell,\ell'\leq n}\Rr)\ra_{\fR^\mathbf{g}}=\E  \la \mathbf{h}\Ll(\Ll(\alpha^\ell\wedge\alpha^{\ell'}\Rr)_{1\leq \ell,\ell'\leq n}\Rr)\ra_{\fR}
\end{align*}
where $\fR^\mathbf{g}$ is defined in~\eqref{e.dfR^g}, and we recall that $\langle \cdot \rangle_\fR = \langle \cdot \rangle_{\fR^0}$. 
\end{proposition}

\begin{proof}
For every $\ell, \ell',n \in \N$, we write $Q^{\ell, \ell'} := \al^\ell \wedge \al^{\ell'}$ and $Q^{\le n} := (Q^{\ell,\ell'})_{1 \le \ell, \ell' \le n}$. 
Without loss of generality, we can assume that $\mathbf{h}$ is bounded and continuous. We fix $\ep > 0$. By Proposition~\ref{p.approx_os}, there exists $n \in \N$ and a bounded continuous function $F_\ep$ such that 
\begin{align}\label{e.cts_inv_eq_1}
    \E \la \mathbf{h}\Ll(Q^{\leq n}\Rr)\ra_{\fR^\mathbf{g}} \approx_\eps \E \la  \mathbf{h}\Ll(Q^{\leq n}\Rr) F_\eps\Ll( q\Ll(Q^{\leq n}\Rr)\Rr)\ra_\fR,
\end{align}
where $q$ acts entrywise. For each $K \in \N$, we write
\begin{equation*}
\zeta^{(K)} := \sum_{k = 0}^K \frac{k}{K+1} \1_{ \Ll[ \frac{k}{K+1}, \frac{k+1}{K+1} \Rr)  },
\end{equation*}
with the usual understanding that $\zeta^{(K)}(1) = \lim_{u \nearrow 1} \zeta^{(K)}(u) = K/(K+1)$. 
In particular, the sequence $(\zeta^{(K)})_{K \in \N} \subset \mcl Q_\infty$ converges in $L^1$ to the identity map $x \mapsto x$ as $K$ tends to infinity. For every $\ell, \ell',n \in \N$, we write $Q_K^{\ell, \ell'} := \zeta^{(K)}(\al^\ell \wedge \al^{\ell'})$ and $Q_K^{\le n} := (Q_K^{\ell, \ell'})_{1 \le \ell, \ell' \le n}$. Since $Q$ satisfies the Ghirlanda--Guerra identities by construction of the continuous Poisson--Dirichlet cascade, it is immediate that $Q_K$ also satisfies the Ghirlanda--Guerra identities. By \cite[Corollary~5.32]{HJbook}, we have that $Q_K^{\le n}$ converges in law to $Q^{\le n}$. We want to argue that $(Q_K^{\le n}, q(Q_K^{\le n}))$ also converges in law to $(Q^{\le n}, q(Q^{\le n}))$ as $K$ tends to infinity. This is not immediate since $q$ is not continuous. Notice first that the on-diagonal overlaps are deterministic, and for every $\ell \in \N$, we have
\begin{equation*}  Q_K^{\ell, \ell} = \frac{K}{K+1} \xrightarrow{K \to \infty} 1 = Q^{\ell,\ell}. 
\end{equation*}
Since $q$ is continuous at $1$, the question thus only concerns the off-diagonal overlaps. We recall that each of these off-diagonal overlaps is uniformly distributed over $[0,1]$. Since the set of points of discontinuity of $q$ is countable, a union bound guarantees that the probability for $Q^{\le n}_K$ to belong to the set of points of discontinuity of the mapping $q$ (as applied entrywise to the array) is zero. 
We can therefore apply the continuous mapping theorem (e.g. \cite[Exercise~A.10 and solution]{HJbook}) to deduce that indeed $(Q_K^{\le n}, q(Q_K^{\le n}))$ converges in law to $(Q^{\le n}, q(Q^{\le n}))$ as $K$ tends to infinity. In particular, for every $K$ sufficiently large, we have that 
\begin{equation*}  \E \la  \mathbf{h}\Ll(Q^{\leq n}\Rr) F_\eps\Ll( q\Ll(Q^{\leq n}\Rr)\Rr)\ra_\fR \approx_\ep \E \la  \mathbf{h}\Ll(Q_K^{\leq n}\Rr) F_\eps\Ll( q\Ll(Q_K^{\leq n}\Rr)\Rr)\ra_\fR.
\end{equation*}
By Proposition~\ref{p.equiv_cascade}, the latter quantity can be represented using the discrete Poisson--Dirichlet cascade based on the path $q \circ \zeta^{(K)}$, which is defined on a tree of depth $K$. Denoting the weights of this discrete Poisson--Dirichlet cascade by $(\nu_{\bal})_{\bal \in \N^K}$, Proposition~\ref{p.equiv_cascade} ensures that
\begin{equation*}  \E \la  \mathbf{h}\Ll(Q_K^{\leq n}\Rr) F_\eps\Ll( q\Ll(Q_K^{\leq n}\Rr)\Rr)\ra_\fR = \E  \sum_{\bal^1,\dots,\bal^n \in \N^K} \nu_{\bal^1}\cdots \nu_{\bal^n} \mathbf{h}\Ll(R^{\le n}\Rr) F_\ep \Ll( q \Ll( R^{\le n} \Rr)  \Rr) ,
\end{equation*}
where we used the shorthand $R^{\ell,\ell'} := (\bal^\ell \wedge \bal^{\ell'})/(K+1)$ and $R^{\le n} := (R^{\ell, \ell'})_{1 \le \ell, \ell' \le n}$. Another application of Proposition~\ref{p.approx_os}, in the same setting of $q$-admissible overlap structures as for \eqref{e.cts_inv_eq_1}, yields that 
\begin{equation*}  \E  \sum_{\bal^1,\dots,\bal^n \in \N^K} \nu_{\bal^1}\cdots \nu_{\bal^n} \mathbf{h}\Ll(R^{\le n}\Rr) F_\ep \Ll( q \Ll( R^{\le n} \Rr)  \Rr) 
\approx_\ep 
\sum_{\bal^1,\dots,\bal^n \in \N^K} \nu^{\mathbf{g}}_{\bal^1}\cdots \nu^{\mathbf{g}}_{\bal^n} \mathbf{h}\Ll(R^{\le n}\Rr).
\end{equation*}
By the invariance property from Lemma~\ref{l.invar_discrete}, we have that
\begin{equation*}  \sum_{\bal^1,\dots,\bal^n \in \N^K} \nu^{\mathbf{g}}_{\bal^1}\cdots \nu^{\mathbf{g}}_{\bal^n} \mathbf{h}\Ll(R^{\le n}\Rr) 
= \sum_{\bal^1,\dots,\bal^n \in \N^K} \nu_{\bal^1}\cdots \nu_{\bal^n} \mathbf{h}\Ll(R^{\le n}\Rr).
\end{equation*}
To summarize, we have shown that
\begin{equation*}  \E \la \mathbf{h}\Ll(Q^{\leq n}\Rr)\ra_{\fR^\mathbf{g}} \approx_{3\ep} \sum_{\bal^1,\dots,\bal^n \in \N^K} \nu_{\bal^1}\cdots \nu_{\bal^n} \mathbf{h}\Ll(R^{\le n}\Rr).
\end{equation*}
Using this relation also with $\mathbf{g} = 0$, we conclude that
\begin{equation*}  \E \la \mathbf{h}\Ll(Q^{\leq n}\Rr)\ra_{\fR^\mathbf{g}} \approx_{6\ep}    \E \la \mathbf{h}\Ll(Q^{\leq n}\Rr)\ra_{\fR} .
\end{equation*}
Since $\ep > 0$ was arbitrary, the proof is complete. 
\end{proof}

\begin{remark}[Left-continuous paths]\label{r.left-cts}
All results in this section also hold if, instead of~$\C_\infty$, we take $q$ from the set
\begin{align}\label{e.Pi=}
    \Pi:= \Ll\{ q : [0,1] \to S^\D_+ \ : \ q \text{ is left-continuous with right limits, and is increasing} \Rr\}. 
\end{align}
The only place where any kind of continuity of a path is needed is in the proof of Proposition~\ref{p.gaussian_cascade}. There, to show that $w^q$ is jointly measurable, one requires $q$ to be continuous at~$1$, as used in~\eqref{e.cts_at_1}.  Every path in $\Pi$ is also continuous at $1$ by definition.
\end{remark}

\section{The free energy with continuous cascades}
\label{s.def.cont}

We recall from the beginning of the previous section that we denote by $\fR$ a Poisson--Dirichlet cascade such that the law of one overlap is uniformly distributed over~$[0,1]$. The object $\fR$ is a (random) probability measure on the Hilbert space $\mfk H$, and we denote its (random) support by $\mfk U$. For each realization of $\mfk R$, we denote by $\la \cdot \ra_{\fR}$ the expectation with respect to the measure $\fR^{\otimes \N}$, with canonical random variables $(\alpha^\ell)_{\ell \in \N}$. The expectation~$\E$ takes the average over the randomness of $\fR$ itself. The event in \eqref{e.good.mfkU} has $\P$-probability one, and whenever this event is realized and for each $q \in \mcl Q_\infty$, there exists an $\R^D$-valued Gaussian process $(w^q(\al))_{\al \in \mfk U}$ with covariance given by \eqref{e.gaussian_cascade}. On this same event, we take $(w^{q}_i)_{i\in\N}$ to be i.i.d.\ copies of $w^{q}$ and define, for every $N\in\N$ and $\alpha \in \mfk U$,
\begin{align}\label{e.W^mu(alpha)}
    W^{q}_N(\alpha) := (w^{q}_1(\alpha),\dots, w^{q}_N(\alpha)).
\end{align}
The random variable $W^q_N(\alpha)$ takes values in $\R^{D\times N}$. We can now define, for each $t \ge 0$, $q \in \mcl Q_\infty$, and $N \in \N$, the free energy
\begin{equation}
\label{e.def.barFN.continuous}
\bar F_N(t,q) := -\frac{1}{N}\E \log \iint \exp \Ll(H_N^{t,{q}}(\sigma,\alpha) \Rr)\d P_N(\sigma)\d \fR(\alpha),
\end{equation}
where
\begin{equation}
\label{e.H^t,mu_N}
H_N^{t,{q}}(\sigma,\alpha)  = \sqrt{2t}H_N(\sigma) - tN\xi\Ll(\frac{\sigma\sigma^\intercal}{N}\Rr)+\sqrt{2}W^{q}_N(\alpha)\cdot \sigma -  {q}(1)\cdot \sigma\sigma^\intercal.
\end{equation}
We also define the associated Gibbs measure
\begin{align}
\label{e.<>_N}
    \la \cdot\ra_N \propto  \exp\Ll( H^{t,q}_N(\sigma,\alpha)\Rr) \d P_N(\sigma) \d \fR(\alpha),
\end{align}
defined for each fixed realization of $\fR$ and of the random Gaussian field $(w^q(\alpha))_{\alpha \in \mfk U}$. We use the symbol $\propto$ as shorthand for more precise expressions as in \eqref{e.def.disc.gibbs}.  The canonical random variable under $\la \cdot \ra_N$ is denoted by $(\si, \al)$. We also write $(\si^\ell, \al^\ell)_{\ell \in \N}$ to denote independent copies of $(\si, \al)$ under $\la\cdot \ra_N$. The expectation $\E$ in \eqref{e.def.barFN.continuous} integrates all sources of randomness, and we recall that for measurability reasons, we always take the average over the Gaussian process $(W_N^q(\alpha))_{\alpha \in \mfk U}$ first.

In \eqref{e.def.barFN.continuous}, we allowed ourselves to use the same notation as in \eqref{e.def.FN.piece.const}. When $q$ is piecewise constant, these two quantities coincide, by Proposition~\ref{p.equiv_cascade}. This proposition also ensures that Gibbs averages computed under the measures in \eqref{e.<>_N} and in \eqref{e.def.disc.gibbs} coincide.

Recall that we extended the free energy defined in \eqref{e.def.FN.piece.const} to all paths $q \in \mcl Q_1$ using the Lipschitz continuity property from Proposition~\ref{p.barF_N_Lip}. The next proposition states that the right-hand side of \eqref{e.def.barFN.continuous} satisfies the same Lipschitz continuity property, and therefore the two constructions do coincide for all paths and there is no abuse of notation here. This proposition in fact identifies the derivatives of $\bar F_N$ with respect to $t$ and $q$ precisely. 
We set
\begin{align}\label{e.L^infty<1}
    L^\infty_{\leq 1} := \Ll\{q \in L^\infty :\: |\kappa|_{L^\infty}\leq 1\Rr\}.
\end{align}
\begin{proposition}[Differentiability of $\bar F_N$]
\label{p.F_N_smooth}
Let $N \in \N$, and let $\bar F_N : \R_+ \times \mcl Q_\infty \to \R$ be the free energy defined in \eqref{e.def.barFN.continuous}. We have for every $t,t' \in \R_+$ and $q,q' \in \mcl Q_\infty$ that 
\begin{equation*}
\Ll|\bar F_N(t,{q}) - \bar F_N(t',{q'})\Rr|\leq \Ll|{q}-{q'}\Rr|_{L^1} + |t-t'| \, \sup_{|a| \le 1} |\xi(a)|.
\end{equation*}
In particular, the free energy in \eqref{e.def.barFN.continuous} can be extended by continuity to $\R_+\times \mcl Q_1$, and it coincides everywhere with the free energy defined in \eqref{e.def.FN.piece.const} and extended by continuity to $\R_+\times \mcl Q_1$ in Proposition~\ref{p.barF_N_Lip}. Moreover, the restriction of the function $\bar F_N$ to $\R_+ \times \mcl Q_2$ is Fréchet (and Gateaux) differentiable everywhere, jointly in its two variables. We denote its Fréchet (and Gateaux) derivative in $q$ by $\dr_q \bar F_N(t,q) = \dr_q \bar F_N(t,q, \cdot) \in L^2([0,1]; S^D)$. For every $t \ge 0$, we have, for every $q \in \mcl Q_2$,
\begin{equation}
\label{e.bounds.der.FN}
\partial_q \bar F_N(t,q) \in \mcl Q \cap L^\infty_{\le 1}, \qquad \Ll|\partial_t  \bar F_N(t,q) \Rr| \le \sup_{|a| \le 1} |\xi(a)|,
\end{equation}
and, for every $q \in \mcl Q_\infty$  and $\kappa \in L^2([0,1]; S^D)$,
\begin{equation}
\label{e.def.der.FN}
\la \kappa,\dr_q \bar F_N(t,{q})\ra_\cH = \E \la \kappa \Ll(\alpha\wedge\alpha'\Rr)\cdot \frac{\sigma\sigma'^\intercal}{N}\ra_N,\qquad \partial_t \bar F_N(t,{q}) = \E \la \xi \Ll(\frac{\sigma\sigma'^\intercal}{N}\Rr)\ra_N.
\end{equation}
Finally, for every $r \in [1,+\infty]$,  $t \ge 0$ and $q, q' \in \mcl Q_2$ with $q'-q \in L^r$, we have
\begin{equation}
\label{e.continuity.der.FN}
\Ll|\dr_q \bar F_N(t,{q'}) -\dr_q \bar F_N(t,{q})\Rr|_{L^r}\leq 16N\Ll|q'-{q}\Rr|_{L^r}.
\end{equation}
In particular, the mapping $q \mapsto \dr_q \bar F_N(t,q)$ can be extended to $\mcl Q_1$ by continuity, and the properties in \eqref{e.bounds.der.FN} and \eqref{e.continuity.der.FN} remain valid with $q, q' \in \mcl Q_1$. 
\end{proposition}
For the statement of \eqref{e.def.der.FN}, we impose that $q \in \mcl Q_\infty$ so that the Gibbs measure appearing there	is well-defined. 
\begin{proof}[Proof of Proposition~\ref{p.F_N_smooth}]
For clarity of presentation we only prove the results concerning $\dr_q \bar F_N$; showing the joint differentiability in $(t,q)$ does not pose additional difficulties, the property of $\partial_t \bar F_N(t,q)$ in \eqref{e.def.der.FN} is a Gaussian integration by parts as in \eqref{e.der.t}, and the property of $\partial_t \bar F_N(t,q)$ in \eqref{e.bounds.der.FN} follows from that in \eqref{e.def.der.FN} and \eqref{e.bound.on.sisi'}. 

We fix $t \ge 0$, $q, q' \in \mcl Q_\infty$, and write $\kappa := q'-q \in L^\infty([0,1]; S^D)$. 
Conditionally on $\fR$, we construct the Gaussian processes $(w^{q}(\alpha))_{\alpha \in \mfk U}$ and $(w^{\kappa}(\alpha))_{\alpha \in \mfk U}$ from Proposition~\ref{p.gaussian_cascade} independently of one another, and we observe that, for every $s \in [0,1]$,
\begin{align}
\label{e.wq.distr.identity}
    \Ll(w^{{q}+ s  \kappa}(\alpha) \Rr)_{\alpha\in \mfk U}\stackrel{\d}{=} \Ll(w^{{q}}(\alpha) + \sqrt{s} w^{{\kappa}}(\alpha)\Rr)_{\alpha \in \mfk U}.
\end{align}
We may as well take the right-hand side of \eqref{e.wq.distr.identity} as our definition of $w^{q+s\kappa}$, so that these processes are realized on the same probability space and coupled together in a natural way as $s$ varies in $[0,1]$. It follows from \eqref{e.gaussian_cascade} that, for every $\al, \al' \in \mfk U$, 
\begin{align*}
    \E\Ll[ \Ll(\frac{\d}{\d  s }w^{{q}+ s  \kappa}(\alpha) \Rr)\Ll(w^{{q}+ s  \kappa}(\alpha')\Rr)^\intercal \Rr] = \frac{1}{2}\kappa\Ll(\alpha\wedge\alpha'\Rr).
\end{align*}
Let $(w^{{q}+ s  \kappa}_i)_{1\leq i\leq N}$ be i.i.d.\ copies of $w^{{q}+ s  \kappa}$ and define $W^{{q}+ s  \kappa}_N=\Ll(w^{{q}+ s  \kappa}_1,\dots, w^{{q}+ s  \kappa}_N\Rr)$ as in~\eqref{e.W^mu(alpha)}.
By this, the Gaussian integration by parts in \cite[Theorem~4.6]{HJbook}, and the fact that $\alpha\wedge\alpha=1$ almost surely under $\P\fR$, we deduce that
\begin{align*}
    \frac{\d}{\d  s }\bar F_N(t,{q}+ s \kappa) \big|_{ s  =0}
    & = -\frac{1}{N}\E \la \sqrt{2}\Ll(\frac{\d}{\d  s }W^{{q}+ s  \kappa}_N(\alpha)\Rr)\cdot\sigma - \kappa(1)\cdot\sigma\sigma^\intercal\ra_N
    = \E \la R^{1,2}\ra_N,
\end{align*}
where we write 
\begin{equation}
\label{e.def.R.kappa}
R^{\ell,\ell'} := \kappa(\alpha^\ell\wedge\alpha^{\ell'})\cdot \frac{\sigma^\ell(\sigma^{\ell'})^\intercal}{N},
\end{equation}
and we observe that $\kappa(1)$ is well-defined since $q(1)$ and $q'(1)$ are well-defined. Differentiating once more yields that
\begin{align}
\label{e.diff.q.second}
    \frac{\d^2}{\d  s ^2}\bar F_N(t,{q}+ s \kappa)\big|_{ s  =0} = 2N\E\la R^{1,2}R^{1,2}-4R^{1,2}R^{1,3}+3R^{1,2}R^{3,4}\ra_N.
\end{align}
We recall from \eqref{e.bound.on.sisi'} that for every $\sigma, \sigma'$ in the support of $P_N$, we have $\Ll|\frac{\si \si'^\intercal}{N}\Rr| \le 1$. Using also  the invariance of Poisson--Dirichlet cascades in Proposition~\ref{p.invar_cts} and the Cauchy--Schwarz inequality, we obtain that 
\begin{align}
\label{e.|d/dF_N|<}
    \Ll|\frac{\d}{\d  s }\bar F_N(t,{q}+ s \kappa)\big|_{ s  =0}\Rr| &  \leq \E \la \Ll|R^{1,2}\Rr|\ra_N\leq \E \la |\kappa\Ll(\alpha\wedge\alpha'\Rr)| \ra_N = \E \la |\kappa\Ll(\alpha\wedge\alpha'\Rr)| \ra_\fR = |\kappa|_{L^1}, 
\end{align}
as well as
\begin{multline}
\label{e.|d^2F_N|<}
    \Ll|\frac{\d^2}{\d  s ^2}\bar F_N(t,{q}+ s \kappa)\big|_{ s =0}\Rr|  \leq 16 N\E \la \Ll|R^{1,2}\Rr|^2\ra_N \leq  16 N\E \la \Ll|\kappa\Ll(\alpha\wedge\alpha'\Rr)\Rr|^2\ra_N  
    \\
    =16 N \E \la \Ll|\kappa\Ll(\alpha\wedge\alpha'\Rr)\Rr|^2\ra_\fR =16 N|\kappa|_\cH^2.  
\end{multline}
We define the function $\partial_q \bar F_N(t,q) = \dr_q \bar F_N(t,q,\cdot) \in L^2([0,1]; S^D)$ by duality so that for every $\kappa \in L^2([0,1]; S^D)$,
\begin{equation}
\label{e.def.drq.FN}
\la \kappa, \dr_q\bar F_N(t,{q})\ra_\cH = \E \la \kappa\Ll(\alpha\wedge\alpha'\Rr)\cdot\frac{\sigma\sigma'^\intercal}{N}\ra_N.
\end{equation}
The estimates \eqref{e.|d/dF_N|<} and \eqref{e.|d^2F_N|<} imply that
\begin{equation}
\label{e.FN.frechet}
\Ll| \bar F_N(t,q') - \bar F_N(t,q) - \la q'-q, \dr_q\bar F_N(t,{q})\ra_\cH\Rr| \le 8N|q'-q|_{L^2}^2. 
\end{equation}
It follows from \eqref{e.bound.on.sisi'} and \eqref{e.def.drq.FN} that  for every $\kappa \in L^2([0,1]; S^D)$, 
\begin{equation}  
\label{e.bound.drq}
\la \kappa, \dr_q\bar F_N(t,{q})\ra_\cH \leq |\kappa|_{L^1}.
\end{equation}
By duality, we deduce that $\partial_q \bar F_N(t,q) \in L^\infty_{\le 1}$. 
Therefore, we also have that
\begin{equation*}  |\bar F_N(t,q) - \bar F_N(t,q')| \le |q-q'|_{L^1}. 
\end{equation*}
In particular, we can extend the free energy $\bar F_N$ in \eqref{e.def.barFN.continuous} to all $\R_+ \times \mcl Q_1$ by continuity. We have already observed that the definitions of $\bar F_N(t,q)$ in \eqref{e.def.barFN.continuous} and \eqref{e.def.FN.piece.const} coincide when $q$ is piecewise constant, and each has been extended continuously to $\R_+\times \mcl Q_1$, so the two definitions coincide everywhere on $\R_+ \times \mcl Q_1$.

By continuity, the estimate in \eqref{e.FN.frechet}, which we showed for $q,q' \in \mcl Q_\infty$, extends to every $q' \in \mcl Q_2$. This already shows that every point $q \in \mcl Q_\infty$ is a point of Fréchet differentiability of $\bar F_N(t,\cdot) : \mcl Q_2 \to \R$, and the Fréchet derivative at this point is indeed $\partial_q \bar F_N(t,q)$. We now wish to extend this to every $q \in \mcl Q_2$. We thus fix $q \in \mcl Q_2$ and take a sequence $(q_n)_{n \in \N}$ of elements of $\mcl Q_\infty$ that converges to $q$ in $L^2$. Since $\partial_q \bar F_N(t,q_n) \in L^\infty_{\le 1}$ for every $n \in \N$, we can extract a subsequence that converges weakly in $L^2$; we denote the limit by $p$. Since $(q_n)_{n \in \N}$ converges to $q$ strongly in $L^2$, substituting $q$ by $q_n$ in \eqref{e.FN.frechet} and letting $n$ tend to infinity along the subsequence yields that, for every $q' \in \mcl Q_2$,
\begin{equation*}  \Ll| \bar F_N(t,q') - \bar F_N(t,q) - \la q'-q, p\ra_\cH\Rr| \le 8N|q'-q|_{L^2}^2. 
\end{equation*}
This characterizes $p$ uniquely, so in fact the subsequence extraction was not necessary and $\partial_q \bar F_N(t,q_n)$ converges weakly to an object that depends only on $q$ itself, not on the approximating sequence. We denote this object by $\dr_q \bar F_N(t,q)$, a notation that is consistent with the previous definition that was restricted to the case $q \in \mcl Q_\infty$. We have thus shown that $\bar F_N(t,\cdot) : \mcl Q_2 \to \R$ is Fréchet differentiable everywhere. 

Since the inequality \eqref{e.bound.drq} is preserved by taking weak $L^2$ limits on $\dr_q \bar F_N(t,q)$, we also have that $\dr_q \bar F_N(t,q)$ belongs to $L^\infty_{\le 1}$ for every $q \in \mcl Q_2$.

We recall from Proposition~\ref{p.monotone} that $\bar F_N(t,\cdot)$ is $\mcl Q_2^\star$-increasing, as defined in~\eqref{e.C_2^*-nondec}. For every smooth $q \in \mcl Q$ with derivative bounded away from zero and every smooth $\kappa \in \mcl Q_2^\star$, we must therefore have that
\begin{equation*}  \la \kappa, \dr_q\bar F_N(t,{q})\ra_\cH   \ge 0. 
\end{equation*}
By density, this is in fact valid for every $\kappa \in \mcl Q_2^\star$. This means that $\partial_q \bar F_N(t,q)$ belongs to the dual cone 
\begin{equation*}  (\C^*_2)^*:=\Ll\{{q}\in L^2([0,1]; S^D) \, \mid \, \forall {q'} \in \C_2^*, \  \la {q},{q'}\ra_\cH\geq 0 \Rr\}.
\end{equation*}
Since $\C_2$ is a non-empty closed convex cone in a Hilbert space, we have by \cite[Corollary~6.33]{bauschke2011convex}  that $(\C^*_2)^* = \C_2$. (One can also prove this more directly by proceeding similarly to the proof of Lemma~\ref{l.charact.C2star}.) We have therefore verified that $\partial_q \bar F_N(t,q) \in \mcl Q_2$ for every smooth $q \in \mcl Q$ whose derivative is bounded away from zero. Since $\mcl Q_2$ is closed and convex, it is also weakly closed, by \cite[Corollary~1.4]{ET} (or again by direct arguments). For general $q \in \mcl Q_2$, since we can obtain $\dr_q \bar F_N(t,q)$ as a weak $L^2$ limit of similar objects computed at smooth paths with derivative bounded away from zero, we deduce that $\dr_q \bar F_N(t,q)$ belongs to $\mcl Q_2$ for every $q \in \mcl Q_2$. 

So far we have only observed the continuity of the mapping $q \mapsto \dr_q \bar F_N(t,q)$ in the sense of weak $L^2$ convergence. Now that we also know that $\dr_q \bar F_N(t,q) \in \mcl Q \cap L^\infty_{\le 1}$, we can appeal to Lemma~\ref{l.compact.embed} and upgrade the continuity property of the mapping $q \mapsto \dr_q \bar F_N(t,q)$ to the strong $L^r$ topology, for every $r < +\infty$. Precisely, for every $r < +\infty$ and every sequence $(q_n)_{n \in \N} \subset \mcl Q_2$ converging to $q \in \mcl Q_2$ in $L^2$, we have 
\begin{equation}  
\label{e.conv.der.Lr}
\lim_{n \to + \infty} \Ll|\dr_q \bar F_N(t,q_n) - \dr_q \bar F_N(t,q) \Rr|_{L^r} = 0.
\end{equation}

It remains to show \eqref{e.continuity.der.FN}. 
We temporarily display the dependence of $\la\cdot\ra_N$ on ${q}$ explicitly by writing $\la\cdot\ra_N = \la\cdot\ra_{N,{q}}$. We fix ${q}, {q}'\in\C_\infty$ as in the beginning of this proof, with $\kappa = q'-q$, and for each $\eta \in L^\infty$, we aim to compute the derivative in $s$ of 
\begin{equation*}  \E \la \eta\Ll(\alpha\wedge\alpha'\Rr)\cdot \frac{\sigma\sigma'^\intercal}{N}\ra_{N, {q} + s \kappa }.
\end{equation*}
We realize the Gaussian processes $(w^{q + s \kappa}(\alpha))_{\alpha \in \mfk U}$ as discussed below \eqref{e.wq.distr.identity}. As for \eqref{e.diff.q.second}, recalling the notation $R^{\ell, \ell'}$ in \eqref{e.def.R.kappa} and denoting 
$Q^{\ell,\ell'} := \eta(\alpha^\ell\wedge\alpha^{\ell'})\cdot \frac{\sigma^\ell(\sigma^{\ell'})^\intercal}{N}$, we have
\begin{align*}
    \frac{\d}{\d s } \E \la \eta\Ll(\alpha\wedge\alpha'\Rr)\cdot\frac{ \sigma\sigma'^\intercal}{N}\ra_{N,{q} + s\kappa } = 2N\E \la Q^{1,2}\Ll(R^{1,2}-4R^{1,3}+3R^{3,4}\Rr)\ra_{N,{q} + s \kappa }.
\end{align*}
By H\"older's inequality and the invariance of Poisson--Dirichlet cascades from Proposition~\ref{p.invar_cts}, we get, for every $r, r' \in [1,+\infty]$ with $\frac 1 r + \frac 1 {r'} = 1$, 
\begin{align*}
\Ll|\frac{\d}{\d s } \E \la \eta\Ll(\alpha\wedge\alpha'\Rr)\cdot \frac{\sigma\sigma'^\intercal}{N}\ra_{N,{q} + s \kappa }\Rr| 
    &\leq 16N \E \la \Ll|Q^{1,2}\Rr|^{r'}\ra_{N,{q}  + s \kappa}^{1/r'}\E \la \Ll|R^{1,2}\Rr|^r \ra_{N,{q}  + s \kappa}^{1/r}
    \\
    &\leq 16N\E \la \Ll|\eta\Ll(\alpha\wedge\alpha'\Rr)\Rr|^{r'} \ra_{N,{q}  + s \kappa }^{1/r'}\E \la \Ll|\kappa\Ll(\alpha\wedge\alpha'\Rr)\Rr|^r \ra_{N,{q}  + s \kappa }^{1/r}\notag
    \\
     &= 16N \E \la \Ll|\eta\Ll(\alpha\wedge\alpha'\Rr)\Rr|^{r'} \ra_{\fR}^{1/r'}\la \Ll|\kappa\Ll(\alpha\wedge\alpha'\Rr)\Rr|^r \ra_{\fR}^{1/r}
     \notag
     \\
     & = 16N |\eta|_{L^{r'}}|q'-q|_{L^r}. \notag
\end{align*}
Integrating the above in $ s $ from $0$ to $1$ and using~\eqref{e.def.drq.FN}, we get
\begin{align}
\label{e.estim.der.der}
    \Ll|\la \eta, \dr_q\bar F_N(t,{q}')\ra_\cH - \la \eta, \dr_q\bar F_N(t,{q})\ra_\cH\Rr|\leq 16 N|\eta|_{L^{r'}}|{q}'-{q}|_{L^r}.
\end{align}
We established \eqref{e.estim.der.der} for every $q, q' \in \mcl Q_\infty$ and $\eta \in L^\infty$. By duality, this implies that, for every $q, q' \in \mcl Q_\infty$,
\begin{equation*}  \Ll| \dr_q \bar F_N(t,q') - \dr_q \bar F_N(t,q) \Rr|_{L^r} \le 16N|q'-q|_{L^r}. 
\end{equation*}
For $r < +\infty$, we can use \eqref{e.conv.der.Lr} to extend this estimate to every $q,q' \in \mcl Q_2$ with $q'-q \in L^r$. The case $r = +\infty$ can then be obtained by using the result for finite $r$ and letting $r$ tend to infinity. The proof is therefore complete.
\end{proof}

Recall from Proposition~\ref{p.init.cond} that the initial condition is given by $\psi = \bar F_1(0,\cdot)$, which can be written explicitly as
\begin{align}\label{e.psi=}
    \psi({q}) = - \E\log  \iint \exp\Ll(\sqrt{2}w^{q}(\alpha)\cdot \sigma - {q}(1)\cdot \sigma\sigma^\intercal\Rr) \d P_1(\sigma) \d \fR(\alpha).
\end{align}
Proposition~\ref{p.F_N_smooth} immediately yields similar statements concerning the function $\psi$. We record them here for future reference. For every ${q}\in\C_\infty$, we write
\begin{align}\label{e.<>_mu}
    \la \cdot \ra_{q}\propto \exp \Ll(w^{{q}}(\alpha)\cdot\sigma-{q}(1)\cdot\sigma\sigma^\intercal\Rr) \d P_1(\sigma)\d\fR(\alpha).
\end{align}

\begin{corollary}\label{c.psi_smooth}
The function $\psi$ is $L^1$-Lipschitz continuous, with Lipschitz constant equal to $1$. 
The restriction of the function $\psi$ to $\mcl Q_2$ is Fréchet (and Gateaux) differentiable everywhere; we denote its Fréchet (and Gateaux) derivative by $\dr_q \psi(q) = \dr_q \psi(q, \cdot) \in L^2([0,1]; S^D)$. We have, for every $q \in \mcl Q_2$,
\begin{equation}
\label{e.bound.der.psi}
\partial_q \psi(q) \in \mcl Q \cap L^\infty_{\le 1},
\end{equation}
and, for every $q \in \mcl Q_\infty$ and $\kappa \in L^2([0,1]; S^D)$,
\begin{equation}
\label{e.def.der.psi}
\la \kappa,\dr_q \psi({q})\ra_\cH = \E \la \kappa \Ll(\alpha\wedge\alpha'\Rr)\cdot \sigma\sigma'^\intercal\ra_q.
\end{equation}
Moreover, for every $r \in [1,+\infty]$ and $q, q' \in \mcl Q_2$ with $q'-q \in L^r$, we have
\begin{equation}
\label{e.continuity.der.psi}
\Ll|\dr_q \psi({q}) -\dr_q \psi({q'})\Rr|_{L^r}\leq 16\Ll|{q}-{q'}\Rr|_{L^r}.
\end{equation}
In particular, the mapping $q \mapsto \dr_q \psi(q)$ can be extended to $\mcl Q_1$ by continuity, and the properties in \eqref{e.bound.der.psi} and \eqref{e.continuity.der.psi} remain valid with $q, q' \in \mcl Q_1$.  
\end{corollary}

We recall that Proposition~\ref{p.precompact} asserts that the sequence $(\bar F_N)_{N \in \N}$ is precompact. The regularity estimates on $\bar F_N$ that hold uniformly over $N$ naturally transfer to any subsequential limit. Using also the results from Section~\ref{s.diff}, we infer that any such limit must be Gateaux differentiable ``almost everywhere''. We summarize these points in the next proposition. 

\begin{proposition}[Regularity of the limit]
\label{p.reg.lim}
Suppose that $\bar F_N$ converges to some limit $f$ pointwise along some subsequence. 
Then for every $r \in (1,+\infty]$, the function $\bar F_N$ converges locally uniformly to $f$ in $\R_+ \times \mcl Q_r$ along the subsequence. We have that $f$ satisfies the same Lipschitz and local semi-concavity properties as $\bar F_N$ stated in Propositions~\ref{p.barF_N_Lip} and~\ref{p.semiconc}, and that $f$ is $\mcl Q_2^\star$-increasing as $\bar F_N$ is from Proposition~\ref{p.monotone}. Moreover, 
\begin{itemize}
    \item for each $t\geq 0$, 
    there is a Gaussian null set $\mathcal{N}_t$ of $\cH$ such that $f(t,\cdot) : \mcl Q_2 \to \R$ is Gateaux differentiable at every point in $\C_2\setminus\mathcal{N}_t$ and $ (\C_\uparrow\cap L^\infty)\setminus\mathcal{N}_t$ is dense in $\C_2$;

    \item there is a Gaussian null set $\mathcal{N}$ of $\R\times\cH$ such that $f : \R_+ \times \mcl Q_2 \to \R$ is Gateaux differentiable on $(\R_+\times \C_2)\setminus\mathcal{N}$ and $(\R_+\times (\C_\uparrow\cap L^\infty))\setminus\mathcal{N}$ is dense in $\R_+\times \C_2$,
\end{itemize}
\end{proposition}
In the statement of Proposition~\ref{p.reg.lim} and throughout the paper, we allow ourselves some imprecision when saying that $\bar F_N$ converges pointwise to $f$ along a subsequence, since we do not clearly specify over which set this pointwise convergence is valid. As long as there exists a dense set of points $(t,q) \in \R_+ \times \mcl Q_1$ such that $\bar F_N(t,q)$ converges to $f(t,q)$ along the subsequence, the convergence immediately extends to every $(t,q) \in \R_+ \times \mcl Q_1$ by the Lipschitz continuity property in Proposition~\ref{p.barF_N_Lip}.

\begin{proof}[Proof of Proposition~\ref{p.reg.lim}]
The fact that the convergence of $\bar F_N$ to $f$ along the subsequence holds locally uniformly in $\R_+\times \mcl Q_r$, for every $r \in (1,+\infty]$, was already mentioned and follows from Proposition~\ref{p.precompact}. The fact that $f$ is Lipschitz, locally semi-concave, and $\mcl Q_2^*$-increasing is immediate from Propositions~\ref{p.barF_N_Lip} and \ref{p.semiconc}. The differentiability statements follow from Proposition~\ref{p.Gateaux_dff_dense}.
\end{proof}

The next proposition asserts that at points of differentiability of the limit $f$, the derivatives of $\bar F_N$ converge to those of $f$.

\begin{proposition}[Convergence of derivatives]
\label{p.conv.der}
Suppose that $\bar F_N$ converges pointwise to some limit $f$ along a subsequence $(N_k)_{k \in \N}$. 

(1) For each $t \ge 0$, if $f(t,\cdot)$ is Gateaux differentiable at $q \in \C_\uparrow$, then $\partial_q \bar F_{N_k}(t, q, \cdot)$ converges in $L^r$ to $\partial_q f(t,q,\cdot)$ for every $r \in [1,+\infty)$.

(2) For each $q \in \mcl Q_1$, if $f(\cdot,q)$ is differentiable at $t > 0$, then $\partial_t \bar F_{N_k}(t, q)$ converges to $\partial_t f(t,q)$. 
\end{proposition}
\begin{proof}
We only show part (1) of the statement, since the second part is only easier. We decompose the proof into three steps.

\medskip

\noindent \emph{Step 1.} We start by recalling that for every $a,b \in S^D$, the largest eigenvalue of $a+b$ is smaller than the sum of the largest eigenvalues of $a$ and $b$, and similarly with the smallest eigenvalue. This can be seen from the variational representation of the largest and smallest eigenvalues of a symmetric matrix. 

Recall also the definition of $\mcl Q_{\uparrow, c}$ in \eqref{e.def.C_c}, and that $\mcl Q_{\uparrow} = \bigcup_{c > 0} \mcl Q_{\uparrow, c}$. Let $q \in \mcl Q_{\uparrow, c}$, and let $\kappa \in C^\infty([0,1]; S^D)$ be a smooth function with $\kappa(0) = 0$. In this step, we show that there is $\delta > 0$ such that for every $s \in [-\de, \de]$, we have $q + s \kappa \in \mcl Q_{\uparrow,\frac{c}{2}}$.

Let $\eta \in (0, c)$ be a parameter to be determined in the course of the argument, and let $\delta > 0$ be sufficiently small that for every $u\le v \in [0,1]$, every eigenvalue of $\delta (\kappa(u)- \kappa(v))$ lies in the interval $[-\eta (v-u),\eta (v-u)]$. (This amounts to taking $\delta > 0$ sufficiently small that every eigenvalue of $\de \dot \kappa(u)$ belongs to $[-\eta , \eta ]$ for every $u \in [0,1]$.) For  every $s \in [-\de, \de]$ and $0 \le u \le v < 1$, we have
\begin{equation*}  (q+s\kappa)(v)-(q+s\kappa)(u)\geq (c-\eta )(v-u)\id. 
\end{equation*}
 Let $\lambda_{\max}$ and $\lambda_{\min}$ be the largest and the smallest eigenvalues of $q(v)-q(u)$ respectively. We have 
\begin{align*}
    \ellipt((q+s\kappa)(v)-(q+s\kappa)(u))\leq 
    \frac{\lambda_{\max} + \eta(v-u)}{\lambda_{\min} - \eta(v-u)} \leq \frac{1 + \tfrac \eta c}{1 - \tfrac \eta c} \, \frac{\lambda_{\max}}{\lambda_{\min}},
\end{align*}
where we used that $\lambda_{\max} \ge \lambda_{\min}\geq c(v-u)$ in the second inequality. In view of the last two displays, we obtain the announced result by choosing $\eta = \frac c 3$. 

\medskip

\noindent \emph{Step 2.} In this step, we show part (1) of the statement except that we only prove that $\partial_q \bar F_{N_k}(t, q) = \partial_q \bar F_{N_k}(t_{N_k}, q_{N_k}, \cdot)$ converges weakly in $L^2$ to $\partial_q f(t,q) = \partial_q f(t,q,\cdot)$. The notion of convergence will be upgraded in the next step. 

By assumption, there exists $c > 0$ such that $q \in \mcl Q_{\uparrow, c}$. Let $\kappa \in C^\infty([0,1]; S^D)$ be a smooth function with $\kappa(0) = 0$. By the result of the previous step, we can find $\delta > 0$ such that for every $s \in [-\de, \de]$, we have $q + s \kappa \in \mcl Q_{\uparrow, \frac c 2}$. 
By Proposition~\ref{p.semiconc}, there is a constant $C < +\infty$ (depending on $c$) such that for every $N \in \N$, $s \in [-\delta, \delta]$, and  $\lambda \in [0,1]$, we have
\begin{align*}
    (1-\lambda)\bar F_{N_k}(t, {q}) + \lambda \bar F_{N_k}(t,{q}+ s \kappa) - \bar F_{N_k}(t, {q}+\lambda s \kappa) \leq C\lambda(1-\lambda)s^2 |\dot\kappa|^2_\cH.
\end{align*}
Dividing both sides by $\lambda$ and sending $\lambda$ to zero, we get
\begin{align*}
    \bar F_{N_k}(t,{q}+s\kappa) - \bar F_{N_k}(t, {q}) -s\la \kappa,\dr_q \bar F_{N_k}(t, {q})\ra_\cH \leq Cs^2 |\dot\kappa|^2_\cH.
\end{align*}
Since $|\dr_q \bar F_{N_k}(t, {q})|_\cH$ is bounded uniformly in $k \in \N$ due to Proposition~\ref{p.F_N_smooth}, for any subsequence $(N'_k)_{k\in\N}$ of $(N_k)_{k\in\N}$, we can extract a further subsequence $(N''_k)_{k\in\N}$ along which $\dr_q \bar F_{N}(t, {q})$ converges weakly to some ${p}\in\cH$. Passing to the limit along $(N''_k)_{k\in\N}$, we get
\begin{align*}
    f(t,{q}+s\kappa) -f(t,{q}) - s \la \kappa, {p}\ra_\cH \leq Cs^2|\dot\kappa|^2_\cH
\end{align*}
for every $s\in[-\delta,\delta]$.
Sending $s$ to zero and using the Gateaux differentiability of $f(t,\cdot)$ at~${q}$, we obtain that
\begin{equation*}  \la \kappa ,\dr_q f(t,{q}) \ra_\cH = \la \kappa, {p}\ra_\cH. 
\end{equation*}
Since the set of $\kappa$'s that we have considered is dense in $\cH$, we  conclude that $\dr_q f(t, {q}) = {p}$. 

We have thus shown that any subsequence $(N'_k)_{k\in\N}$ of $(N_k)_{k\in\N}$ has a further subsequence $(N''_k)_{k\in\N}$ along which $\dr_q \bar F_{N}(t, {q})$ converges weakly to $\dr_q f(t, {q})$. This therefore completes the argument for the weak convergence in $L^2$ of $\dr_q \bar F_{N_k}(t,q)$ to $\dr_q f(t,q)$.

\medskip

\noindent \emph{Step 3.} In this step, we improve the convergence obtained in the previous step from weakly in $L^2$ to strongly in $L^r$ for every $r \in [1,+\infty)$.  We recall from \eqref{e.bounds.der.FN} that $\dr_q \bar F_N(t,q) \in \mcl Q \cap L^\infty_{\le 1}$. The strong convergence in $L^r$ therefore follows from Lemma~\ref{l.compact.embed}. The proof is therefore complete.
\end{proof}

\section{Cavity calculations and ultrametricity}
\label{s.cavity}

The goal of this section is to present a number of cavity calculations and exploit the asymptotic ultrametricity of the Gibbs measure. Our arguments contain and extend the method introduced in \cite{aizenman2003extended} and now often called the Aizenman--Sims--Starr scheme. Roughly speaking, the cavity calculation of \cite{aizenman2003extended} consists in comparing $\bar F_{N+1}$ with $\bar F_N$ by separating out the integration of one of the spin variables. Next, if one knows that the Gibbs measure is approximately ultrametric, then one can observe that the integration of the spun-off spin variable can be expressed in the form of 
\begin{equation*}  (N+1) \bar F_{N+1} - N \bar F_N \simeq
\psi({q}+t\nabla\xi({p})) -t \int_0^1\big(p \cdot \nabla \xi(p) - \xi(p)\big),
\end{equation*}
for some $p \in \C_\infty$ which may depend on $N$. Ultrametricity can indeed be enforced by means of small perturbations of the Hamiltonian, as was shown in \cite{pan.aom} and reviewed in detail in \cite{pan}. Intuitively, for each fixed $N$, the natural choice of $p$ is such that the law of the overlap under the ``cavity Gibbs measure'' with $N$ spins is the law of $p(U)$, where $U$ is uniformly distributed over $[0,1]$. 

We will perform a family of other cavity calculations that will allow us to keep track of arbitrary continuous observables of the overlap. This will have two main benefits that will be detailed in the next section. The first one is that we will indeed be able to track the identity of $p$ in the calculation sketched above. Under a differentiability assumption on the limit free energy $f$, we will justify that we can choose $p = \dr_q f(t,q)$. The second crucial benefit is that we will discover another property that this $p$ must satisfy, namely, that
\begin{equation*}  p = \dr_q \psi(q + t \nabla \xi(p)).
\end{equation*}
Collecting all this information then essentially yields a proof of our main results. 

This section is organized as follows. We first define the various terms that will enter into the cavity calculations. We next perform the Aizenman--Simms--Starr calculation for the free energy. We then do a cavity calculation for a general observable of the overlap of the spun-off spin variables. We continue by establishing that the quantities we obtain in these calculations can be approximated arbitrarily closely by continuous functions of the law of the overlap array. We next discuss Ghirlanda--Guerra identities, which are enforced by small perturbations of the Hamiltonian and ensure ultrametricity. The ultrametricity property implies that the law of the overlap array is a continuous function of the law of the overlap between only two replicas; and Poisson--Dirichlet cascades provide us with a canonical representation of these objects. So all the ``cavity averages'' we obtained in the cavity calculations, which involve the ``cavity Gibbs measure'', can in fact be re-expressed as cavity averages with respect to some fixed Poisson--Dirichlet cascade.

Throughout this section, we fix $(t,{q}) \in \R_+\times \C_\infty$, and the integer $M\in \N$ denotes the number of cavity spins. For the purpose of this work, we only need to consider the case when~$M=1$, but we prefer to carry the cavity calculations with a general $M$, as these might be useful for future applications.

\subsection{Definitions and notation}
\subsubsection{Hamiltonians and perturbation}
We start by introducing the various Hamiltonians that will enter into the cavity calculations. The basic idea of a cavity calculation is to decompose an element $\rho \in \R^{D \times (N+M)}$ into $\rho = (\sigma, \tau)$, with $\sigma \in \R^{D \times N}$ and $\tau \in \R^{D \times M}$, and to express the free energy of Gibbs averages involving $N+M$ variables as averages of the $\tau$ variables under some Gibbs measure over the $\sigma$ variables. We are in particular interested in separating out the dependence in $\sigma$ and $\tau$ of the Hamiltonian $H_{N+M}(\rho) = H_{N+M}((\sigma,\tau))$. Temporarily focusing on the case $M = 1$, we would like for instance to argue that 
\begin{equation}  
\label{e.expansion.HN}
H_{N+1}(\sigma,\tau) \simeq H_{N+1}(\sigma,0) + \tau \cdot \nabla_\tau H_{N+1}(\sigma,0),
\end{equation}
where we use the somewhat informal notation $\nabla_\tau H_{N+1}$ to denote the gradient in $\tau$ of the mapping $(\sigma,\tau) \mapsto H_{N+1}(\sigma,\tau)$. Since $\rho \rho^\intercal = \si \si^\intercal + \tau \tau^\intercal$, we have from \eqref{e.def.xi} that for every $(\sigma,\tau), (\sigma',\tau') \in \R^{D\times (N+1)}$,
\begin{equation*}  \E \Ll[ H_{N+1}(\sigma,\tau) H_{N+1}(\sigma',\tau') \Rr] = (N+1) \xi \Ll( \frac{\sigma \sigma'^\intercal + \tau \tau'^\intercal}{N+1} \Rr) .
\end{equation*}
Differentiating this expression, we get that for every $a, b \in \R^D$, 
\begin{equation}  
\label{e.diff1.defxi}
\E \Ll[ H_{N+1}(\sigma,\tau) \big(b\cdot \nabla_\tau H_{N+1}(\sigma',\tau')\big) \Rr] = \tau b^\intercal \cdot \nabla \xi \Ll( \frac{\sigma \sigma'^\intercal + \tau \tau'^\intercal}{N+1} \Rr) 
\end{equation}
and
\begin{equation}  
\label{e.diff2.defxi}
\E \Ll[ \big(a \cdot \nabla_\tau H_{N+1}(\sigma,0)\big) \big(b \cdot \nabla_\tau H_{N+1}\big)(\sigma',0) \Rr] = b a^\intercal \cdot \nabla \xi \Ll( \frac{\sigma \sigma'^\intercal}{N+1} \Rr) .
\end{equation}
We see from \eqref{e.diff1.defxi} that $(H_{N+1}(\sigma,0))_{\sigma \in \R^{D\times N}}$ and $(\nabla_\tau H_{N+1}(\sigma,0))_{\sigma \in \R^{D \times N}}$ are independent Gaussian fields. The Hamiltonian $H_{N+1}(\sigma,0)$ will serve as our reference Hamiltonian, the common part between $H_{N+1}$ and $H_N$, and we will similarly expand $H_N(\sigma)$ into $H_{N+1}(\sigma,0)$ plus an independent Gaussian field. From now on, we revert to the setting with general $M \in \N$, and introduce convenient notation to denote the various Hamiltonians entering our calculations. While we may not always repeat it, each time we introduce a new Hamiltonian or other Gaussian field, we impose that it be independent of all the previously-defined random fields; and the expectation $\E$ is understood to average all these processes.

In order to enforce the asymptotic validity of the Ghirlanda--Guerra identities, and thereby the ultrametricity of the Gibbs measure, we need to introduce a rich family of perturbations of the Hamiltonian.
Let $(\lambda_n)_{n\in\N}$ be an enumeration of $[0,1]\cap \Q$, and $(a_n)_{n\in\N}$ be an enumeration of $\{a\in \S^\D_+ :|a|\leq 1\}\cap \Q^{D\times D}$. For every $h\in \N^4$, let $(H^h_N(\sigma,\alpha))_{\sigma\in\R^{\D\times N},\,\alpha\in \mfk U}$ 
be an independent centered Gaussian process with covariance
\begin{align*}\E \Ll[H^h_N(\sigma,\alpha)H^h_N(\sigma',\alpha')\Rr] = N\Ll(a_{h_1}\cdot \Ll(\frac{\sigma\sigma'^\intercal}{N}\Rr)^{\odot h_2}+\lambda_{h_3} \alpha\wedge\alpha'\Rr)^{h_4}
\end{align*}
where $\odot$ is the Schur product of matrices, i.e.\ $a\odot b = (a_{i,j}b_{i,j})_{i,j}$.
Let us explain why this process exists. By \cite[(6.1)-(6.2)]{mourrat2023free}, there exists a Gaussian process with covariance $a\cdot \Ll(\frac{\sigma\sigma'^\intercal}{N}\Rr)^{\odot n}$ for every $a\in\S^\D_+$ and $n\in\N$. Proposition~\ref{p.gaussian_cascade} ensures the existence of a Gaussian process with covariance $\alpha\wedge\alpha'$. Linear combinations with non-negative coefficients of covariance functions are still covariance functions, since we can simply take superpositions of independent processes to realize them, and \cite[Proposition~6.5]{mourrat2023free} also allows for the multiplication of two covariance functions. Finally, the measurability can be handled as in the proof of Proposition~\ref{p.gaussian_cascade}.

For each $h \in \N^4$, let $c_h>0$ be a constant such that
\begin{align*}c_h \sqrt{\frac{1}{N}\E  H^h_N(\sigma,\alpha)H^h_N(\sigma,\alpha) }\leq 2^{-|h|_1},
\end{align*}
uniformly over $\sigma \in \supp P_N$, $\alpha \in \mfk U$, and $N \in \N$, and where we write $|h|_1 := \sum_{i=1}^4 h_i$.
Let $\{e_i\}_{i=1}^{D(D+1)/2}$ be an orthonormal basis of $\S^{D}$.
For every
\begin{align}\label{e.x_pert}
    x = \Ll((x_h)_{h\in\N^4}, (x_i)_{i=1}^{D(D+1)/2} \Rr)\in [0,3]^{\N^4 \times \frac{D(D+1)}{2}},
\end{align}
we set
\begin{align}\label{e.H^pert}
\begin{split}
    H_{N}^{x}(\sigma,\alpha) &:= \check H_N^{x}(\sigma,\alpha)+ \sum_{i=1}^{D(D+1)/2} x_i e_i\cdot \sigma\sigma^\intercal, \qquad \text{where} 
    \\
    \check H_N^{x}(\sigma,\alpha)&:= \sum_{h\in \N^4} x_h  c_h H^h_N(\sigma,\alpha).
\end{split}
\end{align}
The second term in the definition of $H_N^x$ will ensure the concentration of the ``self-overlap''~$\frac{\sigma \sigma^\intercal}{N}$. Notice that we impose each coordinate of $x$ to take values in the interval~$[0,3]$. 
We define the free energy with perturbation 
\begin{align}
    \bar F^x_N(t,{q}) &:= -\frac{1}{N}\E\log \iint \exp\Ll(H^{t,{q}}_N(\sigma,\alpha) + N^{-\frac{1}{16}}H^{x}_N(\sigma,\alpha)\Rr)\d P_N(\sigma) \d \fR(\alpha),
    \label{e.F^x_N=}
\end{align}
and its associated Gibbs measure
\begin{align}
    \la\cdot\ra^\orig_{N,x} &\propto  \exp\Ll(H^{t,{q}}_N(\sigma,\alpha) + N^{-\frac{1}{16}}H^{x}_N(\sigma,\alpha)\Rr)\d P_N(\sigma) \d \fR(\alpha).\label{e.<>^orig_Nx=}
\end{align}
The exponent $1/16$ in \eqref{e.F^x_N=} is chosen for convenience, any smaller exponent would also do as long as it is strictly positive. 
We keep denoting by $(\sigma,\alpha)$ the canonical random variable under $\la\cdot\ra^\orig_{N,x}$, and we write $(\sigma^\ell, \alpha^\ell)_{\ell \ge 1}$ to denote independent copies of $(\sigma,\alpha)$. 
The expectation $\E$ in \eqref{e.F^x_N=} integrates all the sources of Gaussian randomness and the randomness of $\fR$.

Our ``reference'' Hamiltonian for the cavity calculation is denoted by $(\tilde H_N(\sigma))_{\sigma\in \R^{D\times N}}$, and it is the centered Gaussian process such that, for every $\si, \si' \in \R^{D\times N}$,
\begin{align*}
    \E \Ll[\tilde H_N(\sigma)\tilde H_N(\sigma')\Rr] = (N+M)\xi\Ll(\frac{\sigma\sigma'^\intercal}{N+M}\Rr).
\end{align*}
This is simply an independent copy of what we were denoting $H_{N+M}(\sigma,0)$ at the opening of this section. Notice that the parameter $M$ appears in the definition of the covariance above. Abusing notation, we prefer to keep the dependence of $\tilde H_N$ on $M$ implicit, as it only takes $\sigma \in \R^{D\times N}$ as input. 
We let $\tilde W^{q}_N$ be an independent copy of $W^{q}_N$ defined in~\eqref{e.W^mu(alpha)}, and, for every $\sigma \in \R^{D\times N}$ and $\alpha \in \mfk U$, we set
\begin{align*}\tilde H^{t,{q}}_N(\sigma,\alpha) := \sqrt{2t}\tilde H_N(\sigma) - t(N+M)\xi\Ll(\frac{\sigma\sigma^\intercal}{N+M}\Rr) +\tilde W^{q}_N(\alpha)\cdot \sigma - {q}(1)\cdot\sigma\sigma^\intercal.
\end{align*}
We denote the free energy and the Gibbs measure arising in the cavity calculation by
\begin{align}\label{e.tildeF_N=}
    \tilde F_N^x(t,{q}) := -\frac{1}{N}\E\log\iint \exp\Ll(\tilde H^{t,{q}}_N(\sigma,\alpha)+N^{-\frac{1}{16}} H^{x}_{N}(\sigma,\alpha)\Rr)\d P_N(\sigma)\d\fR(\alpha)
\end{align}
and
\begin{align}\label{e.<>^cav_N,x}
    \la\cdot\ra^\cav_{N,x} \propto \exp\Ll(\tilde H^{t,{q}}_N(\sigma,\alpha) +  N^{-\frac{1}{16}}H^{x}_N(\sigma,\alpha)\Rr)\d P_N(\sigma)\d\fR(\alpha).
\end{align}

\subsubsection{Notation for the overlaps}

For every $N\in\N$, $\ell,\ell'\in\N$, $h\in \N^4$, and $n\in \N$, we write
\begin{align}\label{e.overlap_notation}
    \begin{cases}
        R^{\ell,\ell'}_{N,\sigma} := \frac{\sigma^\ell\Ll(\sigma^{\ell'}\Rr)^\intercal}{N},\qquad R^{\ell,\ell'}_\alpha :=\alpha^\ell\wedge\alpha^{\ell'},\qquad R^{\ell,\ell'}_N :=\Ll(R^{\ell,\ell'}_{N,\sigma}, R^{\ell,\ell'}_\alpha\Rr);
        \\
        R^{\ell,\ell'}_{N,h} := \Ll(a_{h_1}\cdot \Ll(R^{\ell,\ell'}_{N,\sigma}\Rr)^{\odot h_2}+\lambda_{h_3} R^{\ell,\ell'}_\alpha\Rr)^{h_4};
        \\
        R_N := \Ll(R^{\ell,\ell'}_N\Rr)_{\ell,\ell'\in\N},\qquad R^{\leq n}_N := \Ll(R^{\ell,\ell'}_N\Rr)_{\ell,\ell'\leq n} .
    \end{cases}
\end{align}
For every ${p} \in \C_\infty$, we set
\begin{align}\label{e.overlap_alpha}
    Q^{\ell,\ell'}_{p} := \Ll({p}\Ll(R^{\ell,\ell'}_\alpha\Rr), R^{\ell,\ell'}_\alpha\Rr),\qquad
    Q_{p} := \Ll(Q^{\ell,\ell'}_{p}\Rr)_{\ell,\ell'\in \N},\qquad Q^{\leq n}_{p} := \Ll(Q^{\ell,\ell'}_{p}\Rr)_{1\leq \ell,\ell'\leq n}.
\end{align}
Under some circumstances, we will be able to argue that the $R_{N,\sigma}$-overlaps synchronize with the $R_\alpha$-overlaps; this means that the $R_N$-overlaps are close to the $Q_p$-overlaps for a suitable choice of $p$.

\subsubsection{Definitions for the free-energy cavity calculation}
We now define the Gibbs average that will show up in the free-energy cavity calculation, or in other words in the Aizenman--Sims--Starr scheme \cite{aizenman2003extended}. We define $\theta:\R^{\D\times\D}\to\R$ such that, for every $a \in \R^{D \times D}$,
\begin{align}\label{e.theta=}
    \theta(a) := a\cdot \nabla \xi(a) - \xi(a).
\end{align}
We define the following independent centered Gaussian processes indexed by $\sigma\in \R^{D\times N}$:
\begin{itemize}
    \item let $\msf Z(\sigma) = (\msf Z_1(\sigma),\dots, \msf Z_M(\sigma))$ be an $M$-tuple of independent $\R^{D}$-valued centered Gaussian vectors $\msf Z_i(\sigma)$ with covariance $\E \msf Z_i(\sigma) \msf Z_i(\sigma')^\intercal =\nabla \xi (\frac{\sigma\sigma'^\intercal}{N})$;
    \item let $\msf Y(\sigma)$ be real-valued with covariance $\E \msf Y(\sigma)\msf Y(\sigma') =  M\theta (\frac{\sigma\sigma'^\intercal}{N})$.
\end{itemize}
The fact that there indeed exists such a Gaussian process $\msf Z$ follows from \eqref{e.diff2.defxi} and a simple rescaling. Concerning the existence of the process $\msf Y$, we recall from Subsection~\ref{ss.explicit.H} that $\xi$ takes the form given in \eqref{e.explicit.xi}, where the matrices $\mathsf{C}^{(p)}$ belong to $S^{D^p}_+$ and the series is absolutely convergent. A direct calculation yields that
\begin{equation*}  \theta(a) = \sum_{p = 1}^\infty (p-1)\mathsf{C}^{(p)}\cdot a^{\otimes p}.
\end{equation*}
This is of the same form as \eqref{e.explicit.xi}, so the existence of the process $\msf Y$ follows from the second part of Subsection~\ref{ss.explicit.H}.

For every $\sigma \in \R^{D\times N}$, $\tau \in \R^{D \times M}$, and $\alpha \in \mfk U$, we set
\begin{align}\label{e.U=}
    U(\sigma,\alpha,\tau) := \sqrt{2t}\msf Z(\sigma)\cdot \tau - t\nabla\xi\Ll(\frac{\sigma\sigma^\intercal}{N}\Rr)\cdot\tau\tau^\intercal + \sqrt{2}W^{{q}}_M(\alpha)\cdot\tau-{q}(1)\cdot\tau\tau^\intercal,
\end{align}
and for every $x \in[0,3]^{\N^4 \times \frac{D(D+1)}{2}}$,
\begin{multline}
\label{e.A_N(x)}
    A_N(x)  := \E \log \la \int \exp\Ll(U(\sigma,\alpha,\tau)\Rr)\d P_M(\tau) \ra^\cav_{N,x}
    \\
     - \E \log \la \exp\Ll(\sqrt{2t}\msf Y(\sigma)- tM\theta \Ll( \frac{\sigma\sigma^\intercal}{N}\Rr)\Rr) \ra^\cav_{N,x}.
\end{multline}
With $\psi$ and $\theta$ defined in~\eqref{e.psi=} and~\eqref{e.theta=}, we set, for every ${p}\in \C_\infty$,
\begin{align}\label{e.mathscrP}
    \mathscr{P}_{t,{q}}({p}) := \psi({q}+t\nabla\xi({p})) -t \int_0^1\theta({p}(s))\d s.
\end{align}
We will show that the limit of $A_N(x)$ is related to $\mathscr{P}_{t,{q}}({p})$ for some ${p}$.

\subsubsection{Definitions for the Gibbs-average cavity calculation}
We write $\rho \in \R^{\D\times (N+ M)}$ as $\rho = (\sigma,\tau) \in \R^{\D\times N}\times \R^{\D\times M}$. 
For a bounded continuous function $g:\R^{\D\times\D}\times \R\to\R$, we want to find the limit of $\E \la g(\tau\tau'^\intercal,\alpha\wedge\alpha')\ra^\orig_{N+M,x}$ as $N$ tends to infinity. To study this, we need a few more definitions. 
We define the Gibbs measure $\la\cdot\ra^\star_{N+M,x}$ by
\begin{align}\label{e.<>_N^star}
    \la\cdot\ra^\star_{N+M,x}\propto\exp\Ll(\tilde H_N^{t,{q}}(\sigma,\alpha) +  U(\sigma,\alpha,\tau)
    + N^{-\frac{1}{16}}  H^{x}_N(\sigma,\alpha)\Rr)\d P_{N+M}(\sigma,\tau) \d\fR(\alpha).
\end{align}
The canonical random variable under $\la\cdot\ra^\star_{N+M,x}$ is denoted by $(\rho, \alpha) = (\sigma,\tau,\alpha)$. 
For every $\pi \in \C_\infty$, we define the Gibbs measure $\la\cdot\ra_{\fR,\pi}$ by
\begin{align}\label{e.<>_R,pi}
    \la \cdot\ra_{\fR,\pi} \propto \exp\Ll(\sqrt{2}W^\pi_M(\alpha)\cdot\tau -\pi(1)\cdot \tau\tau^\intercal\Rr) \d P_M(\tau) \d \fR(\alpha),
\end{align}
and we denote by $(\tau,\alpha)$ the canonical random variable under $ \la \cdot\ra_{\fR,\pi}$. 
We will show that the limit of $\E \la g(\tau\tau'^\intercal,\alpha\wedge\alpha')\ra^\orig_{N+M,x}$ is related to $\E \la g(\tau\tau'^\intercal,\alpha\wedge\alpha') \ra_{\fR, \pi}$ for some $\pi$.

\subsection{Cavity calculations per se}

\subsubsection{Free-energy cavity calculation}
\begin{proposition}[Free-energy cavity calculation]
\label{p.cavity_perturbation}
We have, uniformly over $M,N \in \N$ and $x \in [0,3]^{\N^4 \times \frac{D(D+1)}{2}}$,
\begin{align*}&-(N+M)\bar F_{N+M}^x(t,{q}) + N\bar F_N^x(t,{q}) = A_N(x)  + O\Ll(M^2N^{-\frac{1}{16}}\Rr).
\end{align*}
\end{proposition}

\begin{proof}
The argument consists in a rigorous justification of the expansion \eqref{e.expansion.HN} and of a similar expansion for $H_N$.
We keep writing $\rho = (\sigma,\tau) \in \R^{D\times (N+M)}$, so that $\rho\rho^\intercal = \sigma\sigma^\intercal + \tau\tau^\intercal$ and
\begin{align*}
    W^{q}_{N+M}(\alpha)\cdot \rho \stackrel{\d}{=} \tilde W^{q}_N(\alpha)\cdot \sigma + W^{q}_M(\alpha)\cdot \tau.
\end{align*}
Using the Taylor expansion of $\xi$, the local Lipschitzness of $\nabla \xi$, and $|\frac{\sigma\sigma^\intercal}{N}-\frac{\sigma\sigma^\intercal}{N+M}| = O(\frac{M}{N})$, we get, for every $\rho$ in the support of $P_{N+M}$,
\begin{align}\label{e.(N+M)xi(...)=}
\begin{split}
    (N+M)\xi\Ll(\frac{\rho\rho^\intercal}{N+M}\Rr) & = (N+M)\xi \Ll(\frac{\sigma\sigma^\intercal}{N+M}\Rr)+ \nabla\xi\Ll(\frac{\sigma\sigma^\intercal}{N}\Rr)\cdot \tau\tau^\intercal + O\Ll(\frac{M^2}{N}\Rr),
    \\
    N\xi\Ll(\frac{\sigma\sigma^\intercal}{N}\Rr) & = (N+M)\xi\Ll(\frac{\sigma\sigma^\intercal}{N+M}\Rr) + M \theta \Ll(\frac{\sigma\sigma^\intercal}{N}\Rr) + O\Ll(\frac{M^2}{N}\Rr).
\end{split}
\end{align}
These two relations yield, respectively,
\begin{align}\label{e.EHH=EHH}
\begin{split}
    \E H_{N+M}(\rho) H_{N+M}(\rho') &= \E \tilde H_N(\sigma)\tilde H_N(\sigma') + \E (\msf Z(\sigma)\cdot \tau)(\msf Z(\sigma')\cdot \tau') + O\Ll(\frac{M^2}{N}\Rr),
    \\
    \E H_N(\sigma)H_N(\sigma') &= \E\tilde H_N(\sigma)\tilde H_N(\sigma') + \E \msf Y(\sigma)\msf Y(\sigma') + O\Ll(\frac{M^2}{N}\Rr)
\end{split}
\end{align}
uniformly over $\rho$ and $\rho'$ in the support of $P_{N+M}$.
Recalling $H^{x}_N(\sigma,\alpha)$ and $\check H^{x}_N(\sigma,\alpha)$ in~\eqref{e.H^pert}, we verify that
\begin{multline}
\label{e.EHpertHpert=}
    (N+M)^{-\frac{1}{8}}\E \check H^{x}_{N+M}(\rho,\alpha) \check H^{x}_{N+M}(\rho',\alpha') 
    \\ =N^{-\frac{1}{8}} \E \check H^{x}_N(\sigma,\alpha) \check H^{x}_N(\sigma',\alpha') + O\Ll(MN^{-\frac{1}{8}}\Rr).
\end{multline}
Using again that $\rho\rho'^\intercal = \sigma\sigma'^\intercal + \tau\tau'^\intercal$, we see that the second term on the right-hand side in~\eqref{e.H^pert} satisfies
\begin{align}\label{e.sumxg(sigma)=}
    (N+M)^{-\frac{1}{16}}\sum_{i=1}^{D(D+1)/2} x_ie_i\cdot \rho\rho^\intercal = N^{-\frac{1}{16}}\sum_{i=1}^{D(D+1)/2} x_ie_i\cdot \sigma\sigma^\intercal+ O\Ll(MN^{-\frac{1}{16}}\Rr).
\end{align}
We want to insert the above relations into 
\begin{align*}
    &-(N+M)\bar F_{N+M}^x(t,{q}) + N\bar F_N^x(t,{q}) 
    \\ &= \E \log \iint \exp \Ll(H^{t,{q}}_{N+M}(\rho,\alpha)+(N+M)^{-\frac{1}{16}}H^{x}_{N+M}(\rho,\alpha) \Rr)\d P_{N+M}(\rho) \d \fR(\alpha)
    \\
    &- \E \log \iint \exp\Ll( H^{t,{q}}_{N}(\sigma,\alpha) +N^{-\frac{1}{16}}H^{x}_{N}(\sigma,\alpha)\Rr)\d P_{N}(\sigma) \d \fR(\alpha)=: \mathrm{I}_0 - \mathrm{II}_0 .
\end{align*}
Recalling $\la\cdot\ra^\cav_{N,x}$ in \eqref{e.<>^cav_N,x},
we rewrite $A_N(x)$ in \eqref{e.A_N(x)} as
\begin{align*}
    A_N(x)  = \E \log  \iint \exp\Ll(U(\sigma,\alpha,\tau)+ \tilde H^{t,{q}}_N(\sigma,\alpha) +  N^{-\frac{1}{16}}H^{x}_N(\sigma,\alpha)\Rr)\d P_{N+M}(\sigma,\tau) \d \fR(\alpha)
    \\
    - \E \log \iint\exp\Ll(\sqrt{2t}\msf Y(\sigma)- tM\theta \Ll( \frac{\sigma\sigma^\intercal}{N}\Rr) +\tilde H^{t,{q}}_N(\sigma,\alpha) +  N^{-\frac{1}{16}}H^{x}_N(\sigma,\alpha)\Rr)\d P_N(\sigma) \d \fR(\alpha) 
    \\
    =: \mathrm{I}_1 - \mathrm{II}_1.
\end{align*}
Note that all the Hamiltonians are of the form $\mathbf{G}(\mathbf{s}) + \mathbf{D}(\mathbf{s})$ where the first term $\mathbf{G}(\mathbf{s})$ is centered Gaussian, the second term $\mathbf{D}(\mathbf{s})$ is deterministic, and $\mathbf{s}$ is $(\rho,\alpha)$ or $(\sigma,\alpha)$. 
To compare $\mathrm{I}_0$ with $\mathrm{I}_1$ and $\mathrm{II}_0$ with $\mathrm{II}_1$, we consider two interpolations of the form
\begin{gather*}
    \varphi(r) := \E \log\int \exp \Ll(\sqrt{1-r}\mathbf{G}^{0}(\mathbf{s})+(1-r)\mathbf{D}^{0}(\mathbf{s})+\sqrt{r}\mathbf{G}^{1}(\mathbf{s})+r\mathbf{D}^{1}(\mathbf{s})\Rr)\d\mathfrak{P}(\mathbf{s}),
\end{gather*}
where $\mathbf{G}^{0}$ and $\mathbf{G}^{1}$ are independent and $\mathfrak{P}$ is $P_{N+M}\otimes \fR$ or $P_N\otimes \fR$.
For the first interpolation between $\mathrm{I}_0$ and $\mathrm{I}_1$, we take 
\begin{align}
\label{e.def.GD0}
    \mathbf{G}^{0}(\mathbf{s})+\mathbf{D}^{0}(\mathbf{s}) & = H^{t,{q}}_{N+M}(\rho,\alpha)+N^{-\frac{1}{16}}H^{x}_{N+M}(\rho,\alpha),
    \\
\label{e.def.GD1}
    \mathbf{G}^{1}(\mathbf{s})+\mathbf{D}^{1}(\mathbf{s})  & = U(\sigma,\alpha,\tau)+ \tilde H^{t,{q}}_N(\sigma,\alpha) +  N^{-\frac{1}{16}}H^{x}_N(\sigma,\alpha).
\end{align}
We understand that the Gaussian fields $H_{N+M}^x$ and $H_{N}^x$ appearing in \eqref{e.def.GD0} and \eqref{e.def.GD1} respectively are independent.
For the second interpolation, between $\mathrm{II}_0$ and $\mathrm{II}_1$, we take
\begin{align}
\label{e.def.GD0.bis}
    \mathbf{G}^{0}(\mathbf{s})+\mathbf{D}^{0}(\mathbf{s}) & = H^{t,{q}}_{N}(\sigma,\alpha) +N^{-\frac{1}{16}}H^{x}_{N}(\sigma,\alpha),
    \\
\label{e.def.GD1.bis}
    \mathbf{G}^{1}(\mathbf{s})+\mathbf{D}^{1}(\mathbf{s})  & = \sqrt{2t}\msf Y(\sigma)- tM\theta \Ll( \frac{\sigma\sigma^\intercal}{N}\Rr) +\tilde H^{t,{q}}_N(\sigma,\alpha) +  N^{-\frac{1}{16}}H^{x}_N(\sigma,\alpha).
\end{align}
To remain faithful with what we announced, we should take the Gaussian fields $H^x_N(\si, \al)$ appearing in \eqref{e.def.GD0.bis} and \eqref{e.def.GD1.bis} to be independent. This is of course not necessary here, and taking them to be equal is just as convenient; this can be phrased in our setting by absorbing this term into the definition of $\mathfrak{P}$.

We control $|\varphi(1)-\varphi(0)|$ by computing the derivative
\begin{align*}
    \frac{\d}{\d r}\varphi(r) = \E \la \frac{1}{2\sqrt{r}}\mathbf{G}^{1}(\mathbf{s}) - \frac{1}{2\sqrt{1-r}}\mathbf{G}^{0}(\mathbf{s}) + \mathbf{D}^{1}(\mathbf{s}) -\mathbf{D}^{0}(\mathbf{s})\ra_r,
\end{align*}
where $\la\cdot\ra_r$ is the Gibbs measure naturally associated with the free energy expression defining $\varphi(r)$.
We use Gaussian integration by parts (e.g.\ \cite[Theorem~4.6]{HJbook}) to get
\begin{align*}
    \frac{\d}{\d r}\varphi(r) = \E \la \mathbf{V}(\mathbf{s},\mathbf{s})-\mathbf{V}(\mathbf{s},\mathbf{s}')+ \mathbf{D}^{1}(\mathbf{s}) -\mathbf{D}^{0}(\mathbf{s})\ra_r,
\end{align*}
where $\mathbf{s}'$ is an independent copy of $\mathbf{s}$ under $\la\cdot\ra_r$, and
\begin{align}\label{e.V(s,s')=}
    \mathbf{V}(\mathbf{s},\mathbf{s}') := \frac{1}{2}\Ll(\E \mathbf{G}^{1}(\mathbf{s})\mathbf{G}^{1}(\mathbf{s}') - \E \mathbf{G}^{0}(\mathbf{s})\mathbf{G}^{0}(\mathbf{s}') \Rr).
\end{align}
We can then bound $\mathbf{V}$ and $\mathbf{D}^{1} -\mathbf{D}^{0}$ using the estimates in~\eqref{e.(N+M)xi(...)=},~\eqref{e.EHH=EHH},~\eqref{e.EHpertHpert=}, and~\eqref{e.sumxg(sigma)=}. 
For both interpolations, we have $|\frac{\d}{\d r}\varphi(r)|=O(M^2N^{-\frac{1}{16}})$ where $M^2$ comes from~\eqref{e.(N+M)xi(...)=} and~\eqref{e.EHH=EHH}, and $N^{-\frac{1}{16}}$ comes from~\eqref{e.EHpertHpert=} and~\eqref{e.sumxg(sigma)=}.
\end{proof}

\begin{lemma}[Cascade representation of $\mathscr P$]
\label{l.-MP^t,mu(zeta)}
For each ${p}\in\mathcal{Q}_\infty$, conditionally on $\fR$ and on the event where $\mfk U = \supp \fR$ satisfies \eqref{e.good.mfkU}, we define the following independent centered Gaussian processes  indexed by $\alpha \in \mfk U$:
\begin{itemize}
    \item let $Z(\alpha) = (Z_1(\alpha),\dots, Z_M(\alpha))$ be an $M$-tuple of independent $\R^\D$-valued centered Gaussian vectors $Z_i(\alpha)$ with covariance $\E Z_i(\alpha) Z_i(\alpha')^\intercal = \nabla\xi({p}\Ll(\alpha\wedge\alpha'\Rr))$;
    \item let $Y(\alpha)$ be real-valued with covariance $\E Y(\alpha)Y(\alpha') = M\theta({p}\Ll(\alpha\wedge\alpha'\Rr))$.
\end{itemize}
Letting also 
\begin{equation*}  V(\alpha,\tau):= \sqrt{2t} Z(\alpha)\cdot \tau - t \nabla\xi\Ll({p}(1)\Rr)\cdot \tau\tau^\intercal + \sqrt{2}W^{q}_M(\alpha)\cdot \tau - {q}(1)\cdot \tau\tau^\intercal,
\end{equation*}
we have
\begin{multline}
\label{e.-MP^t}
    -M\mathscr{P}_{t,{q}}({p})
    =\E \log  \iint \exp\Ll(V(\alpha,\tau)\Rr)\d P_M(\tau)\d\fR(\alpha)
    \\
      - \E \log \int \exp\Ll(\sqrt{2t}Y(\alpha)- tM\theta \Ll( {p}(1)\Rr)\Rr) \d\fR(\alpha).
\end{multline}
\end{lemma}
\begin{proof}
We denote the right-hand side of \eqref{e.-MP^t} by $\mathrm{I}-\mathrm{II}$.
We have
\begin{equation*}  \sqrt{2t}Z(\alpha)+\sqrt{2}W^{q}_M(\alpha) \stackrel{\d}{=} \sqrt{2}W^{{q} + t\nabla\xi({p})}_M(\alpha),
\end{equation*}
and thus
\begin{align*}
    \mathrm{I} = -M\bar F_M(0,{q} + t\nabla\xi({p})) =- M\psi(0,{q} + t\nabla\xi({p})),
\end{align*}
where we used Proposition~\ref{p.init.cond} in the last equality. 

To compute $\mathrm{II}$, we assume that ${p}$ is a step function; the general case can be recovered by approximation (one can verify that $\mathrm{II}$ is continuous in ${p}$ by a similar argument to that for the continuity in $q$ of $\bar F_N$ from Proposition~\ref{p.F_N_smooth}). Hence, we assume that ${p}=\sum_{l=1}^kp_l\1_{\Ll[s_{l-1},s_l\Rr)}$ for $0=s_0<s_1<\dots<s_k=1$ and $0\leq p_0\leq p_1\leq\cdots \leq p_k$. For brevity, we also write $\tilde \theta := 2tM\theta$. Following the computation of the second term in \cite[Lemma~3.1]{pan} (comparing the second term in \cite[(3.11)]{pan} with that in \cite[(3.15)]{pan}; substituting $s_l, p_l, \tilde\theta$ for $\zeta_p, q_p, \theta$ therein), we obtain that 
\begin{align*}
    &\E \log \int \exp\sqrt{2t}Y(\alpha) \d\fR(\alpha) = \frac{1}{2}\sum_{ l=0}^{ k-1}s_l\left(\tilde\theta(p_{l+1})-\tilde\theta(p_l)\right)
    \\
    &= \frac{1}{2}\left(-\sum_{l=1}^k (s_{l}-s_{l-1})\tilde\theta(p_l) + s_k\tilde\theta(p_k)-s_0\tilde\theta(p_0)\right)
     =\frac{1}{2}\left( -\int_0^1\tilde\theta({p}(s))\d s + \tilde\theta({p}(1))\right),
\end{align*}
which yields
\begin{align*}
    \mathrm{II}= \E \log \int \exp\sqrt{2t}Y(\alpha) \d\fR(\alpha) -tM\theta({p}(1)) = -tM\int_0^1\theta({p}(s))\d s.
\end{align*}
Comparing $\mathrm{I}$ and $\mathrm{II}$ with the definition of $\mathscr{P}_{t,q}(p)$ in~\eqref{e.mathscrP}, we get the desired result.
\end{proof}

\subsubsection{Gibbs-average cavity calculation}
Recall $\la\cdot\ra^\orig_{N,x}$ in~\eqref{e.<>^orig_Nx=} and $\la\cdot\ra^\star_{N+M,x}$ in~\eqref{e.<>_N^star}.
\begin{proposition}[Gibbs-average cavity calculation]
\label{l.pre_comp_2}
For every continuous function $g:\R^{\D\times\D}\times \R\to\R$, we have uniformly over $M, N \in \N$ and $x \in [0,3]^{\N^4 \times \frac{D(D+1)}{2}}$ that
\begin{align*}
    \E \la g\Ll(\tau\tau'^\intercal,\alpha\wedge\alpha'\Rr)\ra^\orig_{N+M,x} = \E \la g\Ll(\tau\tau'^\intercal,\alpha\wedge\alpha'\Rr)\ra^\star_{N+M,x} + O\Ll(M^2 N^{-\frac{1}{16}}\Rr) .
\end{align*}
\end{proposition}

\begin{proof}
We use the estimates in the proof of Proposition~\ref{p.cavity_perturbation} and a similar interpolation computation.
We write $\Theta := g\Ll(\tau\tau'^\intercal,\alpha\wedge\alpha'\Rr)$, $\mathbf{s}:=(\rho,\alpha)$, and $\mathfrak{P} := P_{N+M}\otimes \fR$. 
We rewrite $\la\cdot\ra^\orig_{N+M,x}$ and $\la\cdot\ra^\star_{N+M,x}$ as follows,
\begin{gather*}
    \la\cdot\ra^\orig_{N+M,x} \propto \exp\Ll(\mathbf{G}^0(\mathbf{s}) + \mathbf{D}^0(\mathbf{s})\Rr) \d \mathfrak{P}(\mathbf{s}),
    \\
    \la\cdot\ra^\star_{N+M,x} \propto \exp\Ll(\mathbf{G}^1(\mathbf{s}) + \mathbf{D}^1(\mathbf{s})\Rr) \d \mathfrak{P}(\mathbf{s}),
\end{gather*}
where $\mathbf{G}^i(\mathbf{s})$ collects centered Gaussian terms, $\mathbf{D}^i(\mathbf{s})$ collects deterministic terms in the respective Hamiltonians, and they satisfy
\begin{align*}
\mathbf{G}^0(\mathbf{s}) + \mathbf{D}^0(\mathbf{s}) &= H^{t,{q}}_{N+M}(\rho,\alpha) + (N+M)^{-\frac{1}{16}}H^{x}_{N+M}(\rho,\alpha),
\\
\mathbf{G}^1(\mathbf{s}) + \mathbf{D}^1(\mathbf{s}) &= \tilde H_N^{t,{q}}(\sigma,\alpha) +  U(\sigma,\alpha,\tau)
    + N^{-\frac{1}{16}}  H^{x}_N(\sigma,\alpha).
\end{align*}
Again we understand that $H^x_{N+M}$ and $H^x_N$ are independent here. 
For every $r\in[0,1]$, we define the interpolating Gibbs measure
\begin{align*}
    \la\cdot\ra_r \propto \exp\Ll(\sqrt{1-r}\mathbf{G}^0(\mathbf{s})+\sqrt{r}\mathbf{G}^1(\mathbf{s})+(1-r)\mathbf{D}^0(\mathbf{s})+r \mathbf{D}^1(\mathbf{s})\Rr)\d \mathfrak{P}(\mathbf{s}),
\end{align*}
and we set $\varphi(r) := \E \la \Theta\ra_r$. We aim to show that $|\frac{\d}{\d r}\varphi(r)|=o(1)$ uniformly in $r$. Notice that $\Theta$ depends on two independent copies of $\mathbf{s}$, which we  denote by $\mathbf s^1$ and $\mathbf s^2$. We also give ourselves $\mathbf{s}^3$ and $\mathbf{s}^4$ two additional independent copies of $\mathbf{s}$. We write the Gibbs measure in the definition of $\varphi(r)$ explicitly as a ratio and differentiate to obtain that $\frac{\d}{\d r} \varphi(r)$ is the $\E$-expectation of the $\langle \cdot \rangle_r$-covariance between $\Theta$ and 
\begin{equation*}  \frac{1}{2\sqrt{r}}\big(\mathbf{G}^{1}(\mathbf{s}^1) + \mathbf{G}^{1}(\mathbf{s}^2)\big) - \frac{1}{2\sqrt{1-r}}\big(\mathbf{G}^{0}(\mathbf{s}^1) + \mathbf{G}^0(\mathbf{s}^2)\big) + \mathbf{D}^{1}(\mathbf{s}^1)+ \mathbf{D}^{1}(\mathbf{s}^2) -\mathbf{D}^{0}(\mathbf{s}^1) - \mathbf{D}^{0}(\mathbf{s}^2).
\end{equation*}
We can rewrite this covariance as a Gibbs average by utilizing the additional variable $\mathbf{s}^3$. We next perform a Gaussian integration by parts as in e.g.\ \cite[Theorem~4.6]{HJbook} to obtain a Gibbs average of the form
\begin{align*}
    \frac{\d}{\d r}\varphi(r) = \E \la \Theta\Ll(\sum_{i,j=1}^4 c_{i,j}\mathbf{V}^{i,j}+ \sum_{i=1}^3 c_i \mathbf{U}^i\Rr)\ra_r,
\end{align*}
where
\begin{align*}
    \mathbf{V}^{i,j} &= \E \Ll[\mathbf{G}^1\Ll(\mathbf{s}^i\Rr)\mathbf{G}^1\Ll(\mathbf{s}^j\Rr) - \mathbf{G}^0\Ll(\mathbf{s}^i\Rr)\mathbf{G}^0\Ll(\mathbf{s}^j\Rr)\Rr],\qquad \mathbf{U}^i = \mathbf{D}^1\Ll(\mathbf{s}^i\Rr) - \mathbf{D}^0\Ll(\mathbf{s}^i\Rr),
\end{align*}
and $(c_{i,j})_{i,j=1}^4$ and $(c_i)_{i=1}^3$ are absolute constants. 
Then, we can use the estimates~\eqref{e.(N+M)xi(...)=}, \eqref{e.EHH=EHH}, \eqref{e.EHpertHpert=}, and~\eqref{e.sumxg(sigma)=} to see that $|\mathbf{V}^{i,j}|$ and $|\mathbf{U}^i|$ are $O\Ll(M^2N^{-\frac{1}{16}}\Rr)$. Since $\Theta$ is also bounded a.s., we deduce $|\frac{\d}{\d r}\varphi(r)|=O\Ll(M^2N^{-\frac{1}{16}}\Rr)$ uniformly in $r$. The desired result thus follows.
\end{proof}

\subsubsection{Other useful interpolation estimates} We now collect two more useful estimates that are obtained using similar arguments. Recall $\bar F^x_N(t,{q})$ in~\eqref{e.F^x_N=} and $\tilde F^x_N(t,{q})$ in~\eqref{e.tildeF_N=}.
\begin{lemma}[Control of perturbative part]
\label{l.limF^x-F}
For each fixed $M \in \N$, we have
\begin{align*}
    \lim_{N\to\infty} \sup_{x} \Ll|\bar F^x_N(t,{q}) - \bar{F}_N(t,{q})\Rr| =0,\qquad \lim_{N\to\infty} \sup_{x} \Ll|\tilde F^x_N(t,{q}) - \bar F_N(t,{q})\Rr| =0 ,
\end{align*}
where the suprema are over $x \in [0,3]^{\N^4 \times \frac{D(D+1)}{2}}$.
\end{lemma}

\begin{proof}
We use a similar interpolation as in the proof of Proposition~\ref{p.cavity_perturbation}. We recall the normalizing factor $\frac{1}{N}$ in the definition of the free energy.
The first convergence follows from the interpolation and the presence of $N^{-\frac{1}{16}}$ in front of $H^{x}_N$. The second convergence follows from these facts and, additionally, the second line in~\eqref{e.EHH=EHH} together with the boundedness of $\E \msf Y(\sigma)\msf Y(\sigma')$.
\end{proof}

\begin{lemma}[Regularity in $N$ of the free energy]
\label{l.NF_N-(N+M)F_(N+M)}
There exists $C < +\infty$ such that for every $N, M \in \N$, $t \ge 0$ and $q \in \mcl Q_1$, 
\begin{align*}
    \Ll|N\bar F_N(t,{q})- (N+M)\bar F_{N+M}(t,{q})\Rr|\leq CM\Ll(t + |{q}|_{L^1}\Rr).
\end{align*}
\end{lemma}

\begin{proof}
Let us first assume that ${q} \in \C_\infty$.
As before, we write $\rho =(\sigma,\tau)$, and carry out an interpolation computation.
Let $\mathbf{s}:=(\rho,\alpha)$, and $\mathfrak{P} := P_{N+M}\otimes \fR$.
We set
\begin{align*}
\mathbf{G}^0(\mathbf{s}) + \mathbf{D}^0(\mathbf{s}) &= H^{t,{q}}_{N+M}(\rho,\alpha) ,\quad 
\mathbf{G}^1(\mathbf{s}) + \mathbf{D}^1(\mathbf{s}) =  H^{t,{q}}_{N}(\sigma,\alpha)
\end{align*}
where $\mathbf{G}^i(\mathbf{s})$ collects centered Gaussian terms and $\mathbf{D}^i(\mathbf{s})$ collects deterministic terms.
For every $r\in[0,1]$, we set
\begin{gather*}
    \varphi(r) := \E \log\int \exp \Ll(\sqrt{1-r}\mathbf{G}^{0}(\mathbf{s})+(1-r)\mathbf{D}^{0}(\mathbf{s})+\sqrt{r}\mathbf{G}^{1}(\mathbf{s})+r\mathbf{D}^{1}(\mathbf{s})\Rr)\d\mathfrak{P}(\mathbf{s}).
\end{gather*}
We have $\varphi(0) = (N+M)\bar F_{N+M}(t,{q})$ and $\varphi(1) = N\bar F_N(t,{q})$.
Again, using Gaussian integration by parts, we get
\begin{align*}
    \frac{\d}{\d r}\varphi(r) = \E \la \mathbf{V}(\mathbf{s},\mathbf{s})-\mathbf{V}(\mathbf{s},\mathbf{s}')+ \mathbf{D}^{1}(\mathbf{s}) -\mathbf{D}^{0}(\mathbf{s})\ra_r
\end{align*}
for $\mathbf{V}$ given in \eqref{e.V(s,s')=}, and where $\la\cdot\ra_r$ is the corresponding Gibbs measure. 
This time, we bound $\mathbf{V}$ and $\mathbf{D}^{1} -\mathbf{D}^{0}$ using
\begin{align*}
    Nt\xi\Ll(\frac{\sigma\sigma'^\intercal}{N}\Rr) & = (N+M)t\xi\Ll(\frac{\rho\rho'^\intercal}{N+M}\Rr) + O(M)t,
    \\
    {q}\Ll(\alpha\wedge\alpha'\Rr)\cdot \sigma\sigma'^\intercal &= {q}\Ll(\alpha\wedge\alpha'\Rr)\cdot \rho\rho'^\intercal + O\Ll(M\Rr)\Ll|{q}\Ll(\alpha\wedge\alpha'\Rr)\Rr|
\end{align*}
uniformly in $N$. We get
\begin{align*}
    \Ll|\frac{\d}{\d r}\varphi(r) \Rr|\leq O(M)\Ll(1 + \E \la \Ll|{q}\Ll(\alpha\wedge\alpha'\Rr)\Rr|\ra_r\Rr) = O(M)\Ll(t+|q|_{L^1}\Rr)
\end{align*}
and the announced result follows for ${q}\in \C_\infty$. The full result follows by approximation using Proposition~\ref{p.barF_N_Lip}.
\end{proof}

\subsection{Overlap approximations}
Recall the notation $\approx_\eps$ in \eqref{e.approx_eps}.

\subsubsection{Overlap approximations for the free-energy cavity calculation}
We show that $A_N(x)$ in \eqref{e.A_N(x)} and $-M\mathscr{P}_{t,q}({p})$ in \eqref{e.mathscrP} can be approximated by continuous functions of finitely many entries in the overlap arrays. Recall $R^{\leq n}_N$ in~\eqref{e.overlap_notation}, $Q^{\leq n}_{p}$ in~\eqref{e.overlap_alpha},  $L^\infty_{\leq 1}$ in \eqref{e.L^infty<1}, $\la\cdot\ra^\cav_{N,x}$ in~\eqref{e.<>^cav_N,x}, and that $\la\cdot\ra_\fR$ denotes the expectation with respect to the measure $\fR^{\otimes \N}$. 

\begin{lemma}\label{l.approx_finite_overlap}
For every $\eps>0$, there exists $n\in\N$ and a bounded continuous function $F_\eps : (\R^{D\times D}\times \R^{D\times D})^{n\times n}\to \R $ such that the following two properties hold:
\begin{itemize}
    \item $A_N(x) \approx_\eps \E \la F_\eps\Ll(R^{\leq n}_{N},q\Ll(R^{\leq n}_{\alpha}\Rr)\Rr) \ra^\cav_{N,x}$ uniformly over $N \in \N$ and $x \in [0,3]^{\N^4 \times \frac{D(D+1)}{2}}$;
    \item $-M\mathscr{P}_{t,{q}}({p})\approx_\eps \E \la F_\eps \Ll(Q^{\leq n}_p,q\Ll(R^{\leq n}_\alpha\Rr)\Rr) \ra_\fR$ uniformly over ${p} \in \C\cap L^\infty_{\leq 1}$.
\end{itemize}
\end{lemma}
In the above $q$ acts entrywise (as $p$ does in the definition of $Q_p$).
\begin{proof}
We show that $A_N(x)$ and $-M\mathscr{P}_{t,{q}}({p})$ have similar overlap structures as defined in Definition~\ref{d.os}, which will allow us to apply Proposition~\ref{p.approx_os}. Let $\Gamma$ be a probability measure on some separable Hilbert space $\mathcal{X}$, and let $\mathcal{R} =(\mathcal{R}_1,\mathcal{R}_2) :\mathcal{X}\times \mathcal{X}\to \R^{D\times D}\times \R$ be a measurable function satisfying $|\mathcal{R}_1|,|\mathcal{R}_2|\leq 1$.

Assume that there are independent centered Gaussian processes indexed by $\rho\in \supp\Gamma$:
\begin{itemize}
    \item $Z(\rho) = (Z_1(\rho),\dots, Z_M(\rho))$ is an $M$-tuple of independent $\R^\D$-valued centered Gaussian vectors $z_i(\rho)$ with covariance $\E Z_i(\rho)Z_i(\rho')^\intercal = \nabla \xi\Ll(\mathcal{R}_1(\rho,\rho')\Rr)$;
    \item $W(\rho)=(w_1(\rho),\dots,w_M(\rho))$ is an $M$-tuple of independent $\R^\D$-valued centered Gaussian vectors $w_i(\rho)$ with covariance $\E w_i(\rho)w_i(\rho')^\intercal = {q} \Ll(\mathcal{R}_2(\rho,\rho')\Rr)$;
    \item $Y(\rho)$ is real-valued with covariance $\E Y(\rho)Y(\rho') = M\theta (\mathcal{R}_1(\rho,\rho'))$.
\end{itemize}
We also assume a stochastic continuity property of these processes so that we can construct jointly measurable versions of these as in Proposition~\ref{p.gaussian_cascade}. 
Writing $\la\cdot\ra_\Gamma$ to denote the expectation with respect to $\Gamma ^{\otimes \N}$, we consider
\begin{align*}
    \mathcal{F}_1(\Gamma,\mathcal{R}) & :=\E \log \bigg\langle\int \exp\Big(\sqrt{2t} Z(\rho)\cdot \tau - t \nabla\xi\Ll(\mathcal{R}_1(\rho,\rho)\Rr)\cdot \tau\tau^\intercal 
    \\
    &\qquad \qquad+ \sqrt{2}W(\rho)\cdot \tau - {q}(\mathcal{R}_2(\rho,\rho))\cdot \tau\tau^\intercal\Big)\d P_M(\tau)\bigg\rangle_\Gamma,
    \\
    \mathcal{F}_2(\Gamma,\mathcal{R}) &:= \E \log \la \exp\Ll(\sqrt{2t}Y(\rho)- tM\theta \Ll( \mathcal{R}_1(\rho,\rho)\Rr)\Rr)\ra_\Gamma,
\end{align*}
where $\E$ integrates the Gaussian randomness in $Z(\rho)$, $W(\rho)$ and $Y(\rho)$.
Comparing with the expression for $A_N(x)$ in~\eqref{e.A_N(x)}, we see that $A_N(x)$ is the average of $\mathcal{F}_1(\Gamma,\mathcal{R})-\mathcal{F}_2(\Gamma,\mathcal{R})$ with $(\frac{\sigma}{\sqrt{N}},\alpha)$, $R_N$, and $\la\cdot\ra^\cav_{N,x}$ substituted for $\rho$, $\mathcal{R}$, and $\la\cdot\ra_\Gamma$ (this further averaging is with respect to the randomness intrinsic to $\la \cdot \ra^\cav_{N,x}$ itself).
Using Lemma~\ref{l.-MP^t,mu(zeta)}, we see that $-M\mathscr{P}_{t,{q}}({p})$ is the average of $\mathcal{F}_1(\Gamma,\mathcal{R})-\mathcal{F}_2(\Gamma,\mathcal{R})$ with $\alpha$, $Q_{p}$, and $\fR$ substituted for $\rho$, $\mathcal{R}$, and $\Gamma $.

We will apply Proposition~\ref{p.approx_os} to $\mathcal{F}_1(\Gamma,\mathcal{R})$ and $\mathcal{F}_2(\Gamma,\mathcal{R})$. For $\mathcal{F}_1(\Gamma,\mathcal{R})$, the corresponding function $C$ is $C_1:\mathcal{R}\mapsto 2t\nabla\xi(\mathcal{R}_1)+2q(\mathcal{R}_2)$. One can also identify the corresponding Lipschitz function $\mathbf{g}$ as an integral in $\tau$. Note that there are $M$ independent copies of the $\R^\D$-valued processes here. Proposition~\ref{p.approx_os} is still applicable and we can approximate
\begin{align*}
    \mathcal{F}_1(\Gamma,\mathcal{R}) \approx_{\frac{\eps}{2}} \E \la F_{1,\eps}(C_1(\mathcal{R}^{\leq n_1}))\ra_\Gamma
\end{align*}
for some $n_1$ and some continuous $F_{1,\eps}$, where $C_1$ acts entrywise.

For $\mathcal{F}_2(\Gamma,\mathcal{R})$, the corresponding function $C$ is $C_2:\mathcal{R}\mapsto M\theta(\mathcal{R}_1)$ and $\mathbf{g}$ is linear. Proposition~\ref{p.approx_os} is directly applicable and we obtain the approximation
\begin{align*}
    \mathcal{F}_2(\Gamma,\mathcal{R}) \approx_{\frac{\eps}{2}} \E \la F_{2,\eps}(C_2(\mathcal{R}^{\leq n_2}))\ra_\Gamma
\end{align*}
for some $n_2$ and some continuous $F_{2,\eps}$, where $C_2$ acts entrywise.

Since both $\nabla\xi$ and $\theta$ are continuous, we can absorb them into continuous functions to get
\begin{align*}
    \mathcal{F}_1(\Gamma,\mathcal{R}) -\mathcal{F}_1(\Gamma,\mathcal{R})\approx_\eps \E \la F_\eps(\mathcal{R}^{\leq n}_1,q(\mathcal{R}^{\leq n}_2))\ra_\Gamma.
\end{align*}
We emphasize that $q$ cannot be absorbed because it is in general not continuous.
For simplicity, we also add the variable $\mathcal{R}_2$ (which forms $\mathcal{R}$ together with $\mathcal{R}_1$) and express the approximating function as $F_\eps(\mathcal{R}^{\leq n},q(\mathcal{R}^{\leq n}_2))$.
In view of \eqref{e.overlap_notation} and \eqref{e.overlap_alpha}, for $\mathcal{R}$ being $R_N$ or $Q_p$, we have $\mathcal{R}_2 = R_\alpha$ in both cases. Hence, we arrive at the approximations of the announced form.
\end{proof}

\subsubsection{Overlap approximations for the Gibbs-average cavity calculation}

In the same fashion, we want to approximate $\E \la g(\tau\tau'^\intercal, \alpha\wedge\alpha')\ra^\star_{N+M,x}$ and $\E \la g(\tau\tau'^\intercal, \alpha\wedge\alpha')\ra_{\fR,{q}+t\nabla\xi({p})}$, where the two Gibbs measures are defined in \eqref{e.<>_N^star} and~\eqref{e.<>_R,pi}, respectively. Also, recall the definition of $\la\cdot\ra^\cav_{N,x}$ in \eqref{e.<>^cav_N,x}, and that $\la \cdot\ra_\fR$ denotes the expectation with respect to~$\fR^{\otimes \N}$.

\begin{lemma}\label{l.approx_finite_overlap_2}
For every bounded continuous function $g:\R^{D\times D}\times \R\to\R$ and every $\eps>0$, there is $n\in\N$ and a bounded continuous function $F_\eps : (\R^{D\times D}\times \R^{D\times D}\times \R)^{n\times n}\to \R $ such that the following two properties hold:
\begin{itemize}
    \item $\E \la g(\tau\tau'^\intercal, \alpha\wedge\alpha')\ra^\star_{N+M,x} \approx_\eps \E \la F_\eps\Ll(R^{\leq n}_{N},q\Ll(R^{\leq n}_{\alpha}\Rr)\Rr) \ra^\cav_{N,x}$ uniformly over $N,x$;
    \item $\E \la g(\tau\tau'^\intercal, \alpha\wedge\alpha')\ra_{\fR,{q}+t\nabla\xi({p})}\approx_\eps \E \la F_\eps \Ll(Q^{\leq n}_p,q\Ll(R^{\leq n}_\alpha\Rr)\Rr) \ra_\fR$ uniformly over ${p} \in \C\cap L^\infty_{\leq 1}$.
\end{itemize}
\end{lemma}

Again, $q$ acts entrywise.

\begin{proof}
Comparing the definition of $\la\cdot\ra^\star_{N+M,x}$ and that of $\la\cdot\ra^\cav_{N,x}$, we can rewrite
\begin{align}
    &\E \la g\Ll(\tau\tau'^\intercal, \alpha\wedge\alpha'\Rr)\ra^\star_{N+M,x} \notag
    \\
    &= \E \frac{\la \iint g\Ll(\tau\tau'^\intercal, \alpha\wedge\alpha'\Rr) e^{U(\sigma,\alpha,\tau)+U(\sigma',\alpha',\tau')} \d P_1(\tau) \d P_1(\tau')\ra^\cav_{N,x}}{\la \iint e^{U(\sigma,\alpha,\tau)+U(\sigma',\alpha',\tau')} \d P_1(\tau)\d P_1(\tau')\ra^\cav_{N,x} } \label{e.<>/<>_1}
\end{align}
for $U$ given in~\eqref{e.U=}.
Similarly, we can rewrite
\begin{align}&\E \la g\Ll(\tau\tau'^\intercal, \alpha\wedge\alpha'\Rr)\ra_{\fR,{q}+t\nabla\xi({p})}\notag
    \\
    &= \E \frac{\la \iint g\Ll(\tau\tau'^\intercal, \alpha\wedge\alpha'\Rr)e^{V(\alpha,\tau)+V(\alpha',\tau')}\d P_1(\tau) \d P_1(\tau')\ra_\fR}{\la \iint e^{V(\alpha,\tau)+V(\alpha',\tau')}\d P_1(\tau) \d P_2(\tau')\ra_\fR}   \label{e.<>/<>_2}
\end{align}
where
$V(\alpha,\tau)$ is in Lemma~\ref{l.-MP^t,mu(zeta)} and satisfies $V(\alpha,\tau) {=} \sqrt{2t} Z(\alpha)\cdot \tau - t \nabla\xi\Ll({p}(1)\Rr)\cdot \tau\tau^\intercal + \sqrt{2}W^{q}_M(\alpha)\cdot \tau - {q}(1)\cdot \tau\tau^\intercal$.

In the formalism of Definition~\ref{d.os}, we can see that~\eqref{e.<>/<>_1} and~\eqref{e.<>/<>_2} are expressed in the same overlap and covariance structure. Here, $\Gamma$ corresponds to $\la\cdot\ra^\cav_{N,x}$ and $\fR$; $\mathcal{R}=(\mathcal{R}_1,\mathcal{R}_2)$ (which is $\R^{\D\times\D}\times\R$-valued) corresponds to $R_N$ and $Q_p$; and $C$ is $\mathcal{R}\mapsto 2t\nabla\xi(\mathcal{R}_1)+2q(\mathcal{R}_2)$.

To apply Proposition~\ref{p.approx_os}, let us first assume that $g$ does not depend on $\tau\tau'^\intercal$.  Then, we can identify the function $\mathbf{g}$ as an integral in $\tau$ and set $\mathbf{h}=g$ so that both~\eqref{e.<>/<>_1} and~\eqref{e.<>/<>_2} can be rewritten as the left-hand side of~\eqref{e.2nd_approx_os}. Applying Proposition~\ref{p.approx_os}, we can approximate \eqref{e.<>/<>_1} and \eqref{e.<>/<>_1} by
\begin{align*} 
    \E \la g\Ll(\mathcal{R}^{1,2}_2\Rr) F_\eps \Ll(C(\mathcal{R}^{\leq n})\Rr) \ra_\Gamma
\end{align*}
for some $n\in\N$ and continuous $F_\eps$, where we used the observation that $\mathcal{R}_2=R_\alpha$ for $\mathcal{R}$ being either $R_N$ or $Q_p$ (see~\eqref{e.overlap_notation} and~\eqref{e.overlap_alpha}). Since the functions $g$ and $\nabla\xi$ in the definition of $C$ are continuous, we can absorb them into $F_\eps$ so that the above approximating term becomes
\begin{align*} 
    \E \la F_\eps \Ll(\mathcal{R}^{\leq n},q\Ll(\mathcal{R}_2^{\leq n}\Rr)\Rr) \ra_\Gamma,
\end{align*}
as announced. Since $q$ is not continuous in general, it cannot be absorbed.

If $g$ depends on $\tau\tau'^\intercal$, then Proposition~\ref{p.approx_os} is not directly applicable. Nevertheless, we can easily modify the approximation argument in its proof to get the same result.
\end{proof}

\subsection{Ghirlanda--Guerra identities}

Recall the notation of overlaps in~\eqref{e.overlap_notation}.
We denote by $\E_x$ the expectation under which $x$ is an i.i.d.\ sequence of uniform random variables over $[1,2]$.
For $N\in\N$, integer $n\geq 2$, $h\in\N^4$, and a bounded measurable function $\bff:\Ll(\R^{\D\times \D} \times \R\Rr)^{n\times n}\to \R$, we define, for $\la\cdot\ra^\cav = \la \cdot \ra^\cav_{N,x}$ in~\eqref{e.<>^cav_N,x},
\begin{align}\label{e.Delta^x}
    \Delta^x_N(\bff,n,h) = \Ll|\E \la \bff\Ll(R^{\leq n}_N\Rr)R^{1,n+1}_{N,h} \ra^\cav - \frac{1}{n}\E\la \bff\Ll(R^{\leq n}_N\Rr)\ra^\cav \E \la  R^{1,2}_{N,h}\ra^\cav - \frac{1}{n}\sum_{l=2}^n \E \la \bff\Ll(R^{\leq n}_N\Rr) R^{1,l}_{N,h}\ra^\cav\Rr|.
\end{align}
In \eqref{e.Delta^x} and throughout, $\E$ integrates the Gaussian randomness in the Hamiltonian and the randomness in $\fR$. In particular, it does not integrate $x$. 

\begin{proposition}\label{p.perturbation}
The following holds:
\begin{enumerate}
    \item \label{i.self-overlap_concent} $\lim_{N\to\infty} \E_x \E \la \Ll|R^{1,1}_{N,\sigma} - \E \la R^{1,1}_{N,\sigma} \ra^\cav_{N,x}\Rr|\ra^\cav_{N,x} =0$;
    \item \label{i.GG} $\lim_{N\to\infty} \E_x \Delta^x_N(\bff,n,h)=0$ for every $n$, $h$, and $\bff$.
\end{enumerate}

\end{proposition}

We start with a lemma which leads to part~\eqref{i.self-overlap_concent} of the proposition. We recall the notation for the coordinates of $x$ in \eqref{e.x_pert}. 

\begin{lemma}
There exists a constant $C < +\infty$ not depending on $x$ such that, for every $i\in \{1,\dots,D(D+1)/2\}$,
\begin{align*}
    \int_1^2\E \la\Ll|e_i\cdot \sigma\sigma^\intercal - \E \la e_i\cdot \sigma\sigma^\intercal\ra^\cav_{N,x}\Rr| \ra^\cav_{N,x} \d x_i \leq CN^{\frac{25}{32}}.
\end{align*}
\end{lemma}
\begin{proof}
We write $\la\cdot\ra=\la \cdot\ra^\cav_{N,x}$ for brevity.
Fixing any $x$ except for the entry $x_i$,
we set
\begin{align*}
    \varphi(x_i) := \log \iint \exp\Ll(\tilde H^{t,{q}}_N(\sigma,\alpha) +  N^{-\frac{1}{16}}H^{x}_N(\sigma,\alpha)\Rr) \d P_N(\sigma) \d \fR(\alpha)
\end{align*}
and $\phi(x_i) :=\E \varphi(x_i)$.
We compute the derivatives:
\begin{gather*}
    \varphi'(x_i)  = N^{-\frac{1}{16}} \la e_i\cdot \sigma\sigma^\intercal\ra, \qquad \varphi''(x_i) = N^{-\frac{1}{8}}  \la (e_i\cdot \sigma\sigma^\intercal)^2 - \la e_i\cdot \sigma\sigma^\intercal\ra^2 \ra,
    \\
    \phi'(x_i)  = N^{-\frac{1}{16}} \E \la e_i\cdot \sigma\sigma^\intercal\ra,\qquad 
    \phi''(x_i)  = N^{-\frac{1}{8}} \E \la (e_i\cdot \sigma\sigma^\intercal)^2 - \la e_i\cdot \sigma\sigma^\intercal\ra^2 \ra.
\end{gather*}
Since $|e_i\cdot \sigma\sigma^\intercal|\leq N$, we have, for all $x_i\in[0,3]$,
\begin{align}\label{e.|phi'|<}
    |\phi'(x_i)|\leq N^\frac{15}{16}.
\end{align}
Integrating $\phi''(x_i)$ over $x_i\in[1,2]$ yields
\begin{align*}
    \int_1^2 \E \la (e_i\cdot \sigma\sigma^\intercal)^2 - \la e_i\cdot \sigma\sigma^\intercal\ra^2 \ra \, \d x_i=N^\frac{1}{8}(\phi'(2)-\phi'(1)) \leq 2N^\frac{17}{16}.
\end{align*}
By a  concentration argument such as \cite[Theorem~1.2]{pan}, we have
\begin{align*}
    \sup_{x}\E |\varphi-\E \varphi| \leq CN^\frac{1}{2}
\end{align*} 
for some constant $C< +\infty$. Defining
\begin{align*}
    \delta := |\varphi(x_i+y_i) -\phi(x_i+y_i)| + |\varphi(x_i-y_i)-\phi(x_i-y_i)| + |\varphi(x_i)-\phi(x_i)|,
\end{align*}
we thus have $\int_1^2\delta \d x_i\leq CN^\frac{1}{2}$ for all $y_i\in[0,1]$. By H\"older's inequality (as in e.g.\ \cite[Exercise~2.6]{HJbook}), the functions $\varphi$ and $\phi$ are convex. By \cite[Lemma~3.2]{pan}, we can obtain from basic properties of convex functions that
\begin{align*}
    |\varphi'(x_i) - \phi'(x_i)|\leq \phi'(x_i+y_i) - \phi'(x_i-y_i) + \frac{\delta}{y_i}.
\end{align*}
Using~\eqref{e.|phi'|<} and the mean value theorem, we get
\begin{align*}
    \int_1^2 \Ll(\phi'(x_i+y_i)-\phi'(x_i-y_i)\Rr)\d x_i = \phi(2+y_i) -\phi(2-y_i) - \phi(1+y_i) +\phi(1-y_i) \leq 4 N^\frac{15}{16} y_i.
\end{align*}
Therefore,
\begin{align*}
    \int_1^2 \E|\varphi'(x_i) - \phi'(x_i)| \d x_i \leq 4 N^\frac{15}{16} y_i + \frac{CN^\frac{1}{2}}{y_i}.
\end{align*}
Inserting the expressions of $\varphi'$ and $\phi'$ and setting $y_i= N^{-\frac{7}{32}}$, we get the desired result. 
\end{proof}

\begin{proof}[Proof of Proposition~\ref{p.perturbation}]

We write $\la\cdot\ra = \la\cdot\ra^\cav_{N,x}$ for brevity.
The previous lemma implies
\begin{align*}
    \int_1^2 \E \la   \Ll|e_i\cdot\frac{\sigma\sigma^\intercal }{N} - \E \la e_i\cdot \frac{\sigma\sigma^\intercal }{N} \ra \Rr|\ra\d x_i \leq C N^{-\frac{7}{32}}
\end{align*}
for every $i \in \{1,\ldots D(D+1)/2\}$.
Since $\{e_i\}_{i=1}^{D(D+1)/2}$ is an orthonormal basis, we deduce that
\begin{align*}
    \lim_{N\to\infty} \int_{\Ll[1,2\Rr]^{D(D+1)/2}} \E \la \Ll|\frac{\sigma\sigma^\intercal }{N} - \E \la \frac{\sigma\sigma^\intercal }{N} \ra \Rr|\ra\d x=0,
\end{align*}
which proves Proposition~\ref{p.perturbation}~\eqref{i.self-overlap_concent}.

To prove the second part, we proceed as in \cite[Theorem~3.3]{pan} (with $h$, $N^\frac{7}{16}$, $N^\frac{1}{2}$, $N^{-\frac{1}{2}}x_hc_h H^h_N(\sigma,\alpha)$ substituted for $p$, $s$, $\nu_N(s)$, $g_p(\sigma)$ therein) to get that for every $h$, there is a constant $C_h$ such that
\begin{align*}
    \int_1^2 \E \la \Ll| H^h_N -\E \la H^h_N\ra\Rr|\ra \d x_h \leq C_h N^\frac{3}{4}.
\end{align*}
Let $\bff$ be bounded and measurable and write $\bff = \bff(R^{\leq n})$. Without loss of generality, we assume $\|\bff\|_\infty\leq 1$. Then
\begin{align*}
    \Ll|\E\la \bff H^h_N\ra - \E \la \bff\ra \E \la H^h_N\ra\Rr| \leq  \E \la \Ll|H^h_N - \E \la H^h_N\ra \Rr|\ra.
\end{align*}
For brevity, we omit the subscript $N$ from overlaps.
By the Gaussian integration by parts in e.g.\ \cite[Theorem~4.6]{HJbook}, we have
\begin{align*}
    &\E\la \bff H^h_N\ra - \E \la \bff\ra \E \la H^h_N\ra
    \\
    & = x_hc_hN^\frac{15}{16}\Ll(\E \la \bff\Ll(\sum_{\ell=1}^n R^{1,\ell}_h - n R^{1,n+1}_h\Rr)\ra + \E\la \bff\ra \E \la  R^{1,2}_h -R^{1,1}_h\ra\Rr).
\end{align*}
The above three displays yield
\begin{align*}
\E_x\Ll| n\E \la \bff R^{1,n+1}_h \ra - \E \la \bff\ra \E \la R^{1,2}\ra  -  \sum_{\ell=2}^n \E \la \bff R^{1,\ell}_h\ra\Rr| 
\leq  \E_x \E \la \Ll|R^{1,1}_h - \E \la R^{1,1}_h\ra  \Rr|\ra +C_h N^{-\frac{3}{16}} .
\end{align*}
Since $R^{1,1}_\alpha =1$ and $R^{1,1}_\sigma$ is bounded, there exists a constant $C'_h < +\infty$ such that $\Ll|R^{1,1}_h - R^{2,2}_h\Rr|\leq C'_h \Ll|R^{1,1}_\sigma - R^{2,2}_\sigma\Rr|$ and thus
\begin{align*}
     \E \la \Ll|R^{1,1}_h - \E \la R^{1,1}_h\ra  \Rr|\ra \leq \E \la \Ll| R^{1,1}_h - R^{2,2}_h \Rr|\ra  \leq C'_h \E \la \Ll| R^{1,1}_\sigma - R^{2,2}_\sigma \Rr|\ra
     \\
     \leq 2C'_h \E \la \Ll| R^{1,1}_\sigma  -  \E \la R^{1,1}_\sigma\ra \Rr|\ra.
\end{align*}
Taking $\E_x$ in this display, by Proposition~\ref{p.perturbation}~\eqref{i.self-overlap_concent}, we see that the right-hand side vanishes as $N\to\infty$.
Inserting this into the previous display, we get Proposition~\ref{p.perturbation}~\eqref{i.GG}.
\end{proof}

\subsection{Identifying the limits}

Let $\Delta^x_N(\bff,n,h)$ be as given in~\eqref{e.Delta^x}. We enumerate as $((\bff_j,n_j,h_j))_{j\in\N}$ all the triples $(\bff,n,h)$ where $\bff:\Ll(\R^{\D\times\D}\times\R\Rr)^{n\times n}\to\R$ is a monomial (with coefficient $1$), $n\in\N$, $h\in\N^4$. 
Then, we modify each $\bff_j$ in two steps:
first, since $R^{\leq n}$ is bounded, we can adjust $\bff_j$ outside a bounded set to make it bounded;
secondly, we rescale~$\bff_j$ to make sure that $\Delta^x_N(\bff_j,n_j,h_j)\leq 1$. For each $N \in \N$ and $x \in [0,3]^{\N^4 \times \frac{D(D+1)}{2}}$, we set
\begin{align}\label{e.Delta_N(x)}
    \Delta_N(x) := \E \la \Ll|R^{1,1}_{N,\sigma} -\E \la R^{1,1}_{N,\sigma}\ra^\cav_{N,x}\Rr|\ra^\cav_{N,x}
    +\sum_{j=1}^\infty 2^{-j} \Delta^x_N(\bff_j,n_j,h_j).
\end{align}
\begin{align}\label{e.mathcal_K}
 \mbox{We denote by $\mathcal{K}$ the convex hull of the set} \Ll\{\tau\tau^\intercal:\:\tau\in\supp P_1\Rr\}.
\end{align}
The set $\mathcal{K}$ is closed, is such that $\mathcal{K}\subset \S^\D_+$, and we recall from \eqref{e.bound.on.sigma} that we have $|a|\leq 1$ for every $a\in\mathcal{K}$.
Recall the definitions of $R_N$ in \eqref{e.overlap_notation}, $A_N(x)$ in~\eqref{e.A_N(x)}, $\sP_{t,{q}}$ in \eqref{e.mathscrP}, $\la \cdot\ra^\orig_{N,x}$ in \eqref{e.<>^orig_Nx=}, $\la \cdot\ra^\cav_{N,x}$ in \eqref{e.<>^cav_N,x}, and $\la\cdot\ra_{\fR,\pi}$ in \eqref{e.<>_R,pi}. Whenever we speak of the convergence in law of an overlap array, we always mean this in the sense of finite-dimensional distributions. 

\begin{proposition}\label{p.cavity_lim}
We fix $M\in\N$ and $(t,q)\in\R_+\times \C_\infty$, and suppose that there is a sequence $(N_k,x_k)_{k\in\N}$ such that $\lim_{k\to\infty} N_k=+\infty$ and $\lim_{k\to\infty}\Delta_{N_k}(x_k)=0$.
Then, there are a subsequence $(N'_k,x'_k)_{k\in\N}$, ${p}\in \C\cap L^\infty_{\leq 1}$, and $a\in\mathcal{K}$ satisfying $a\geq{p}$ such that
\begin{enumerate}
    \item \label{i.p.cavity_lim_1} $R_{N'_k}$ under $\E\la\cdot\ra^\cav_{N'_k,x'_k}$ converges in law to
    \begin{align*}
        \Ll(Q^{\ell,\ell'}_{p}\1_{\ell\neq \ell'}+ (a,1) \1_{\ell =\ell'}\Rr)_{\ell,\ell'\in\N}
\end{align*}
    under $\E \la\cdot\ra_\fR$ as $k$ tends to infinity;
    \item \label{i.p.cavity_lim_2} we have $\lim_{k\to\infty}A_{N'_k}\Ll(x'_k\Rr) =- M \mathscr{P}_{t,{q}}({p})$;
    \item \label{i.p.cavity_lim_3}
    for every bounded continuous $g:\R^{\D\times\D}\times \R\to\R$,
    \begin{equation*}  \lim_{k\to\infty} \E \la g\Ll(\tau\tau'^\intercal, \alpha\wedge\alpha'\Rr)\ra^\orig_{N'_k+M,x'_k} =\E \la g\Ll(\tau\tau'^\intercal, \alpha\wedge\alpha'\Rr)\ra_{\fR,{q}+t\nabla\xi({p})}.
    \end{equation*}
\end{enumerate}
\end{proposition}

In the statement of Proposition~\ref{p.cavity_lim}, when we say that $(N_k', x_k')_{k \in \N}$ is a subsequence, we mean that $\lim_{k \to \infty} N_k' = +\infty$ and that $\{(N_k', x_k') \mid k \in \N\} \subset \{(N_k, x_k) \mid k \in \N\}$.

We define the \emph{quantile function} $\kappa:[0,1)\to \R$ of a real-valued random variable $X$ as the right-continuous increasing function satisfying
\begin{align*}
    \E g(X) = \int_0^1 g(\kappa(s))\d s
\end{align*}
for every bounded measurable function $g:\R\to\R$. We can obtain $\kappa$ by taking the right-continuous inverse of the probability distribution function of $X$. 
Consistently with~\eqref{e.def.q.1}, we also set $\kappa(1) := \lim_{s\nearrow 1}\kappa(s)$ whenever $\kappa$ is bounded.

\begin{proof}[Proof of Proposition~\ref{p.cavity_lim}]

We denote by $R^\mathrm{diag}_N := \Ll(R^{\ell,\ell}_N\Rr)_{\ell\in\N}$ the diagonal part consisting of self-overlaps, and by $R^\mathrm{off}_N := \Ll(R^{\ell,\ell'}_N\Rr)_{\ell\neq \ell'}$ the off-diagonal part. We will use similar notation for the diagonal and off-diagonal parts of other overlap arrays.

\medskip

\noindent \emph{Part~\eqref{i.p.cavity_lim_1}}.
Since $R^{1,1}_{N,\sigma} \in \S^\D_+$ and $\Ll|R^{1,1}_{N,\sigma}\Rr|\leq 1$, we can extract a subsequence $\Ll(N'_k,\,x'_k\Rr)_{k\in\N}$ from $\Ll(N_k,\,x_k\Rr)_{k\in\N}$ such that
\begin{align}\label{e.E<R11>=q}
    \lim_{k\to\infty} \E \la R^{1,1}_{N'_k,\sigma} \ra^\cav_{N'_k,x'_k} = z
\end{align}
for some $z\in\S^\D_+$.
Since each $ R^{\ell,\ell'}_N$ is bounded, by passing to a further subsequence, we can assume that
\begin{align}\label{e.RtoR}
\mbox{$R_{N'_k}$ under $\E \la\cdot\ra^\cav_{N'_k,x'_k}$ converges in law to $R_\infty$ as $k$ tends to infinity,}
\end{align}
for some random array $R_\infty$ (and in the sense of finite-dimensional distributions).
We will show that the law of $R_\infty$ has a natural representation under $\E\la\cdot\ra_{\fR}$. 
We write $R^{\ell,\ell'}_\infty := \Ll(R^{\ell,\ell'}_{\infty,\sigma}, R^{\ell,\ell'}_{\infty,\alpha}\Rr)$ where $R^{\ell,\ell'}_{\infty,\bullet}$ is the limit of $R^{\ell,\ell'}_{N,\bullet}$ with $\bullet$ being $\sigma$ or $\alpha$. For $R^{\ell,\ell'}_\infty$, the subscripts $\sigma$ and $\alpha$ are purely symbolic.

First, we determine $R^\mathrm{diag}_\infty$.
Due to the first term on the right of~\eqref{e.Delta_N(x)}, the assumption $\lim_{k\to\infty}\Delta_{N_k}\Ll(x^{N_k}\Rr)=0$ ensures the concentration of $R^{\ell,\ell}_{N,\sigma}$.
Also, it is clear from $R^{\ell,\ell}_{N,\alpha}=1$ that $R^{\ell,\ell}_{\infty,\alpha}=1$, which together with \eqref{e.E<R11>=q} yields that $R_\infty^{\mathrm{diag}}$ is deterministic and
\begin{align}\label{e.diag_R}
    R_{\infty}^{\ell,\ell} =(z,1) \quad \forall \ell\in\N.
\end{align}

Next, we determine $R^\mathrm{off}_\infty$.
The same assumption on $\Delta_{N_k}(x^{N_k})$ also ensures that~$R_\infty$ satisfies the Ghirlanda--Guerra identities. Panchenko's synchronization result \cite[Theorem~4]{pan.vec} guarantees the existence of a Lipschitz function $\Psi:\R_+\to\S^\D_+\times \R_+$ satisfying $\Psi(s)\geq \Psi(s')$ for all $s\geq s'$ such that 
\begin{align}\label{e.R=Phi(tr(R))}
    R^{\ell,\ell'}_\infty = \Psi\Ll(\tr\Ll(R^{\ell,\ell'}_\infty\Rr)\Rr) \quad \text{a.s.}\ \forall \ell,\ell'\in\N,
\end{align}
where we used the shorthand $\tr\Ll(R^{\ell,\ell'}_\infty\Rr):=\tr\Ll(R^{\ell,\ell'}_{\infty,\sigma}\Rr)+ R^{\ell,\ell'}_{\infty,\alpha}$.

To proceed, we determine the distribution of $\tr\Ll(R^{\ell,\ell'}_\infty\Rr)$.
We denote the quantile function of $\tr\Ll(R^{1,2}_\infty\Rr)$ by $\kappa$. The assumption $\lim_{k\to\infty}\Delta_{N_k}\Ll(x^{N_k}\Rr)=0$ implies that the Ghirlanda--Guerra identities hold for $\Ll(\tr\Ll(R^{\ell,\ell'}_\infty\Rr)\Rr)_{\ell,\ell'\in\N}$. Using the Dovbysh--Sudakov representation in \cite[Theorem~1.7]{pan} together with the characterization of Poisson--Dirichlet cascades by the overlap distribution in \cite[Theorems~2.13 and~2.17]{pan}, we deduce that $\Ll(\tr\Ll(R^{\ell,\ell'}_\infty\Rr)\Rr)_{\ell\neq \ell'}$ has the same law as the overlap array associated with a Poisson--Dirichlet cascade, and this cascade must be such that the quantile function of one overlap is $\kappa$. 
Using this and~\eqref{e.diag_R}, we set $r:=\tr(z)+1$ so that
\begin{align}\label{e.tr(R)=kappa}
    \tr\Ll(R^{\ell,\ell'}_\infty\Rr) = \kappa\Ll(\alpha^\ell\wedge\alpha^{\ell'}\Rr)\1_{\ell\neq \ell'} + r\1_{\ell=\ell'},
\end{align}
where $\Ll(\alpha^{\ell}\Rr)_{\ell\in\N}$ is sampled from $\la\cdot\ra_\fR$. To be precise, we have argued that the two sides of~\eqref{e.tr(R)=kappa} have the same law, jointly over $\ell, \ell' \in \N$, where the law of the right-hand side is under $\E \la \cdot \ra_\fR$. Since so far we had only specified the law of $R_\infty$ but not its realization itself, we may as well choose the probability space to be that for $\E \langle \cdot \rangle_\fR$ and realize $\tr\Ll(R^{\ell,\ell'}_\infty\Rr)$ according to \eqref{e.tr(R)=kappa}; the full array $R_\infty$ is then specified according to \eqref{e.R=Phi(tr(R))}, and we get the representation
\begin{align}\label{e.R^l,l'_first}
    R^{\ell,\ell'}_\infty = \Psi\Ll(\kappa\Ll(\alpha^\ell\wedge\alpha^{\ell'}\Rr)\1_{\ell\neq \ell'}+r \1_{\ell=\ell'}\Rr),\quad\forall \ell,\ell'\in\N.
\end{align}
In order to simplify~\eqref{e.R^l,l'_first}, we write $\Psi=:(\Psi_\sigma,\Psi_\alpha)$ so that $\Psi_\sigma$ is $\S^\D_+$-valued and $\Psi_\alpha$ is $\R_+$-valued.
By the invariance of Poisson--Dirichlet cascades from Proposition~\ref{p.invar_cts}, we have that $R^{\ell,\ell'}_{N,\alpha}$ under $\E\la\cdot\ra^\cav_{N,x}$ has the same law as $\alpha^\ell\wedge \alpha^{\ell'}$ under $\E\la\cdot\ra_\fR$. Therefore, $R^{\ell,\ell'}_{\infty,\alpha}$ must also have the same law, which means that $\Psi_\alpha(\kappa(\alpha^\ell\wedge\alpha^{\ell'}))$ is uniformly distributed over $[0,1]$ under $\E\la\cdot\ra_\fR$ for $\ell\neq \ell'$. Since $\alpha^\ell\wedge\alpha^{\ell'}$ is itself uniform over $[0,1]$ and $\Psi\circ\kappa$ is increasing and right-continuous, we must have
\begin{align*}\Psi_\alpha(\kappa(s))= s,\quad\forall s\in[0,1].
\end{align*}
Thus, we can simplify~\eqref{e.R^l,l'_first} into
\begin{align}\label{e.R^l,l'}
    R^{\ell,\ell'}_\infty = \bigg(\Psi_\sigma\Ll(\kappa\Ll(\alpha^\ell\wedge\alpha^{\ell'}\Rr)\1_{\ell\neq \ell'}+r \1_{\ell=\ell'}\Rr),\quad \alpha^\ell\wedge\alpha^{\ell'}\bigg),\quad\forall \ell,\ell'\in\N.
\end{align}
Before concluding, we show that
\begin{align}\label{e.zeta(s)<r}
    \kappa(s)\leq r,\quad\forall s\in[0,1].
\end{align}
Note that~\eqref{e.R=Phi(tr(R))} implies that $R^{\ell,\ell'}_{\infty,\sigma} \in\S^\D_+$ for $\ell\neq \ell'$. The Cauchy--Schwarz inequality implies that $v\cdot R^{\ell,\ell'}_{N,\sigma} v\leq \frac{1}{2} v\cdot \Ll(R^{\ell,\ell}_{N,\sigma}+R^{\ell',\ell'}_{N,\sigma}\Rr)v$ for every $v\in\R^\D$, which yields
\begin{align}\label{e.R<(R+R)/2}
    R^{\ell,\ell'}_{\infty,\sigma} \leq \frac{1}{2}\Ll(R^{\ell,\ell}_{\infty,\sigma}+R^{\ell',\ell'}_{\infty,\sigma}\Rr).
\end{align}
Taking the trace on both sides of \eqref{e.R<(R+R)/2} and using \eqref{e.tr(R)=kappa}, we get
\begin{align*}
    \kappa\Ll(\alpha^\ell\wedge\alpha^{\ell'}\Rr) \leq r.
\end{align*}
Since $\alpha^\ell\wedge\alpha^{\ell'}$ is uniformly distributed on $[0,1]$ for $\ell\neq \ell'$ and $\kappa$ is right-continuous, we arrive at \eqref{e.zeta(s)<r}.

We set ${p}:= \Psi_\sigma\circ \kappa$ and $a := \Psi_\sigma(r)$. 
Since $a$ is the subsequential limit of $\frac{\sigma\sigma^\intercal}{N}$, we have that $a$ belongs to the set $\mcl K$ defined in~\eqref{e.mathcal_K}, and in particular $a\in\S^\D_+$ and $|a|\leq 1$. Since the function $\Psi_\sigma$ is increasing, the inequality \eqref{e.zeta(s)<r} implies that ${p}\leq a$. Since $\kappa$, as a quantile function, is also increasing, we have ${p} \in \C$.
Notice that, for $y,y'\in \S^\D_+$ satisfying $y\geq y'$, we have $(y+y')\cdot (y-y')\geq 0$ and thus $|y|^2\geq |y'|^2$. This along with ${p}\leq a$ gives $|{p}|_{L^\infty}\leq |a|\leq 1$.
We have thus verified the conditions on ${p}$ and $a$ stated in the proposition.
Part~\eqref{i.p.cavity_lim_1} follows from \eqref{e.RtoR} and the representation in~\eqref{e.R^l,l'}.

\medskip

\noindent \emph{Part~\eqref{i.p.cavity_lim_2}}.
Due to the possibility that $r>\kappa(1)$, in general, a representation of $R_\infty$ as a Poisson--Dirichlet cascade might be off on the diagonal. 
To circumvent this, we build a sequence of Poisson--Dirichlet cascades whose overlap arrays approximate $R_\infty$ arbitrarily closely.
Allowed by~\eqref{e.zeta(s)<r}, we can
choose a sequence $(\kappa_m)_{m\in\N}$ of right-continuous increasing 
step functions from $[0,1)$ to $[0,r]$ that converges to $\kappa$ in $L^1([0,1);\R)$ as $m$ tends to infinity. For each $m \in \N$, we can also impose that $\lim_{s \nearrow 1}\kappa_m(s) = r$ while preserving the properties mentioned in the previous sentence. As usual, we set $\kappa_m(1) := \lim_{s\nearrow1}\kappa_m(s)$, and thus have $\kappa_m(1)=r$.
Setting ${p}_m:=\Psi_\sigma\circ\kappa_m$, we have that ${p}_m\in \C_\infty$ and
\begin{align}\label{e.zeta_m}
    {p}_m(1) = \Psi_\sigma(r),\qquad \lim_{m\to\infty}\Ll|{p}_m - {p}\Rr|_{L^1} =0.
\end{align}
For each $m$, we consider $Q_{{p}_m}$ defined in~\eqref{e.overlap_alpha}. Notice that, from our choice of $p$ and~\eqref{e.R^l,l'}, we have $\Ll(Q^{\ell,\ell'}_{p}\Rr)_{\ell\neq \ell'}=\Ll(R^{\ell,\ell'}_\infty\Rr)_{\ell\neq \ell'}$.

By the second property in~\eqref{e.zeta_m} and the continuity of Poisson--Dirichlet cascades in the overlap distribution (\cite[Corollary~5.32]{HJbook} or \cite[Theorem 2.17]{pan}), we deduce that $\Ll(Q^{\ell,\ell'}_{{p}_m}\Rr)_{\ell\neq \ell'}$ converges in law to $\Ll(Q^{\ell,\ell'}_{p}\Rr)_{\ell\neq \ell'}=\Ll(R^{\ell,\ell'}_\infty\Rr)_{\ell\neq \ell'}$ under $\E\la\cdot\ra_\fR$. Due to the first property in~\eqref{e.zeta_m}, we have $Q^{\ell,\ell}_{{p}_m} =R^{\ell,\ell}_\infty$, which is deterministic. We thus obtain that, under $\E\la\cdot\ra_\fR$,
\begin{align}\label{e.QtoR}
    \mbox{$Q_{{p}_m}$ converges in law to $R_\infty$ as $m$ tends to infinity. }
\end{align}
In order to proceed, we would like to improve upon the convergences in \eqref{e.RtoR} and \eqref{e.QtoR} into 
\begin{multline}
\label{e.joint_cvg_in_law_1}
\mbox{$\big(R_{N'_k}, q(R_\al)\big)$ under $\E \la\cdot\ra^\cav_{N'_k,x'_k}$ converges in law }
\\ \mbox{to $\big(R_\infty, q(R_\alpha)\big)$ under $\E \la \cdot \ra_\fR$ as $k$ tends to infinity,}
\end{multline}
and, under $\E \la \cdot \ra_\fR$,
\begin{equation}
\label{e.joint_cvg_in_law_2}
\mbox{$\big(Q_{p_m},q(R_\alpha)\big)$ converges in law to $\big(R_\infty,q(R_\alpha)\big)$ as $m$ tends to infinity.}
\end{equation}
These convergences can be obtained by appealing to the continuous mapping theorem as in the proof of Proposition~\ref{p.invar_cts}: in short, the off-diagonal elements of $R_\alpha$ are uniformly distributed over $[0,1]$, so they have probability zero to belong to the set of points of discontinuity of $q$. 

We now fix any $\eps>0$ and recall the notation $\approx_\eps$ from~\eqref{e.approx_eps}.
For the function $F_\eps$ given by Lemma~\ref{l.approx_finite_overlap} and using the convergences in law in \eqref{e.joint_cvg_in_law_1} and  \eqref{e.joint_cvg_in_law_2}, we have for sufficiently large $k$ and $m$ that
\begin{align*}
    A_{N'_k}\Ll(x'_k\Rr) 
    \stackrel{\ref{l.approx_finite_overlap}}{\approx_\eps}  \E \la F_\eps\Ll(R^{\leq n}_{N'_k},q\Ll(R^{\leq n}_{\alpha}\Rr)\Rr) \ra^\cav_{N'_k,x'_k} 
    \stackrel{\eqref{e.joint_cvg_in_law_1}}{\approx_\eps} \E \la F_\eps\Ll(R^{\leq n}_{\infty}, q\Ll(R^{\leq n}_{\sigma}\Rr)\Rr)\ra_\fR 
    \\
    \stackrel{\eqref{e.joint_cvg_in_law_2}}{\approx_\eps} \E \la F_\eps\Ll(Q^{\leq n}_{{p}_m},q\Ll(R^{\leq n}_\alpha\Rr)\Rr)\ra_\fR 
    \stackrel{\ref{l.approx_finite_overlap}}{\approx_\eps} -M\mathscr{P}_{t,{q}}({p}_m).
\end{align*}
The symbol $\approx_\ep$ does not define a transitive relation, and cumulatively we obtain an error of $4\ep$ between the first and the last term in the string of approximate equalities. 
Recall the expression of $\mathscr{P}_{t,{q}}$ in~\eqref{e.mathscrP}.
These along with~\eqref{e.zeta_m} and the Lipschitzness of $\psi$ from Corollary~\ref{c.psi_smooth} imply $\lim_{m\to\infty}\mathscr{P}_{t,{q}}({p}_m) = \mathscr{P}_{t,{q}}({p})$.
Hence, sending $k$ and $m$ to infinity, we get
\begin{align*}
    \liminf_{k\to\infty}  A_{N'_k}\Ll(x'_k\Rr) \approx_{4\eps} - M\mathscr{P}_{t,{q}}({p}) \approx_{4\eps} \limsup_{k\to\infty}  A_{N'_k}\Ll(x'_k\Rr).
\end{align*}
Sending $\eps$ to zero, we obtain part~\eqref{i.p.cavity_lim_2}.

\medskip

\noindent \emph{Part~\eqref{i.p.cavity_lim_3}}.
Recall $\la\cdot\ra^\star_{N+M,x}$ in \eqref{e.<>_N^star}.
For every $k\in\N$ and ${p}' \in \C_\infty$, we write
\begin{align*}
    r(k) := \E \la g\Ll(\tau\tau'^\intercal, \alpha\wedge\alpha'\Rr)\ra^\star_{N'_k+M,x'_k},\qquad \quad \mathfrak{r}\Ll({p}'\Rr) :=\E \la g\Ll(\tau\tau'^\intercal, \alpha\wedge\alpha'\Rr)\ra_{\fR,{q}+t\nabla\xi({p}')}.
\end{align*}
For every $\eps > 0$, we denote by $F_\ep$ the function given by Lemma~\ref{l.approx_finite_overlap_2} and we use~\eqref{e.joint_cvg_in_law_1} and \eqref{e.joint_cvg_in_law_2} to get that, for every sufficiently large $k$ and $m$,
\begin{align*}
    r(k) 
    \stackrel{\ref{l.approx_finite_overlap_2}}{\approx_\eps} \E \la F_\eps \Ll(R^{\leq n}_{N'_k},q\Ll(R^{\leq n}_{\alpha}\Rr)\Rr)\ra^\cav_{N'_k,x'_k}  
    \stackrel{\eqref{e.joint_cvg_in_law_1}}{\approx_\eps} \E \la F_\eps\Ll(R^{\leq n}_\infty,q\Ll(R^{\leq n}_{\alpha}\Rr)\Rr)\ra_\fR 
    \\
    \stackrel{\eqref{e.joint_cvg_in_law_2}}{\approx_\eps} \E \la F_\eps\Ll(Q^{\leq n}_{{p}_m},q\Ll(R^{\leq n}_{\alpha}\Rr)\Rr)\ra_\fR 
    \stackrel{\ref{l.approx_finite_overlap_2}}{\approx_\eps} \mathfrak{r}\Ll({p}_m\Rr).
\end{align*}
Using the second property in~\eqref{e.zeta_m} and a computation similar to the one in~\eqref{e.diff.q.second}, we have $\lim_{m\to\infty} \mathfrak{r}\Ll({p}_m\Rr)=\mathfrak{r}({p})$. Therefore, sending $m$ and $k$ to infinity and then $\ep$ to zero, we obtain that $\lim_{k\to\infty} r(k) = \mathfrak{r}({p})$. Lemma~\ref{l.pre_comp_2} allows us to replace $\la\cdot\ra^\star_{N'_k+M,x'_k}$ in $r(k)$ by $\la\cdot\ra^\orig_{N'_k+M,x'_k}$ in the limit. Hence, part~\eqref{i.p.cavity_lim_3} follows. 
\end{proof}

\subsection{Corollaries}

We now combine Propositions~\ref{p.perturbation} and~\ref{p.cavity_lim}. Recall the definition of~$\la\cdot\ra^\cav_{N,x}$ in~\eqref{e.<>^cav_N,x}, and that $\la\cdot\ra_{\fR}$ denotes the expectation with respect to $\fR^{\otimes \N}$. Recall also~$\sP_{t,{q}}$ from \eqref{e.mathscrP} and the definitions of the overlaps in~\eqref{e.overlap_notation} and~\eqref{e.overlap_alpha}. 
\begin{corollary}\label{c.cavity}
For every $(t,{q})\in \R_+\times \C_\infty$, there are sequences $(N^\pm_k)_{k\in\N}$ of increasing integers, $(x_k^\pm)_{k\in\N}$, ${p}_\pm \in \C\cap L^\infty_{\leq 1}$, and $a_\pm \in \mathcal{K}$ satisfying $a_\pm\geq p_\pm$ such that
\begin{enumerate}
    \item \label{i.c.cavity_I_1} $R_{N^\pm_k}$ under $\E\la\cdot\ra^\cav_{N^\pm_k,x^\pm_k}$ converges in law to
    \begin{align*}
        \Ll(Q^{\ell,\ell'}_{{p}_\pm}\1_{\ell\neq \ell'}+ (a_\pm,1) \1_{\ell =\ell'}\Rr)_{\ell,\ell'\in\N}
    \end{align*}
    under $\E \la\cdot\ra_\fR$ as $k$ tends to infinity;
    \item \label{i.c.cavity_I_2} 
    we have
    \begin{equation*}  \mathscr{P}_{t,{q}}({p}_-)\leq \liminf_{N\to\infty}\bar F_N(t,{q}) \leq \limsup_{N\to\infty}\bar F_N(t,{q}) \leq \mathscr{P}_{t,{q}}({p}_+).
    \end{equation*}
    
\end{enumerate}
    
\end{corollary}

\begin{proof}
Using Lemma~\ref{l.limF^x-F} and Proposition~\ref{p.cavity_perturbation} with $M=1$, we have
\begin{align*}
    \liminf_{N\to\infty}\bar F_N(t,{q}) = \liminf_{N\to\infty} \E_x \bar F_N^x(t,{q}) \geq -\limsup_{N\to\infty} \E_x A_N(x).
\end{align*}
Lemmas~\ref{l.limF^x-F} and~\ref{l.NF_N-(N+M)F_(N+M)} imply that $\sup_{N,x}|A_N(x)|\leq c$ for some constant $c$. 
Proposition~\ref{p.perturbation} and \eqref{e.Delta_N(x)} yield that $\lim_{N\to\infty} \E_x \Delta_N(x) =0$. Using these two facts and following the argument in \cite[Lemma~3.3]{pan}, we can find a sequence $(x_N)_{N\in\N}$ satisfying $\lim_{N\to\infty}\Delta_N(x_N) = 0$ and $A_N(x_N) \geq \E_x A_N(x)+o(1)$ as $N$ tends to infinity. 

We remark that \cite[Lemma~3.3]{pan} is originally stated to get an upper bound for $A_N(x_N)$. Examining the proof, one can see that as long as $A_N(x)$ is bounded uniformly, the argument still works for the other side needed here. Hence,
\begin{align*}
    \liminf_{N\to\infty}\bar F_N(t,{q}) \geq -\limsup_{N\to\infty} A_N(x_N).
\end{align*}
We extract $(N_k)_{k\in\N}$ along which the $\limsup_{N\to\infty} A_N(x_N)$ is achieved.
By Proposition~\ref{p.cavity_lim} at $M=1$, we can find a subsequence $(N^-_k,\, x^-_k)_{k\in\N}$ of $(N_k,\,x_{N_k})_{k\in\N}$ and ${p}_-\in\C\cap L^\infty$ such that part~\eqref{i.c.cavity_I_1} holds and the first inequality in part~\eqref{i.c.cavity_I_2} holds.
The other half follows from a similar argument.
\end{proof}

Recall $\la\cdot\ra^\orig_{N,x}$ in~\eqref{e.<>^orig_Nx=}, $\la\cdot\ra^\cav_{N,x}$ in~\eqref{e.<>^cav_N,x}, and $\la\cdot\ra_{\fR,\pi}$ in~\eqref{e.<>_R,pi}.
\begin{corollary}\label{c.cavity_II}
For any $M\in\N$, any $(t,{q})\in \R_+\times (\C\cap L^\infty)$, and any sequence of increasing integers, there are a subsequence $(N_k)_{k\in\N}$, $(x_k)_{k\in\N}$, ${p} \in \C\cap L^\infty_{\leq 1}$, and $a\in \mathcal{K}$ satisfying $a\geq p$ such that
\begin{enumerate}
    \item \label{i.c.cavity_II_1} $R_{N_k}$ under $\E\la\cdot\ra^\cav_{N_k,x_k}$ converges in law to 
    \begin{align*}
        \Ll(Q^{\ell,\ell'}_{p}\1_{\ell\neq \ell'}+ (a,1) \1_{\ell =\ell'}\Rr)_{\ell,\ell'\in\N}
    \end{align*}
    under $\E \la\cdot\ra_\fR$  as $k$ tends to infinity;
    \item \label{i.c.cavity_II_2} for every bounded continuous $g:\R^{\D\times\D}\times \R\to\R$, we have
    \begin{equation*}  \lim_{k\to\infty} \E \la g\Ll(\tau\tau'^\intercal, \alpha\wedge\alpha'\Rr)\ra^\orig_{N_k+M,x_k} =\E \la g\Ll(\tau\tau'^\intercal, \alpha\wedge\alpha'\Rr)\ra_{\fR,{q}+t\nabla\xi({p})}.
    \end{equation*}
    \end{enumerate}
\end{corollary}

\begin{proof}
Recall $\Delta_N(x)$ in~\eqref{e.Delta_N(x)}.
By Proposition~\ref{p.perturbation}, we have $\lim_{N\to\infty}\E_x \Delta_N(x)=0$. As discussed in the previous proof, we can find $(x_N)_{N\in\N}$ such that $\lim_{N\to\infty} \Delta_N(x_N)=0$. The desired result then follows from Proposition~\ref{p.cavity_lim}. 
\end{proof}

\section{Proofs of the main results}
\label{s.concl}

We now leverage the results of the previous section to complete the proofs of our main results. We first show Theorem~\ref{t.main2} as well as the fact that any subsequential limit of the free energy $\bar F_N$ must solve the equation ``almost everywhere'', as was announced in \eqref{e.satisfy.eqn.intro}. We then deduce Theorem~\ref{t.main1} by a continuity argument. We next identify, for ``almost every'' choice of $q$, the limit law of the array of the conditional expectations of the $\sigma$-overlaps given the $\alpha$-overlaps. Up to the additional perturbation provided by the parameter $\hat t$ introduced around \eqref{e.def.tdH}, we can then slightly upgrade this result and thereby obtain a proof of Theorem~\ref{t.main3}. 

\subsection{Critical point identification}

Recall the definition of $\sP_{t,q}$ in~\eqref{e.mathscrP}. In the next proposition, we assume the convergence of the entire sequence $\Ll(\bar F_N\Rr)_{N\in\N}$.

\begin{proposition}
\label{p.1st_fpe}
Suppose that $\Ll(\bar F_N\Rr)_{N\in\N}$ converges pointwise to some limit $f$, and let $t \ge 0$ and $q \in \mcl Q_\uparrow \cap L^\infty$. If $ f(t,\cdot)$ is Gateaux differentiable at ${q}$, then $f(t,{q}) = \mathscr{P}_{t,{q}}(\dr_q  f(t,{q}))$.

\end{proposition}

\begin{proof}
Let $(N^\pm_k)_{k\in\N}$, $(x^\pm_k)_{k\in\N}$, and ${p}_\pm$ be given by Corollary~\ref{c.cavity}, whose part~\eqref{i.c.cavity_I_2} implies 
\begin{align}\label{e.P<limF<P}
    \mathscr{P}_{t,{q}}({p}_-) \leq f(t,{q})\leq  \mathscr{P}_{t,{q}}({p}_+).
\end{align}
Recall the definitions of $\tilde F^x_N$ in \eqref{e.tildeF_N=} and of $\la\cdot\ra^\cav_{N,x}$ in \eqref{e.<>^cav_N,x}. We set $\tilde F_k^\pm := \tilde F^{x^\pm_k}_{N^\pm_k}$.
For every continuous $\kappa\in L^\infty$, by a computation similar to that for \eqref{e.def.der.FN}, we have, for every $k \in \N$,
\begin{align*}
    \la \kappa, \dr_q \tilde F_k^\pm(t,{q}) \ra_\cH = \E \la \kappa\Ll(\alpha\wedge\alpha'\Rr)\cdot \frac{\sigma\sigma'^\intercal}{N^\pm_k}  \ra^\cav_{N^\pm_k,x^\pm_k}.
\end{align*}
By Lemma~\ref{l.limF^x-F}, the function $\Ll(\tilde F_k^\pm\Rr)_{k\in\N}$ also converges pointwise to $f$ as $k$ tends to infinity. Following the proof of Proposition~\ref{p.conv.der} verbatim, we obtain that $\dr_q \tilde F_k^\pm(t,{q})$ converges to~$\dr_q  f(t,{q})$ in $L^2$.
Using this and Corollary~\ref{c.cavity}~\eqref{i.c.cavity_I_1} yields that
\begin{align*}
    \la \kappa, \dr_q  f(t,{q})\ra_{L^2} = \E \la \kappa\Ll(R^{1,2}_\alpha\Rr)\cdot{p}_\pm(R^{1,2}_\alpha)\ra_{\fR}  = \int_0^1 \kappa(s)\cdot{p}_\pm(s)\d s.
\end{align*}
Since the set of continuous $\kappa\in L^\infty$ is dense in $\cH$, we have ${p}_\pm = \dr_q  f(t,q)$, which along with~\eqref{e.P<limF<P} gives the desired result.
\end{proof}

In the following result, the limit is subsequential. 

\begin{proposition}\label{p.2nd_fpe}
Suppose that $\Ll(\bar F_{N}\Rr)_{N\in\N}$ converges pointwise to some limit $f$ along a subsequence $(N_k)_{k\in\N}$, and let $t \ge 0$. 
If $f(t,\cdot)$ is Gateaux differentiable at some ${q} \in \C_\uparrow\cap L^\infty$, then
\begin{equation}
\label{e.fixed_pt_nbla_f}
    \dr_q f(t,{q}) = \dr_q \psi ({q}+t\nabla\xi(\dr_q f(t,{q}))).
\end{equation}
If, in addition, $t>0$ and $f(\cdot,q)$ is differentiable at $t$, then
\begin{align}\label{e.satify_eqn}
    \partial_t f(t,{q}) = \int_0^1\xi(\dr_q f(t,{q})). 
\end{align}
\end{proposition}

\begin{proof}
By passing to a further subsequence, we may assume that $(N_k)_{k\in\N}$ is the subsequence given in Corollary~\ref{c.cavity_II} (for $M=1$) along with $(x_k)_{k\in\N}$ and ${p} \in \C\cap L^\infty_{\leq 1}$.

Recall $\tilde F_N^x$ in~\eqref{e.tildeF_N=} and $\la\cdot\ra^\cav_{N,x}$ in~\eqref{e.<>^cav_N,x}.
As in the proof of \eqref{e.def.der.FN}, we have for any bounded continuous $\kappa:[0,1]\to\S^\D$ that
\begin{align}\label{e.dtildeF}
    \la \kappa, \dr_q \tilde F_{N_k}^{x_k}(t,{q})\ra_\cH = \E \la \kappa\Ll(\alpha\wedge\alpha'\Rr)\cdot \frac{\sigma\sigma'^\intercal}{N_k} \ra^\cav_{N_k,x_k},\quad \partial_t \tilde F_{N_k}^{x_k}(t,{q}) = \E \la \xi\Ll(\frac{\sigma\sigma'^\intercal}{N_k}\Rr) \ra^\cav_{N_k,x_k}.
\end{align}
By Lemma~\ref{l.limF^x-F}, we have that $\tilde F_{N_k}^{x_k}$ converges pointwise to $f$.
Following the proof of Proposition~\ref{p.conv.der}, we deduce that $\dr_q \tilde F_{N_k}^{x_k}(t,{q})$ converges to $\dr_q f(t,q)$ in $L^2$. Using also the convergence of the overlaps from Corollary~\ref{c.cavity_II}~\eqref{i.c.cavity_II_1}, we let $k$ tend to infinity in the first relation in~\eqref{e.dtildeF} to get that
\begin{align*}
    \la \kappa, \dr_q f(t,{q})\ra_\cH = \E \la \kappa\Ll(\alpha\wedge\alpha'\Rr) \cdot {p}\Ll(\alpha\wedge\alpha'\Rr)\ra_\fR = \la \kappa,{p}\ra_\cH. 
\end{align*}
Since the set of bounded continuous functions spans $L^2$, we deduce that
\begin{align}\label{e.zeta=nabla_f}
    {p}= \dr_q f(t,{q}).
\end{align}
Under the additional assumption on the differentiability in $t$ (and an adaptation  to $\tilde F_{N_k}^{x_k}$ of the part of Proposition~\ref{p.conv.der} concerning the derivative in $t$), we can pass to the limit in the second relation in~\eqref{e.dtildeF} and get that
\begin{align*}\partial_t f(t,{q}) = \E \la \xi\Ll({p}\Ll(\alpha\wedge\alpha'\Rr)\Rr)\ra_\fR = \int_0^1\xi({p}(u))\d u .
\end{align*}
This and \eqref{e.zeta=nabla_f} yield \eqref{e.satify_eqn}.

To obtain \eqref{e.fixed_pt_nbla_f}, we relate ${p}$ to the derivative of $\psi$. Recall $\bar F_N^x$ in~\eqref{e.F^x_N=} and $\la\cdot\ra^\orig_{N,x}$ in~\eqref{e.<>^orig_Nx=}.
By Lemmas~\ref{l.limF^x-F} and~\ref{l.NF_N-(N+M)F_(N+M)}, we have that $\bar F_{N_k+1}^{x_k}$ converges pointwise to $f$. Following the proof of \eqref{e.def.der.FN}, we have for any bounded continuous $\kappa$ that
\begin{align*}
     \la \kappa, \dr_q \bar F_{N_k+1}^{x_k}(t,{q})\ra_\cH = \E \la \kappa\Ll(\alpha\wedge\alpha'\Rr)\cdot \frac{\rho\rho'^\intercal}{N_k+1} \ra^\orig_{N_k+1,x_k}.
\end{align*}
We write $\rho$ as an $(N_k +1)$-tuple of $D$-dimensional coordinates $\rho =\Ll(\rho_1,\dots,\rho_{N_k},\rho_{N_k+1}\Rr)$, which are i.i.d.\ under $P_{N_k+1} = P_1^{\otimes (N_k+1)}$. Since the Hamiltonian in $\la\cdot\ra^\orig_{N_k+1,x_k}$ only depends on $\rho\rho^\intercal$ and $\rho\rho'^\intercal$, we deduce that $\la\cdot\ra^\orig_{N_k+1,x_k}$ is invariant under permutations of the coordinates of $\rho$. Therefore, writing $\sigma := (\rho_1,\dots\rho_{N_k})$ and $\tau := \rho_{N_k+1}$, we have 
\begin{align*}
     \la \kappa, \dr_q \bar F_{N_k+1}^{x_k}(t,{q})\ra_\cH = \E \la \kappa\Ll(\alpha\wedge\alpha'\Rr)\cdot \tau\tau'^\intercal \ra^\orig_{N_k+1,x_k}.
\end{align*}
By an adaptation of Proposition~\ref{p.conv.der}, we have that $\dr_q \bar F_{N_k+1}^{x_k}(t,{q})$ converges in $L^2$ to $\dr_q  f(t,q)$. Using also Corollary~\ref{c.cavity_II}~\eqref{i.c.cavity_II_2} (with $M=1$), we can pass to the limit in the previous display and get
\begin{align*}
    \la \kappa, \dr_q f(t,{q})\ra_\cH = \E \la \kappa\Ll(\alpha\wedge\alpha'\Rr)\cdot \tau\tau'^\intercal \ra_{\fR,q+t\nabla\xi({p})},
\end{align*}
where $\la\cdot\ra_{\fR,q+t\nabla\xi({p})}$ is given in~\eqref{e.<>_R,pi} with $M=1$. Notice that since $M = 1$, the Gibbs measure $\la\cdot\ra_{\fR,{q}+t\nabla\xi({p})}$ is the same as the Gibbs measure $\la \cdot\ra_{{q}+t\nabla\xi({p})}$ defined in \eqref{e.<>_mu}.
By~\eqref{e.def.der.psi}, the right-hand side in the previous display is equal to $\la \kappa, \dr_q \psi({q}+t\nabla\xi({p}))\ra_\cH$. Combining this with~\eqref{e.zeta=nabla_f} yields that
\begin{align}\label{e.nabla_psi=zeta}
    \dr_q \psi({q}+t\nabla\xi({p})) = {p}.
\end{align}
In conclusion, \eqref{e.fixed_pt_nbla_f} follows from~\eqref{e.zeta=nabla_f} and~\eqref{e.nabla_psi=zeta}.
\end{proof}

Recall $\mathscr{P}_{t,{q}}$ in~\eqref{e.mathscrP}, $\theta$ in \eqref{e.theta=}, and $\mcl J_{t,{q}}$ in~\eqref{e.def.Parisi.functional}. We have
\begin{align}\label{e.P=J}
    \mathscr{P}_{t,{q}}({p}) = \mcl J_{t,{q}}\Ll({q}+ t \nabla\xi({p}),{p}\Rr).
\end{align}

\begin{proof}[Proof of Theorem~\ref{t.main2}]
Let ${p}$ and ${q'}$ be as given in the statement. The relation~\eqref{e.fixed_pt_nbla_f} in Proposition~\ref{p.2nd_fpe} implies that $({q'},p)$ is a critical point.
The convergence in~\eqref{e.main1.sg} follows from Proposition~\ref{p.1st_fpe} and~\eqref{e.P=J}.
\end{proof}

Proposition~\ref{p.2nd_fpe} also proves that any subsequential limit of the free energy must satisfy the Hamilton--Jacobi equation \eqref{e.satisfy.eqn.intro} at every point of differentiability of the limit. We recall that this limit function must be Gateaux differentiable jointly in its two variables ``almost everywhere'', by Proposition~\ref{p.reg.lim}.

\subsection{Critical point representation}

\begin{proposition}\label{p.1st_stronger_id}
Suppose that $\Ll(\bar F_N\Rr)_{N\in\N}$ converges pointwise to some limit $f$.
 For every $(t,{q})\in\R_+\times \C_2$, there is ${p} \in \C\cap L^\infty_{\leq 1}$ such that
\begin{gather}\label{e.p.1st_stronger_id}
    f(t,{q}) = \mathscr{P}_{t,{q}}({p}), \qquad
    {p} =\dr_q\psi({q} + t\nabla\xi({p})). 
\end{gather}
\end{proposition}

\begin{proof}
Proposition~\ref{p.reg.lim} yields a sequence $({q}_n)_{n\in\N}$ in $  \C_\uparrow \cap L^\infty$ convergent to ${q}$ in $\cH$ such that $f(t,\cdot)$ is Gateaux differentiable at every ${q}_n$. By the continuity of $f$ given there, we have
\begin{align}\label{e.f(t,mu)=limf(t,mu_n)}
    f(t,{q}) = \lim_{n\to\infty} f(t,{q}_n).
\end{align}
We write ${p}_n := \dr_q f(t,{q}_n)$. Proposition~\ref{p.reg.lim} also implies that $\sup_{n\in\N}|{p}_n|_{L^\infty} \leq 1$.
By passing to a subsequence, we can assume that ${p}_n$ converges weakly in $\cH$ to some limit ${p}$. Arguing as in Step 3 of the proof of Proposition~\ref{p.conv.der}, we can upgrade this to the strong convergence in~$L^2$, and we also see that ${p} \in \C\cap L^\infty_{\leq 1}$.

Recall $\mathscr{P}_{t,{q}}$ in~\eqref{e.mathscrP}. By Corollary~\ref{c.psi_smooth} and the local Lipschitzness of $\nabla\xi$ and $\theta$, we have
\begin{align*}
    \Ll|\mathscr{P}_{t,{q}_n}({p}_n) - \mathscr{P}_{t,{q}}({p})\Rr| &\leq  \Ll|{q}_n-{q}\Rr|_{L^1}+t\Ll|\nabla\xi({p}_n)-\nabla\xi({p})\Rr|_{L^1}  + t\Ll|\theta({p}_n)-\theta({p})\Rr|_{L^1}
    \\
    &\leq Ct \Ll|{p}_n-{p}\Rr|_{L^1} + \Ll|{q}_n-{q}\Rr|_{L^1},
\end{align*}
for some constant $C$ depending only on $\xi$. By Jensen's inequality, we have that $q_n$ and $p_n$ converge in $L^1$ to $q$ and $p$ respectively as $n$ tends to infinity, and therefore
\begin{equation*}  \lim_{n\to\infty} \mathscr{P}_{t,{q}_n}({p}_n) = \mathscr{P}_{t,{q}}({p}).
\end{equation*}
By construction, the function $f(t,\cdot)$ is Gateaux differentiable at ${q}_n$. Proposition~\ref{p.1st_fpe} therefore ensures that 
\begin{equation*}  f(t,{q}_n) = \mathscr{P}_{t,{q}_n}({p}_n).
\end{equation*}
These along with~\eqref{e.f(t,mu)=limf(t,mu_n)} yields $f(t,{q}) = \mathscr{P}_{t,{q}}({p})$, which gives the first relation in~\eqref{e.p.1st_stronger_id}.

Proposition~\ref{p.2nd_fpe} ensures that ${p}_n = \dr_q\psi({q}_n+t\nabla\xi({p}_n))$.
Using \eqref{e.continuity.der.psi} and the local Lipschitzness of $\nabla\xi$, we get
\begin{align*}
    \Ll|\dr_q\psi({q}_n+t\nabla\xi({p}_n))- \dr_q\psi({q}+t\nabla\xi({p}))\Rr|_{L^2} \leq C|{q}_n-{q}|_{L^2} + C|{p}_n-{p}|_{L^2},\quad\forall n\in\N,
\end{align*}
for some $C$ only depending on $\xi$. Sending $n$ to infinity, we get ${p} = \dr_q\psi({q}+t\nabla\xi({p}))$, which gives the second relation in~\eqref{e.p.1st_stronger_id}.
\end{proof}

\begin{proof}[Proof of Theorem~\ref{t.main1}]
Let ${p}$ be given in Proposition~\ref{p.1st_stronger_id} and set ${q'}:= {q} + t\nabla\xi({p})$. The second relation in~\eqref{e.p.1st_stronger_id} ensures that $({q'},p)$ is a critical point. The first relation in~\eqref{e.p.1st_stronger_id} along with~\eqref{e.P=J} yields~\eqref{e.main1.sg}.
\end{proof}

\subsection{Convergence of the overlap}

\subsubsection{Convergence of the conditional overlap}

Recall the Gibbs measure $\la \cdot\ra_N$ associated with the original free energy given in~\eqref{e.<>_N}.
For $\ell,\ell'\in\N$, we consider the conditional expectation of the spin overlap
\begin{align*}
    R_{N,\sigma|\alpha}^{\ell,\ell'}
    :=\E \la \frac{\sigma^\ell\Ll(\sigma^{\ell'}\Rr)^\intercal}{N} \,\Big|\,\alpha^\ell\wedge \alpha^{\ell'}\ra_N,
\end{align*}
where the conditioning is with respect to $\E \la \cdot\ra_N$ (not $\la\cdot\ra_N$). We also recall that $\la \cdot \ra_{\fR}$ is the expectation with respect to the measure $\fR^{\otimes \N}$.

\begin{proposition}\label{p.cvg_overlap_og}
Suppose that $\Ll(\bar F_N\Rr)_{N\in\N}$ converges pointwise to some limit $f$  along a subsequence $(N_k)_{k\in\N}$, and let $t \ge 0$. 
If $f(t,\cdot)$ is Gateaux differentiable at some ${q} \in \C_\uparrow\cap L^\infty$, then $\Ll(R_{N_k,\sigma|\alpha}^{\ell,\ell'}\Rr)_{\ell,\ell'\in\N:\:\ell\neq \ell'}$ under $\E \la \cdot\ra_{N_k}$ converges in law to $\Ll({p}\Ll(\alpha^\ell\wedge\alpha^{\ell'}\Rr)\Rr)_{\ell,\ell'\in\N:\:\ell\neq \ell'}$ under $\E \la\cdot\ra_\fR$ as $k$ tends to infinity, where ${p} = \dr_q f(t,q)$.

\end{proposition}

\begin{proof}
We write ${p}_N := \dr_q \bar F_N \Ll(t,{q}\Rr)$.
Using the invariance of the Poisson--Dirichlet cascade in Proposition~\ref{p.invar_cts}, the computation of $\dr_q \bar F_N \Ll(t,{q}\Rr)$ in \eqref{e.def.der.FN}, and the definition of conditional expectation, we have the following identities, for any bounded continuous $\kappa : [0,1] \to S^D$:
\begin{multline*}
    \E \la \kappa\Ll(\alpha\wedge\alpha'\Rr)\cdot {p}_N\Ll(\alpha\wedge\alpha'\Rr) \ra_N = \int_0^1\kappa(s)\cdot {p}_N(s)\d s = \la \kappa, \dr_q \bar F_N \Ll(t,{q}\Rr)\ra_\cH
    \\
    =\E \la \kappa \Ll(\alpha\wedge\alpha'\Rr)\cdot \frac{\sigma\sigma'^\intercal}{N}\ra_N = \E \la \kappa \Ll(\alpha\wedge\alpha'\Rr)\cdot \E \la\frac{\sigma\sigma'^\intercal}{N}\,\Big|\, \alpha\wedge\alpha' \ra_N\ra_N.
\end{multline*}
Hence, we deduce that
\begin{align}\label{e.cond_overlap=p_N()}
    \E \la\frac{\sigma\sigma'^\intercal}{N}\,\Big|\, \alpha\wedge\alpha'\ra_N = {p}_N\Ll(\alpha\wedge\alpha'\Rr)
\end{align}
a.s.\ under $\E\la\cdot\ra_N$. For every ${p}'\in \C_\infty$, we write $O_{{p}'}^{\ell,\ell'} := {p}'\Ll(\alpha^\ell\wedge\alpha^{\ell'}\Rr)$,
 so that under $\E \la\cdot\ra_N$, 
\begin{align}\label{e.R=O}
    R_{N,\sigma|\alpha}^{\ell,\ell'} = O_{{p}_N}^{\ell,\ell'},\quad\forall \ell,\ell'\in \N:\: \ell\neq \ell'.
\end{align}
Proposition~\ref{p.conv.der} ensures that ${p}_N$ converges to ${p}$ in $L^1$ along the subsequence, and thus
\begin{align}\label{e.E<|p_N-p|>}
  \E \la \Ll|{p}_N\Ll(\alpha\wedge\alpha'\Rr)-{p}\Ll(\alpha\wedge\alpha'\Rr)\Rr|\ra_\fR = \Ll|{p}_N- {p}\Rr|_{L^1}
\end{align}
tends to zero as $N$ tends to infinity along the subsequence. 
We fix any $n\in\N$ and any bounded Lipschitz $g:\Ll(\R^{\D\times\D}\times \R\Rr)^{n^2-n} \to \R$. We write $R^\square_{N,\sigma|\alpha} := \Ll(R_{N,\sigma|\alpha}^{\ell,\ell'}\Rr)_{\ell,\ell'\leq n; \ell\neq \ell'}$ and similarly for $O^\square_{{p}_N}$ and $O^\square_{p}$.
The invariance of the Poisson--Dirichlet cascade from Proposition~\ref{p.invar_cts} implies that $O_{{p}_N}$ under $\E\la\cdot\ra_N$ has the same distribution as $O_{{p}_N}$ under~$\E \la\cdot\ra_\fR$.
Combining this with~\eqref{e.R=O} yields that
\begin{align*}
    &\Ll|\E \la g\Ll(R^\square_{N,\sigma|\alpha}\Rr) \ra_N - \E \la g\Ll(O^\square_{p}\Rr) \ra_\fR \Rr| = \Ll| \E \la g\Ll(O^\square_{{p}_N}\Rr) \ra_N - \E \la g\Ll(O^\square_{p}\Rr) \ra_\fR \Rr|
    \\
    &=  \Ll| \E \la g\Ll(O^\square_{{p}_N}\Rr) \ra_\fR - \E \la g\Ll(O^\square_{p}\Rr) \ra_\fR \Rr|
     \leq \|g\|_\mathrm{Lip}\E \la \Ll|O^\square_{{p}_N} - O^\square_{p}\Rr| \ra_\fR.
\end{align*}
Using also that \eqref{e.E<|p_N-p|>} tends to zero along the subsequence therefore yields the announced result. 
\end{proof}

\subsubsection{Convergence of the overlap under perturbation}

Let $\Ll(\hat H_N(\sigma)\Rr)_{\sigma\in\R^{D\times N}}$ be the Gaussian process in~\eqref{e.hatH_N}.
For every $\Ll(t,\hat t,{q}\Rr)\in \R_+ \times \R_+\times \C_\infty$, we consider
\begin{multline}
\label{e.def.FN.t.hat}
    \hat F_N\Ll(t,\hat t,{q}\Rr) 
    \\
    := -\frac{1}{N} \E\log \int \exp\Ll(H_N^{t,{q}}(\sigma,\alpha)+\sqrt{2\hat t}\hat H_N(\sigma)- \hat t N\Ll|\frac{\sigma\sigma^\intercal}{N}\Rr|^2\Rr) \d P_N(\sigma)\d \fR(\alpha),
\end{multline}
where $H^{t,{q}}_N(\sigma,\alpha)$ is given in~\eqref{e.H^t,mu_N}. 
We denote the associated Gibbs measure by $\la \cdot\ra_{N,\hat t}\, $. 
Recall the functional $\hat {\mcl J}_{t,\hat t,q}$ defined in~\eqref{e.hatJ=}, and the overlaps $R_N$ and $Q_{p}$ defined in~\eqref{e.overlap_notation} and~\eqref{e.overlap_alpha}.
\begin{proposition}\label{p.cvg_overlap}
Suppose that $\Ll(\hat F_N\Rr)_{N\in\N}$ converges pointwise to some limit $f$  along a subsequence $(N_k)_{k\in\N}$. 
Then, for each $t \ge 0$, the function $f\Ll(t,\cdot,\cdot\Rr) : \R_+ \times \mcl Q_2 \to \R$ is Gateaux differentiable (jointly in its two variables) on a subset of $\R_+ \times (\C_\uparrow\cap L^\infty)$ that is dense in $\R_+ \times \C_2$. Moreover, for every $\hat t \ge 0$ and every point ${q} \in \C_\uparrow\cap L^\infty$ of Gateaux differentiability of $f(t,\hat t, \cdot)$, the following holds for ${p} = \dr_q f\Ll(t,\hat t, {q}\Rr)$ and ${q'} = {q} + t\nabla \xi\Ll({p}\Rr) + 2 \hat t {p}$:
\begin{enumerate}
    \item \label{i.p.cvg_overlap_1} ${p} = \dr_q \psi \Ll({q'}\Rr)$;
    \item \label{i.p.cvg_overlap_2} if $(N_k)_{k\in\N}$ is the full sequence $(N)_{N\in\N}$, then $\lim_{N\to\infty}\hat F_N\Ll(t,\hat t, {q}\Rr) = \hat{\mathcal{J}}_{t,\hat t, {q}}({q'},{p})$;
    \item \label{i.p.cvg_overlap_3} if $\hat t>0$ and $f(t,\cdot,{q})$ is differentiable at $\hat t$, then $\Ll(R_{N_k}^{\ell,\ell'}\Rr)_{\ell,\ell'\in\N:\:\ell\neq \ell'}$ under $\E \la \cdot\ra_{N_k,\hat t}$ converges in law to $\Ll({p}\Ll(\alpha^\ell\wedge\alpha^{\ell'}\Rr)\Rr)_{\ell,\ell'\in\N:\:\ell\neq \ell'}$ under $\E \la\cdot\ra_\fR$ as $k$ tends to infinity.
\end{enumerate}

\end{proposition}

\begin{proof}
By straightforward modifications, the differentiability follows from Proposition~\ref{p.reg.lim}, part~\eqref{i.p.cvg_overlap_1} follows from Proposition~\ref{p.2nd_fpe}, and part~\eqref{i.p.cvg_overlap_2} follows from Proposition~\ref{p.1st_fpe}.

Hence, we focus on part~\eqref{i.p.cvg_overlap_3}.
For convenience, we drop the subsequence $(N_k)_{k \in \N}$ in the notation, but we understand that we only take limits along the subsequence throughout the proof.
By an adaptation of Proposition~\ref{p.conv.der}, we have (along the subsequence)
\begin{align}\label{e.cvg_mathscrF}
    \lim_{N\to\infty} \Ll|\dr_q \hat F_N \Ll(t,\hat t,{q}\Rr) - \dr_q f\Ll(t,\hat t,{q}\Rr)\Rr|_\cH = 0\quad \text{ and } \quad \lim_{N\to\infty} \partial_{\hat t} \hat F_N\Ll(t,\hat t,{q}\Rr) = \partial_{\hat t} f\Ll(t,\hat t,{q}\Rr).
\end{align}
For brevity, we write ${p}_N := \dr_q \hat F_N \Ll(t,\hat t,{q}\Rr)$, ${p} := \dr_q f\Ll(t,\hat t,{q}\Rr)$, and $\la \cdot\ra_N := \la\cdot\ra_{N,\hat t}\, $.
We have
\begin{align*}
    \partial_{\hat t} \hat F_N\Ll(t,\hat t,{q}\Rr) = \E \la \Ll|\frac{\sigma\sigma'^\intercal}{N}\Rr|^2\ra_N.
\end{align*}
Adapting the proof of Proposition~\ref{p.cvg_overlap_og}, we have that ~\eqref{e.cond_overlap=p_N()} holds, which implies that
\begin{align}
    \E \la \Ll|\frac{\sigma\sigma'^\intercal}{N}-{p}_N\Ll(\alpha\wedge\alpha'\Rr)\Rr|^2\ra_N 
    &= \E \la \Ll|\frac{\sigma\sigma'^\intercal}{N}\Rr|^2\ra_N - \E \la \Ll|{p}_N\Ll(\alpha\wedge\alpha'\Rr)\Rr|^2\ra_N  \notag
    \\
    &= \partial_{\hat t} \hat F_N\Ll(t,\hat t,{q}\Rr) - \Ll|\dr_q\hat F_N\Ll(t,\hat t,{q}\Rr)\Rr|^2_\cH \label{e.E<>=E<>-E<>}
\end{align}
where in the second equality we used the invariance of the Poisson--Dirichlet cascade.

We can now argue as for~\eqref{e.satify_eqn} in Proposition~\ref{p.2nd_fpe} to conclude. Indeed, since $t$ is fixed here, we view $\hat F_N(t,\cdot,\cdot)$ and $f(t,\cdot,\cdot)$ as functions of $(\hat t,{q})$. We can substitute $\hat t$ and $|\cdot|^2$ for $t$ and $\xi$ in Proposition~\ref{p.2nd_fpe} to get
\begin{gather}\label{e.f(s,mu)_PDE}
    \partial_{\hat t} f\Ll(t,\hat t,{q}\Rr) = \Ll|\dr_q f\Ll(t,\hat t,{q}\Rr)\Rr|^2_\cH.
\end{gather}
Combining~\eqref{e.cvg_mathscrF}, \eqref{e.E<>=E<>-E<>}, and~\eqref{e.f(s,mu)_PDE}, we obtain that (along the subsequence)
\begin{align}
\label{e.RNQN.close}
    \lim_{N\to\infty} \E \la \Ll|\frac{\sigma\sigma'^\intercal}{N}-{p}_N\Ll(\alpha\wedge\alpha'\Rr)\Rr|^2\ra_N =0.
\end{align}
The first convergence in~\eqref{e.cvg_mathscrF} ensures that (along the subsequence)\begin{align}
\label{e.pNp.close}
    \lim_{N\to\infty} \E \la \Ll|{p}_N\Ll(\alpha\wedge\alpha'\Rr)-{p}\Ll(\alpha\wedge\alpha'\Rr)\Rr|^2\ra_\fR =0.
\end{align}
We fix any $n\in\N$ and any bounded Lipschitz $g:\Ll(\R^{\D\times\D}\times \R\Rr)^{n^2-n} \to \R$.
We write $R^\square_N := \Ll(R_N^{\ell,\ell'}\Rr)_{\ell,\ell'\leq n; \ell\neq \ell'}$ and similarly for $Q^\square_{{p}_N}$ and $Q^\square_{p}$. 
The invariance of the Poisson--Dirichlet cascade implies that $Q_{{p}_N}$ under $\E\la\cdot\ra_N$ has the same distribution as $Q_{{p}_N}$ under~$\E \la\cdot\ra_\fR$.
We therefore have that, for every Lipschitz function $g$, 
\begin{align*}
    &\Ll|\E \la g\Ll(R^\square_N\Rr) \ra_N - \E \la g\Ll(Q^\square_{p}\Rr) \ra_\fR \Rr| 
    \\
    &=  \Ll|\E \la g\Ll(R^\square_N\Rr) \ra_N - \E \la g\Ll(Q^\square_{{p}_N}\Rr) \ra_N +\E \la g\Ll(Q^\square_{{p}_N}\Rr) \ra_\fR - \E \la g\Ll(Q^\square_{p}\Rr) \ra_\fR \Rr|
    \\
    & \leq \|g\|_\mathrm{Lip}\Ll(\E \la \Ll|R^\square_N - Q^\square_{{p}_N}\Rr| \ra_N + \E \la \Ll|Q^\square_{{p}_N} - Q^\square_{p}\Rr| \ra_\fR\Rr). 
\end{align*}
Combining this with \eqref{e.RNQN.close} and \eqref{e.pNp.close} therefore completes the proof.
\end{proof}

\begin{proof}[Proof of Theorem~\ref{t.main3}]
The theorem follows from Proposition~\ref{p.cvg_overlap}.
\end{proof}

\subsection{Stability and high-temperature uniqueness of the critical point}
We now turn to a proof of Proposition~\ref{p.unique.crit.high.temp} on the uniqueness of the critical point at high temperature. 
\begin{proof}[Proof of Proposition~\ref{p.unique.crit.high.temp}]
The fact that $\mcl J_{t,q}$ has at least one critical point is a consequence of Theorem~\ref{t.main1} (and could also be proved directly).
For each $i \in \{1,2\}$, let $(q'_i,{p}_i)\in \mcl Q_2^2$ be a critical point of~$\mcl J_{t,q}$. 
Inserting~\eqref{e.crit.p} into~\eqref{e.crit.q'}, we get ${p}_i = \dr_q \psi(q+t\nabla \xi({p}_i))$.
By \eqref{e.continuity.der.psi}, we have
\begin{align*}
    \Ll|{p}_1-{p}_2\Rr|_{L^2}\leq 16 t \Ll|\nabla\xi({p}_1)-\nabla\xi({p}_2)\Rr|_{L^2}.
\end{align*}
By the local Lipschitzness of $\nabla\xi$, we have
\begin{align*}
    C := \sup_{a,a'\in \S^\D_+:\:|a|,|a'|\leq 1;\; a\neq a'}\frac{\Ll|\nabla\xi(a)-\nabla\xi(a')\Rr|}{|a-a'|}<+\infty.
\end{align*}
The property in \eqref{e.bound.der.psi} implies that $|{p}_i(s)|\leq 1$ for each $i\in\{1,2\}$ and $s\in[0,1)$. Therefore,
\begin{align*}
    \Ll|{p}_1-{p}_2\Rr|_{L^2}\leq 16  Ct \Ll|{p}_1-{p}_2\Rr|_{L^2}.
\end{align*}
Hence, if $ t < t_c := (16 C)^{-1}$, we must have ${p}_1 ={p}_2$. Since $q'_i$ is determined by~\eqref{e.crit.p}, the uniqueness follows.
\end{proof}

We close this section with a proof of Proposition~\ref{p.stability} on the stability of relevant critical points.
\begin{proof}[Proof of Proposition~\ref{p.stability}]
We take $(t_n, q_n) \in \R_+ \times \mcl Q_2$ and $(q'_n,p_n) \in \mcl Q_2^2$ as in the statement. The fact that $(q'_n,p_n)$ is a critical point of $\mcl J_{t_n, q_n}$ (see~\eqref{e.crit.p} and~\eqref{e.crit.q'}) means that
\begin{equation}
\label{e.stability.crit}
p_n = \partial_q \psi(q'_n) \quad \text{ and } \quad  q'_n = q_n + t_n \nabla \xi(p_n).
\end{equation}
By \eqref{e.bound.der.psi} from Corollary~\ref{c.psi_smooth}, we have that $p_n \in \mcl Q \cap L^\infty_{\le 1}$. By Lemma~\ref{l.compact.embed}, the sequence $(p_n)_{n \in \N}$ is therefore precompact in $L^2$. By the second relation in \eqref{e.stability.crit} and the convergence of $(t_n,q_n)_{n \in \N}$ in $\R_+ \times L^2$, we deduce that the sequence $(q'_n, p_n)_{n \in \N}$ is precompact. Moreover, if $(p_n)_{n \in \N}$ converges in $L^2$ along a subsequence, then $(q_n', p_n)_{n \in \N}$ converges in $(L^2)^2$ along the same subsequence. Recall also from~\eqref{e.def.Parisi.functional} that
\begin{equation} 
\label{e.stability.J.tn}
\mcl J_{t_n,q_n}(q'_n,p_n) = \psi(q'_n) + \la p_n , q_n-q'_n\ra_\cH + t_n \int_0^1 \xi(p_n).
\end{equation}
Letting $(q',p) \in \mcl Q_2^2$ denote a subsequential limit of $(q'_n,p_n)_{n \in \N}$, we have that $\psi(q'_n)$ converges to $\psi(q')$ along the subsequence, by the continuity of $\psi$ from Corollary~\ref{c.psi_smooth}. The scalar product in \eqref{e.stability.J.tn} converges to $\la p, q-q'\ra_\cH$ along the subsequence, by the Cauchy--Schwarz inequality. Finally, the integral $\int_0^1 \xi(p_n)$ converges to $\int_0^1 \xi(p)$ along the subsequence, because $\xi$ is continuous and $p_n \in L^\infty_{\le 1}$. So we conclude that as $n$ tends to infinity along the subsequence, we have that $\mcl J_{t_n, q_n} (q'_n, p_n)$ converges to $\mcl J_{t,q}(q',p)$. Using also the uniform Lipschitz continuity of $\bar F_N$ from Proposition~\ref{p.barF_N_Lip}, we obtain \eqref{e.stability.conclusion}. Using the continuity of $\dr_q \psi$ from Corollary~\ref{c.psi_smooth} and \eqref{e.stability.crit}, we also verify that $(q',p)$ is a critical point of $\mcl J_{t,q}$, and thus the proof is complete.
\end{proof}

\section{Further results in the convex case}
\label{s.convex}

In this section, we give a proof of Theorem~\ref{t.parisi.convex}. We also derive a number of alternative representations of the variational problem describing the limit free energy, and explain how to remove the compensating term $-N t \xi(\sigma \sigma^\intercal)/N$ appearing in \eqref{e.def.simple.F}. We also show that the limit free energy is Gateaux differentiable everywhere in $(0,\infty) \times (\mcl Q_\uparrow \cap L^\infty)$, and use this to show that it satisfies the Hamilton--Jacobi equation \eqref{e.satisfy.eqn.intro} at every point in this set. We collect all results into a final statement that also describes the convergence of the conditional $\sigma$-overlaps for every $(t,q) \in (0,\infty) \times (\mcl Q_\uparrow \cap L^\infty)$.

\subsection{Parisi formula when \texorpdfstring{$\xi$}{xi} is convex on \texorpdfstring{$\S^\D_+$}{S\^D\_+}}

When $\xi$ is convex over the whole space $\R^{D\times D}$, an interpolation argument allows one to establish a lower bound on any subsequential limit of $\bar F_N(t,q)$ \cite{gue03, barcon}. When $\D=1$, one can generalize this bound to the case when $\xi$ is only assumed to be convex over $\S^1_+=\R_+$ using the Talagrand positivity principle (see \cite[Theorem~3.4]{pan}). If $\D>1$ and $\xi$ only depends on the $D$ diagonal entries of its argument, then one may as well think of $\xi$ as being defined on $\R^D$, and one can simplify the setup and only consider paths that take values in $\R^D_+$ instead of $S^D_+$. In this context, one can still apply the scalar positivity principle to get the Parisi formula under the assumption that $\xi$ is convex over $\R^\D_+$, see \cite[Section~3.2]{bates2022free}.

For general vector spin glass models, to the best of our knowledge, existing results all require $\xi$ to be convex on $\R^{\D\times\D}$. Here we obtain the lower bound on the limit free energy by combining the results of \cite{mourrat2023free} and \cite{chen2022hamilton}. We derive the converse bound from the results of this paper, but as already mentioned, these results are much more precise than necessary in this case. 

Recall the definitions of $\sP_{t,{q}}$ in \eqref{e.mathscrP} and of $\mathcal{K}$ in \eqref{e.mathcal_K}.

\begin{proposition}[Parisi formula for the enriched model]\label{p.parisi}
If $\xi$ is convex on $\S^\D_+$, then for every $t \ge 0$ and $q \in \mcl Q_2$, we have
\begin{align}\label{e.parisi_hj}
    \lim_{N\to\infty} \bar F_N(t,{q}) = \sup_{{p}\in\C:\: \exists a \in \mathcal{K},\, p\leq a}\sP_{t,{q}}({p})= \sup_{{p}\in\C_\infty}\sP_{t,{q}}({p}).
\end{align}
\end{proposition}
\begin{proof}
We fix any $(t,q)$. 
Under the stated convexity assumption, the lower bound in \cite[Theorem~1.1]{mourrat2023free} expressed in terms of the solution to \eqref{e.hj} is proved in~\cite[Corollary~4.14]{chen2022hamilton} to admit a representation by the Hopf--Lax formula: 
\begin{align}\label{e.hopf-lax}
    \liminf_{N\to\infty} \bar F_N(t,q)  \geq \sup_{{q'}\in \C_\infty} \inf_{p'\in\C_\infty}\Ll\{\psi({q}+{q'})- \la q',p'\ra_\cH +t\int_0^1 \xi({p'}(s))\d s\Rr\}.
\end{align}
On the other hand, by Corollary~\ref{c.cavity}, we can find ${p} \in \C\cap L^\infty_{\leq 1}$ and $a\in\mathcal{K}$ satisfying $a\geq p$ such that
\begin{align}\label{e.limsup<}
    \limsup_{N\to\infty} \bar F_N(t,{q}) \leq \sP_{t,{q}}({p}) = \psi({q}+t\nabla\xi({p}))-t \int_0^1\theta({p}(s))\d s,
\end{align}
where we recall that $\theta$ is defined in~\eqref{e.theta=}.
To compare these two formulas, we need to show that
\begin{align}\label{e.theta=xi*()}
    \theta(b) = \xi^*(\nabla\xi(b)),\quad\forall b\in \S^\D_+,
\end{align}
with $\xi^*$ as in \eqref{e.def.xistar}. 
Fix any $b\in\S^\D_+$. It is clear from their definitions that $\theta(b) \leq  \xi^*(\nabla\xi(b))$. For any $b'\in\S^\D_+$, the convexity of $\xi$ on $S^D_+$ implies that for every $\lambda \in [0,1]$, 
\begin{equation*}  \lambda \xi(b') + (1-\lambda)\xi(b) \ge \xi(\lambda b'+ (1-\lambda)b).
\end{equation*}
Rearranging terms and sending $\lambda$ to zero, we obtain that 
\begin{equation*}  \xi(b')-\xi(b)\geq (b'-b)\cdot \nabla\xi(b),
\end{equation*}
and thus
\begin{equation*}  b\cdot\nabla\xi(b)-\xi(b)\geq b'\cdot\nabla\xi(b)-\xi(b').
\end{equation*}
Taking the supremum over $b'\in\S^\D_+$, we get $\theta(b)\geq \xi^*(\nabla\xi(b))$, and thus complete the argument for~\eqref{e.theta=xi*()}.

Using~\eqref{e.theta=xi*()}, we have
\begin{multline*}
    \int_0^1\theta({p}(s))\d s  = \int_0^1 \xi^*(\nabla\xi({p}(s))) \d s \geq \sup_{p'\in\C_\infty}\Ll\{\la \nabla\xi(p), p'\ra_\cH-\int_0^1\xi(p'(s))\d s \Rr\}  
    \\
    \geq \la \nabla\xi(p), p\ra_\cH-\int_0^1\xi(p(s))\d s = \int_0^1\theta({p}(s))\d s ,
\end{multline*}
which implies
\begin{align}\label{e.int_theta=sup}
    \int_0^1\theta({p}(s))\d s = \sup_{p'\in\C_\infty}\Ll\{\la \nabla\xi(p), p'\ra_\cH-\int_0^1\xi(p'(s))\d s \Rr\}.
\end{align}
Applying this to \eqref{e.hopf-lax}, we get
\begin{align*}
    \liminf_{N\to\infty} \bar F_N(t,q)  & \geq  \inf_{p'\in\C_\infty}\Ll\{\psi({q}+t\nabla\xi(p))- t\la \nabla\xi(p),p'\ra_\cH +t\int_0^1 \xi({p'}(s))\d s\Rr\}
    \\
    &\geq \psi({q}+t\nabla\xi(p))- t\int_0^1 \theta({p}(s))\d s,
\end{align*}
which along with~\eqref{e.limsup<} yields the first equality in~\eqref{e.parisi_hj}. The second equality can be deduced similarly.
\end{proof}

\begin{proof}[Proof of Theorem~\ref{t.parisi.convex}]
Fix any $(t,q)\in\R_+\times \C_2$.
Applying~\eqref{e.int_theta=sup} to \eqref{e.limsup<} and comparing it with~\eqref{e.hopf-lax}, we get
\begin{align*}
    \lim_{N\to\infty} \bar F_N(t,q)  &= \sup_{{q'}\in \C_\infty} \inf_{p'\in\C_\infty}\Ll\{\psi({q}+{q'})- \la q',p'\ra_\cH +t\int_0^1 \xi({p'}(s))\d s\Rr\}
    \\
    & = \sup_{q'\in q+\C_\infty} \inf_{p\in\C_\infty} \mcl J_{t,q}(q',p),
\end{align*}
where the last equality follows from the definition of $\mcl J_{t,{q}}(q',{p})$ in~\eqref{e.def.Parisi.functional}.
Hence, we have verified the first identity in~\eqref{e.parisi.convex}.
It is proved in \cite[Theorem~4.6]{chen2022hamilton} that the unique viscosity solution~$f$ to~\eqref{e.hj}, defined in \cite[Definition~4.2]{chen2022hamilton}, admits the Hopf--Lax representation under the assumption that $\xi$ is convex on $\S^\D_+$. The Hopf--Lax formula for $f$ evaluated at $(t,q)$ has the same expression as the one on the first line of the above display. Since this identity holds at every $(t,q)$, we conclude that $\lim_{N\to\infty}\bar F_N =f$ pointwise everywhere on $\R_+\times \C_2$.
\end{proof}

In the above proof, to show that $\Ll(\bar F_N\Rr)_{N\in\N}$ converges to the solution $f$, we directly match the variational formula for $\lim_{N\to\infty}\bar F_N$ with the Hopf--Lax representation of $f$. We mention another approach via the comparison principle for viscosity solutions. Let $p$ be as given in~\eqref{e.limsup<} and consider the function $g: (t',q')\mapsto \sP_{t',q'}(p)$. Using the convexity of $\xi$, one can verify that $g$ is a subsolution of~\eqref{e.hj}. The comparison principle (e.g.\ \cite[Proposition~3.8 and Remark~4.7]{chen2022hamilton}) implies that $g\leq f$ everywhere. This along with~\eqref{e.limsup<} yields $\limsup \bar F_N(t,q)\leq g(t,q)\leq f(t,q)$. Since $(t,q)$ was arbitrary, we have $\limsup \bar F_N\leq f$ everywhere on $\R_+\times \C_2$. The matching lower bound is provided by \cite[Theorem~1.1]{mourrat2023free}. This argument can be seen in~\cite{chen2022pde} for $D=1$.

In the corollary below, we verify the version of the Hopf--Lax formula appearing in~\eqref{e.parisi.convex.var}.

\begin{corollary}[Alternative form of Hopf--Lax formula]If $\xi$ is convex on $\S^\D_+$, then for every $t \ge 0$ and $q \in \mcl Q_2$, we have
\begin{align}
\label{e.parisi.convex2}
    \lim_{N\to\infty} \bar F_N(t,q) = \sup_{q' \in \mcl Q_\infty} \Ll\{ \psi(q + q') - t \int_0^1 \xi^* \Ll( \frac{q'}{t} \Rr)  \Rr\} .
\end{align}
\end{corollary}
\begin{proof}
We fix any $(t,q)$ and denote the variational formula on the right-hand side of~\eqref{e.parisi.convex2} by $\rhs$.
Using~\eqref{e.theta=xi*()} in~\eqref{e.limsup<}, we get
\begin{equation*}  \limsup_{N\to\infty}\bar F_N(t,q)\leq\rhs.
\end{equation*}
To show the converse bound, we recall the definition of $\xi^*$ in~\eqref{e.def.xistar} to see that, for every $q' \in \mcl Q_\infty$,
\begin{align*}
     \sup_{p'\in\C_\infty}\Ll\{\frac{1}{t}\la q',p'\ra_\cH -\int_0^1\xi(p'(s))\d s\Rr\} \leq \int_0^1 \xi^* \Ll(\frac{q'(s)}{t}\Rr)\d s.
\end{align*}
Using this in~\eqref{e.hopf-lax}, we get $\liminf_{N\to\infty}\bar F_N(t,q)\geq \rhs$. This gives the announced identity at $(t,q)$.
\end{proof}

We can extract from Proposition~\ref{p.parisi} a result that takes a form that is more common in the literature. For every $\pi\in\C_\infty$ and $x\in \S^\D$, we set
\begin{align} \label{e.sP(pi,x)}
\begin{split}
    \sP(\pi,x) := \E\log\iint\exp\Ll(w^{\nabla\xi\circ\pi}(\alpha)\cdot\tau-\frac{1}{2}\nabla\xi\circ\pi(1)\cdot\tau\tau^\intercal+x\cdot \tau\tau^\intercal \Rr)\d P_1(\sigma) \d \fR(\alpha)
     \\+\frac{1}{2}\int_0^1\theta\Ll(\pi(s)\Rr)\d s.
\end{split}
\end{align}
This resembles the classic Parisi functional for the Sherrington--Kirkpatrick model.

\begin{corollary}[Parisi formula for free energy with correction]\label{c.parisi}
If $\xi$ is convex on $\S^\D_+$, then
\begin{align*}
    \lim_{N\to\infty} \frac{1}{N}\E \log \int \exp\Ll(H_N(\sigma) - \frac{N}{2}\xi\Ll(\frac{\sigma\sigma^\intercal}{N}\Rr)\Rr)\d P_N(\sigma) 
     = \inf_{\pi\in\C_\infty} \sP(\pi,0).
\end{align*}
\end{corollary}

\begin{proof}
Recall the expression of $\bar F_N(t,q)$ from~\eqref{e.def.barFN.continuous} and the random Hamiltonian $H_N(\sigma)$ with covariance~\eqref{e.def.xi}.
Setting $t=\frac{1}{2}$ and $q=0$, we get
\begin{align*}
    -\bar F_N\Ll(\frac{1}{2},0\Rr)=\frac{1}{N}\E \log \int \exp\Ll(H_N(\sigma) - \frac{N}{2}\xi\Ll(\frac{\sigma\sigma^\intercal}{N}\Rr)\Rr)\d P_N(\sigma).
\end{align*}
Recalling the definition of $\sP_{t,q}$ in~\eqref{e.mathscrP}, we have $-\sP_{\frac{1}{2},0}(p) = \sP(p,0)$ for every $p\in\C_\infty$, where we also used $\sqrt{2}w^{\frac{1}{2}q'}\stackrel{\d}{=}w^{q'}$ for $q'\in\C_\infty$, which is clear from~\eqref{e.gaussian_cascade}.
Now, the announced identity follows from~\eqref{e.parisi_hj} in Proposition~\ref{p.parisi} evaluated at $t=\frac{1}{2}$ and $q=0$.
\end{proof}

We can further remove the correction term $- \frac{N}{2}\xi\Ll(\frac{\sigma\sigma^\intercal}{N}\Rr)$ and get the following.

\begin{proposition}[Parisi formula]\label{p.parisi_std}
If $\xi$ is convex on $\S^\D_+$, then
\begin{align*}
    \lim_{N\to\infty} \frac{1}{N}\E \log \int \exp\Ll(H_N(\sigma)\Rr)\d P_N(\sigma) 
    & = \sup_{y\in\S^\D_+}\inf_{\pi\in\C_\infty}\Ll\{\sP(\pi,y)-\frac{1}{2}\xi^*(2y)\Rr\}
    \\
    & = \sup_{z\in\S^\D_+}\inf_{y\in\S^\D_+,\ \pi\in\C_\infty}\Ll\{\sP(\pi,y)-y\cdot z + \frac{1}{2}\xi(z)\Rr\}.
\end{align*}
\end{proposition}

\begin{proof}
The method is the same as in~\cite[Section~5]{mourrat2020extending} and~\cite[Section~5]{chen2023self}. For $N\in\N$ and $(t,x) \in \R_+\times \S^\D$, we consider
\begin{align*}
    \cF_N(t,x) := \frac{1}{N}\E\log\int\exp\Ll(H_N(\sigma)- \frac{N}{2}\xi\Ll(\frac{\sigma\sigma^\intercal}{N}\Rr)+ tN\xi\Ll(\frac{\sigma\sigma^\intercal}{N}\Rr)+x\cdot \sigma\sigma^\intercal\Rr)\d P_N(\sigma).
\end{align*}
It is shown in \cite[Proposition~5.2]{chen2023self} that the convexity of $\xi$ over $\S^\D_+$ implies that the limit of $(\cF_N)_{N\in\N}$ in the local uniform topology is a viscosity solution $f$ of $\partial_t f - \xi(\nabla_x f) =0$ on $(0,\infty)\times \S^\D$. The notion of viscosity solutions is defined in \cite[Section~5]{chen2023self}.
Recall that $\S^\D_{++}$ is the set of $\D$-by-$\D$ positive definite matrices, which is the interior of $\S^\D_+$.
Since $\S^\D_{++}$ is an open subset of $\S^\D$, we immediately have that $f$ is a viscosity solution of
\begin{align}\label{e.cF_hj}
    \partial_t f - \xi(\nabla_x f) =0\qquad \text{on}\quad (0,\infty)\times \S^\D_{++}.
\end{align}
We want to use variational formulas for the solution of~\eqref{e.cF_hj} to conclude. For this, we need to first identify the initial condition of $f$ and verify more properties.

Applying Corollary~\ref{c.parisi} with $e^{x\cdot\tau\tau^\intercal} \d P_1(\tau)$ substituted for $P_1$ therein, we have, for every $x\in\S^\D$,
\begin{align}\label{e.f(0,x)=}
    f(0,x) = \lim_{N\to\infty} \cF_N(0,x)  = \inf_{\pi\in\Pi} \sP(\pi,x).
\end{align}
It is verified in \cite[Lemma~5.1]{chen2023self} that $\cF_N$ is Lipschitz uniformly in $N$ and convex, and satisfies that $\cF_N(t,x) \leq \cF_N(t,x')$ whenever $x\geq x'$. Passing to the limit, we have that $f$ is Lipschitz and convex and that $f(t,\cdot)$ is increasing for every $t\geq 0$. These properties along with the well-posedness \cite[Proposition~5.3]{chen2023self} of \eqref{e.cF_hj} implies that $f$ is the viscosity solution of~\eqref{e.cF_hj} with the above initial condition, which is unique in the class of Lipschitz functions that are increasing for each fixed $t$. 

Since $\xi$ is convex on $\S^\D_+$ and $f(0,\cdot)$ is also convex, we can represent $f$ by the Hopf--Lax formula and by the Hopf formula described in \cite[Proposition~5.3]{chen2023self}: for every $t>0$ and $x\in\S^\D_+$,
\begin{align}\label{e.f(t,x)=}
\begin{split}
    f(t,x) & =\sup_{y\in\S^\D_+}\Ll\{f(0,x+y)-t\xi^*\Ll(\frac{y}{t}\Rr)\Rr\}
    \\
    & = \sup_{z\in\S^\D_+}\inf_{y\in\S^\D_+}\Ll\{f(0,y)+(x-y)\cdot z+ t\xi(z)\Rr\}.
\end{split}
\end{align}
Notice that, for every $N\in\N$,
\begin{align*}
    \cF_N\Ll(\frac{1}{2},0\Rr) = \frac{1}{N}\E\log\int \exp(H_N(\sigma))\d P_N(\sigma).
\end{align*}
Hence, setting $t=\frac{1}{2}$ and $x=0$ in~\eqref{e.f(t,x)=}, inserting~\eqref{e.f(0,x)=} to~\eqref{e.f(t,x)=}, and using the convergence of $\cF_N$ to $f$, we can obtain the announced identities.
\end{proof}

\begin{remark}
In the spin glass literature, sometimes it is preferred to consider left-continuous paths in $\Pi$ defined in~\eqref{e.Pi=} instead of the right-continuous ones in $\C_\infty$. 
We explained in Remark~\ref{r.left-cts} that, for $\pi\in\Pi$, the Gaussian process $w^\pi$ exists and satisfies all the properties in Section~\ref{s.cascades}. Hence, $\sP(\pi,x)$ is defined in the same way as in~\eqref{e.sP(pi,x)} for $\pi\in\Pi$ and $x\in\S^\D$. By the $L^1$-Lipschitzness of $\psi$ in Corollary~\ref{c.psi_smooth}, the value of $\sP(\pi,x)$ does not change if we replace $\pi$ by its left-continuous version.
Therefore, $\inf_{\pi\in \C_\infty}$ in Corollary~\ref{c.parisi} and Proposition~\ref{p.parisi_std} can be replaced by $\inf_{\pi \in\Pi}$.
Due to the first equality in Proposition~\ref{p.parisi}, we can further refine $\inf_{\pi\in\Pi}$ therein to $\inf_{\pi\in\Pi:\:\exists a \in \mathcal{K},\, \pi\leq a}$ for $\mathcal{K}$ defined in~\eqref{e.mathcal_K}.
\end{remark}

\subsection{Differentiability of the Parisi formula}

We show that the limit free energy is differentiable. We use a similar argument as that for the so-called generic models, which is presented, for instance, in \cite[Section~3.7]{pan}.

\begin{proposition}\label{p.diff_parisi}
Suppose that $\xi$ is convex on $\S^\D_+$, and let $f$ be the limit of $\Ll(\bar F_N\Rr)_{N\in\N}$ as identified in Proposition~\ref{p.parisi}. 
\begin{itemize}
    \item For each $t\in \R_+$, the function $f(t,\cdot)$ is Gateaux differentiable everywhere on $\C_\uparrow\cap L^\infty$.
    \item The function $f$ is Gateaux differentiable everywhere on $(0,\infty)\times (\C_\uparrow\cap L^\infty)$.
\end{itemize}
\end{proposition}

\begin{proof}
We show the first property; the second one can be verified in an analogous way.
Fix any $(t,{q}) \in \R_+\times (\C_\uparrow\cap L^\infty)$. 
We write $f({q'}) :=f(t,{q'})$ and $\sP_{{q'}}(\pi) := \sP_{t,{q'}}(\pi)$. 
By Proposition~\ref{p.parisi}, we can write, for every $q' \in \mcl Q_2$, 
\begin{align}\label{e.limF_N=sP}
    \lim_{N\to\infty} \bar F_N(t,q' )= f(q') = \sup_{\pi\in\C\cap L^\infty_{\leq 1}} \sP_{q'}(\pi).
\end{align}
This identity at $q$ together with~\eqref{e.limsup<} implies that there is $\pi\in \C\cap L^\infty_{\leq 1}$ such that
\begin{align}\label{e.sP(pi)=f(q)}
    \sP_q(\pi) =  f(q).
\end{align}

\medskip

\noindent \emph{Step 1.}
We show that $f$ is differentiable at $q$ along nice directions.
Let $\kappa\in C^{\infty}([0,1]; S^D)$ be such that $\kappa(0) = 0$. Due to $q\in \C_\uparrow$, by the same argument in Step~1 of the proof of Proposition~\ref{p.conv.der}, there are $c,\delta>0$ such that ${q} + r \kappa\in \C_{\uparrow,c}\cap L^\infty$ for all $r\in [-\delta,\delta]$ (we recall that~$\C_{\uparrow,c}$ is defined in \eqref{e.def.C_c}).
By Proposition~\ref{p.reg.lim}, we have that the function $r\mapsto  f(q+r\kappa)$ is semi-concave on $[-\delta,\delta]$. We want to show that this function is differentiable at $r=0$. For this, we need to introduce the notion of superdifferential (e.g.\ \cite[Definition~3.1.1]{cannarsa2004semiconcave}). For a function $g:I\to\R$ defined on an open interval $I\subset\R$, the \textit{superdifferential} of $g$ at $r\in I$ is the set of real numbers $a$ satisfying
\begin{align*}
    \limsup_{s\to r}\frac{g(s)-g(r)-a(s-r)}{|s-r|}\leq 0.
\end{align*}
It is well-known (e.g.\ \cite[Proposition~3.3.4~(c-d)]{cannarsa2004semiconcave}, or \cite[Theorem~2.13]{HJbook} for a variant with convex functions) that for a semi-concave function $g : I \to \R$ and $r\in I$, the superdifferential at $r$ is not empty, and if it contains exactly one point $a \in \R$, then $g$ is differentiable at $r$ with $\frac{\d}{\d r}g(r)=a$.
So we let $a \in \R$ be any element of the superdifferential of the mapping $r\mapsto  f(q+r\kappa)$ at $r=0$, and proceed to identify it uniquely. 
In the following, we denote by $C$ a constant that depends only on $(t,{q})$, $c$, and $\kappa$, and may vary from instance to instance.
By Proposition~\ref{p.reg.lim}, we have
\begin{align}\label{e.sP_semi_concave}
    (1-\lambda)  f(q) + \lambda  f(q+r\kappa) -  f({q}+ \lambda r \kappa) \leq C\lambda(1-\lambda)r^2
\end{align}
uniformly in $\lambda\in[0,1]$ and $r$. Dividing both sides by $\lambda$, sending $\lambda$ to zero, and using the definition of the superdifferential, we get
\begin{align*}
     f(q+r\kappa) -  f(q) -a r\leq Cr^2.
\end{align*}
Dividing both sides by $r>0$ and using~\eqref{e.limF_N=sP} and~\eqref{e.sP(pi)=f(q)}, we get
\begin{align*}
    a\geq \frac{ f(q+r\kappa)- f(q)}{r}  -Cr\geq \frac{\sP_{q+r\kappa}(\pi)-\sP_q(\pi)}{r}  -Cr.
\end{align*}
Repeating it for $-r<0$, we have
\begin{align*}
    a\leq \frac{ f(q)- f(q-r\kappa)}{r} +Cr\leq \frac{\sP_q(\pi)-\sP_{q-r\kappa}(\pi)}{r}  +Cr.
\end{align*}
Recall the definition of $\sP_{t,q}$ in~\eqref{e.mathscrP}.
By the smoothness of $\psi$ in 
Corollary~\ref{c.psi_smooth}, and the $L^\infty$ bound on $\pi$, there is $C$ such that
\begin{align*}\Ll|\frac{\sP_{q+r\kappa}(\pi)-\sP_q(\pi)}{r}-\frac{\d}{\d r}\sP_{q+r\kappa}(\pi)\Big|_{r=0}  \Rr|  \leq Cr,
\end{align*}
uniformly in $r\in[-\delta,\delta]$.
Combining the above three displays, we obtain
\begin{align*}
    \frac{\d}{\d r}\sP_{q+r\kappa}(\pi)\Big|_{r=0} -Cr \leq a \leq \frac{\d}{\d r}\sP_{q+r\kappa}(\pi)\Big|_{r=0} +Cr
\end{align*}
Now, sending $r$ to zero, this identifies $a$ uniquely, and thus $r\mapsto  f(q+r\kappa)$ is differentiable at $r=0$.

\medskip

\noindent \emph{Step 2.}
We want to identify the derivative of $f$ at $q$.
We use an argument similar to that in the proof of Proposition~\ref{p.1st_fpe}.
Let $(N^\pm_k)_{k\in\N}$, $(x^\pm_k)_{k\in\N}$, and ${p}_\pm$ be given by Corollary~\ref{c.cavity}.
Recall the definitions of $\tilde F^x_N$ in \eqref{e.tildeF_N=} and of $\la\cdot\ra^\cav_{N,x}$ in \eqref{e.<>^cav_N,x}. We set $\tilde F_k^\pm := \tilde F^{x^\pm_k}_{N^\pm_k}$.
As in the proof of Proposition~\ref{p.1st_fpe}, we can compute, for every $\kappa \in C([0,1]; S^D)$,
\begin{align*}
    \frac{\d}{\d r}\tilde F^\pm_k(t,{q}+r\kappa)\Big|_{r =0} = \E \la  \kappa\Ll(\alpha\wedge\alpha'\Rr) \cdot \frac{\sigma\sigma'^\intercal}{N^\pm_k}\ra^\cav_{N^\pm_k,x^\pm_k}.
\end{align*}
By Lemma~\ref{l.limF^x-F}, we also have that $\tilde  F_k^\pm$ converges pointwise to $f$. Similarly to Step~2 of the proof of Proposition~\ref{p.conv.der}, the semi-concavity~\eqref{e.sP_semi_concave} implies the convergence of $\frac{\d}{\d r}\tilde  F^\pm_k(t,{q}+r\kappa)\big|_{r =0}$ to $\frac{\d}{\d r}  f(q+r\kappa)\big|_{r=0}$.
This convergence along with the above display and Corollary~\ref{c.cavity}~\eqref{i.c.cavity_I_1} implies, for every $\kappa \in C([0,1]; S^D)$,
\begin{align}\label{e.dsP=<>}
    \frac{\d}{\d r}  f(q+r\kappa)\Big|_{r=0} = \E \la \kappa\Ll(R^{1,2}_\alpha\Rr) \cdot {p}_\pm\Ll(R^{1,2}_\alpha\Rr)\ra_{\fR}  = \la \kappa, {p}_\pm\ra_\cH.
\end{align}
Since $C([0,1]; S^D)$ is dense in $\cH$, we must have ${p}_\pm={p}$ for some ${p}\in \cH$.

\medskip

\noindent \emph{Step 3.} We are now ready to show that $ f$ is Gateaux differentiable at $q$.
By Proposition~\ref{p.reg.lim} and \eqref{e.limF_N=sP}, we have that $ f:\C_2\to\R$ is Lipschitz continuous. Therefore, by the first step of the proof of Proposition~\ref{p.Gateaux_dff_dense}, the function $ f$ admits a Lipschitz extension $\bar  f: \cH\to \R$. We show that $\bar  f$ is Gateaux differentiable at ${q}$. 
Fix any ${q'}\in \C_2$ and let $(\kappa_n)_{n\in\N}$ be a sequence in $C^{\infty}([0,1]; S^D)$ with $\kappa_n(0) = 0$ that converges to ${q'}$ in $\cH$.
We have
\begin{align*}
   \Ll|\frac{\bar  f({q}+r{q'})-\bar f(q)}{r}-\la {p}, {q'}\ra_\cH\Rr|\leq  \Ll|\frac{\bar  f({q}+r\kappa_n)-\bar f(q)}{r}-\la {p}, \kappa_n\ra_\cH\Rr| + C_0\Ll|\kappa_n -{q'}\Rr|_\cH,
\end{align*}
where $C_0 := \Ll\|\bar f\Rr\|_\mathrm{Lip}+|{p}|_\cH$.
By~\eqref{e.dsP=<>}, sending $r$ to zero and then $n$ to infinity in the above display, we get $\frac{\d}{\d r} \bar f({q}+r{q'})\Big|_{r=0}= \la {p},{q'}\ra_\cH$. Hence, we can conclude that $\bar f$ is Gateaux differentiable at ${q}$ and its derivative is equal to ${p}$.

Since admissible directions at ${q}\in\C_2$ for $ f$ includes $\C_2$ and $\C_2$ spans $\cH$, we can deduce from Definition~\ref{d.gateaux} that $ f$ is Gateaux differentiable at ${q}$ with derivative ${p}$.
This verifies the first announced property. 
\end{proof}

\subsection{Summary}
In the convex case, we summarize our results below.

\begin{corollary}\label{c.convex}
Suppose that $\xi$ is convex on $\S^\D_+$.
Then, the sequence $\Ll(\bar F_N\Rr)_{N\in\N}$ converges pointwise to some limit $f$ on $\R_+\times \C_2$.
At every $(t,q) \in (0,\infty)\times (\C_\uparrow\cap L^\infty)$, the function $f$ is Gateaux differentiable (jointly in its two variables) and satisfies
\begin{align}\label{e.pde_convex}
    \partial_t f(t,q) - \int_0^1 \xi\Ll(\dr_q f(t,q)\Rr) =0.
\end{align}
For every $t\in[0,\infty)$, $f(t,\cdot)$ is Gateaux differentiable at every $q\in \C_\uparrow\cap L^\infty$ and the following holds for $p=\dr_q f(t,q)$ and $p_N= \dr_q \bar F_N(t,q)$:
\begin{enumerate}
    \item \label{i.c.convex_1} $p,\, p_N\in \C\cap L^\infty_{\leq 1}$ for every $N\in\N$ and $(p_N)_{N\in\N}$ converges to $p$ in $L^r$ for every $r\in[1,\infty)$ as $N$ tends to infinity;
    \item \label{i.c.convex_2} $f(t,q) = \sP_{t,q}(p)$ and $p=\dr_q \psi (q+t\nabla \xi(p))$;
    \item \label{i.c.convex_3} $p_N(\alpha\wedge\alpha') = \E \la \frac{\sigma\sigma'^\intercal}{N}\, |\, \alpha\wedge\alpha' \ra_N$ almost surely under $\E \la\cdot\ra_N$ for every $N$, and the overlap array $\Ll(p_N\Ll(\alpha^\ell\wedge\alpha^{\ell'}\Rr)\Rr)_{\ell,\ell'\in\N:\:\ell\neq \ell'}$ under $\E \la\cdot\ra_N$ converges in law to $\Ll(p\Ll(\alpha^\ell\wedge\alpha^{\ell'}\Rr)\Rr)_{\ell,\ell'\in\N:\:\ell\neq \ell'}$ under $\E \la\cdot\ra_\fR$ as $N$ tends to infinity.
\end{enumerate}
\end{corollary}
In part~\eqref{i.c.convex_2}, $\sP_{t,q}$ is given in~\eqref{e.mathscrP}.
In part~\eqref{i.c.convex_3}, the conditional expectation is taken with respect to the measure $\E \la\cdot\ra_N$ for $\la\cdot\ra_N$ given in~\eqref{e.<>_N}.

\begin{proof}[Proof of Corollary~\ref{c.convex}]
The existence of $f$ is given by Proposition~\ref{p.parisi}. The differentiability of $f$ follows from Proposition~\ref{p.diff_parisi}. Proposition~\ref{p.2nd_fpe} yields~\eqref{e.pde_convex}. Part~\eqref{i.c.convex_1} is a consequence of \eqref{e.bounds.der.FN} in Proposition~\ref{p.F_N_smooth} and Proposition~\ref{p.conv.der}.
Part~\eqref{i.c.convex_2} follows from Propositions~\ref{p.1st_fpe} and~\ref{p.2nd_fpe}. Lastly, part~\eqref{i.c.convex_3} is a consequence of Proposition~\ref{p.cvg_overlap_og} and~\eqref{e.cond_overlap=p_N()} in its proof.
\end{proof}

Corollary~\ref{c.convex}~\eqref{i.c.convex_3} only gives the convergence of the conditioned overlap. To achieve the unconditioned convergence, one can add a small perturbation as in Theorem~\ref{t.main3}. The next result shows that the strict convexity of $\xi$ is sufficient without perturbation.

\begin{proposition}
Suppose that the function $\xi$ is strictly convex over $\R^{\D\times\D}$, let $t \in (0,\infty)$ and $q\in \C_\uparrow\cap L^\infty$, and let $p$ be as in Corollary~\ref{c.convex}. 
Then, the off-diagonal overlap array $\Ll(\sigma^l(\sigma^{l'})^\intercal/N\Rr)_{l,l'\in\N:\:l\neq l'}$ under $\E \la\cdot\ra_N$ converges in law to $\Ll(p\Ll(\alpha^l\wedge\alpha^{l'}\Rr)\Rr)_{l,l'\in\N:\:l\neq l'}$ under $\E\la\cdot\ra_\mathfrak{R}$ as $N$ tends to infinity.
\end{proposition}
\begin{proof}
We define, for every $r \ge 0$,
\begin{equation*}  \eta(r) := \inf \Ll\{ \xi(y) -  \xi(x) - (y-x) \cdot \nabla \xi(x) \ \mid \ |x| \le 1, |x-y| = r \Rr\} .
\end{equation*}
The definition of $\eta(r)$ consists in the minimization of a continuous function over a compact set, hence the infimum is achieved. By the assumption of strict convexity, we deduce that $\eta(r) > 0$ for every $r > 0$. By definition, we have that
\begin{align*}
    \xi(y) - \xi(x) -(y-x)\cdot\nabla\xi(x)\geq \eta(|y-x|),
\end{align*}
for every $x,y \in \R^{\D\times\D}$ with $|x| \le 1$. Then, for any random variable $X$ taking values in the unit ball of $\R^{\D\times\D}$, substituting $X,\,\E X$ for $y,x$ and taking the expectation, we get
\begin{align*}
    \E \xi(X)-\xi\Ll(\E X\Rr)\geq \E\, \eta(|X-\E X|).
\end{align*}
As in Corollary~\ref{c.convex}, we set $p_N = \partial_q\bar F_N(t,q)$ so that by part~\eqref{i.c.convex_3} there, we have that $\E \la\frac{\sigma\sigma'^\intercal}{N}\big| \alpha\wedge\alpha'  \ra_N = p_N(\alpha\wedge\alpha')$. Substituting $\frac{\sigma\sigma'^\intercal}{N}$ for $X$ and $\E\la\,\cdot\, \big|\alpha\wedge\alpha'\ra_N$ for $\E$ in the above display and then taking $\E\la\cdot\ra_N$, we get
\begin{align*}
    \E \la \xi\Ll(\frac{\sigma\sigma'^\intercal}{N}\Rr)\ra_N - \E \la\xi \Ll(p_N\Ll(\alpha\wedge\alpha'\Rr)\Rr)\ra_N \geq \E \la\eta\Ll(\Ll|\frac{\sigma\sigma'^\intercal}{N} -  p_N\Ll(\alpha\wedge\alpha'\Rr)\Rr|\Rr) \ra_N.
\end{align*}
Using~\eqref{e.def.der.FN} and the invariance of cascades (Proposition~\ref{p.invar_cts}), we can rewrite the left-hand side in the above display as
\begin{align*}
    \partial_t \bar F_N(t,q) - \int_0^1\xi\Ll(\partial_q\bar F_N(t,q) \Rr).
\end{align*}
By Proposition~\ref{p.conv.der}, we have
\begin{align*}
    \lim_{N\to\infty} \partial_{t} \bar F_N\Ll(t,{q}\Rr) = \partial_{t} f\Ll(t,{q}\Rr)\quad \text{ and } \quad \lim_{N\to\infty} \Ll|\dr_q \bar F_N \Ll(t,{q}\Rr) - \dr_q f\Ll(t,{q}\Rr)\Rr|_\cH = 0.
\end{align*}
Combining the above three displays with~\eqref{e.pde_convex}, we obtain 
\begin{align*}
    \lim_{N\to\infty} \E \la \eta \Ll(\Ll|\frac{\sigma\sigma'^\intercal}{N} - p_N\Ll(\alpha\wedge\alpha'\Rr)\Rr|\Rr)\ra_N =0.
\end{align*}
Since the random variables appearing in the argument of $\eta$ are bounded, and recalling that $\eta$ is strictly positive on $(0,\infty)$, we see that the display above is equivalent to the statement that $\Ll|\frac{\sigma\sigma'^\intercal}{N} - p_N\Ll(\alpha\wedge\alpha'\Rr)\Rr|$ converges to zero in $\E\la \cdot \ra_N$-probability; and this in turn is equivalent to~\eqref{e.RNQN.close}. We can now follow the same steps after~\eqref{e.RNQN.close} to conclude the proof.
\end{proof}

\medskip

\noindent \textbf{Acknowledgements.} HBC has received funding from the European Research Council (ERC) under the European Union’s Horizon 2020 research and innovation programme (grant agreement No.\ 757296).

\small
\bibliographystyle{plain}
\newcommand{\noop}[1]{} \def\cprime{$'$}

\end{document}